\documentclass[12pt]{amsart}
\usepackage{lmodern}
\usepackage{ifthen}
\usepackage{amsfonts}
\usepackage{amsxtra}
\usepackage{amssymb}
\usepackage{array}
\usepackage[margin=1.1in]{geometry}
\usepackage{xcolor}
\definecolor{cite}{rgb}{0.50,0.00,1.00}
\definecolor{url}{rgb}{0.00,0.50,0.75}
\definecolor{link}{rgb}{0.00,0.00,0.50}
\usepackage[colorlinks,linkcolor=link,urlcolor=url,citecolor=cite,pagebackref,breaklinks]{hyperref}
\usepackage{bbm}
\usepackage{mathtools}
\usepackage{mathrsfs}
\usepackage{appendix}
\usepackage[all]{xy}
\usepackage{colonequals}
\usepackage[lite,abbrev,msc-links,alphabetic]{amsrefs}

\makeindex

\numberwithin{equation}{section}

\theoremstyle{plain}
\newtheorem{proposition}{Proposition}[section]

\newtheorem{corollary}[proposition]{Corollary}
\newtheorem{lem}[proposition]{Lemma}
\newtheorem{theorem}[proposition]{Theorem}

\theoremstyle{remark}
\newtheorem{definition}[proposition]{Definition}
\newtheorem{notation}[proposition]{Notation}

\newtheorem{remark}[proposition]{Remark}

\renewcommand{\b}[1]{\mathbf{#1}}
\renewcommand{\c}[1]{\mathcal{#1}}
\renewcommand{\d}[1]{\mathbb{#1}}
\newcommand{\f}[1]{\mathfrak{#1}}
\renewcommand{\r}[1]{\mathrm{#1}}
\newcommand{\s}[1]{\mathscr{#1}}
\renewcommand{\sf}[1]{\mathsf{#1}}
\renewcommand{\(}{\left(}
\renewcommand{\)}{\right)}
\newcommand{\res}{\mathbin{|}}
\newcommand{\ol}{\overline}
\newcommand{\ul}{\underline}
\newcommand{\Sec}{\S}

\newcommand{\ang}[1]{\langle{#1}\rangle}

\newcommand{\bA}{\b A}

\newcommand{\bD}{\b D}

\newcommand{\bG}{\b G}
\newcommand{\bH}{\b H}

\newcommand{\bM}{\b M}

\newcommand{\bP}{\b P}

\newcommand{\bc}{\b c}

\newcommand{\bw}{\b w}

\newcommand{\cA}{\c A}

\newcommand{\cF}{\c F}
\newcommand{\cG}{\c G}
\newcommand{\cH}{\c H}

\newcommand{\cL}{\c L}

\newcommand{\cO}{\c O}

\newcommand{\cT}{\c T}

\newcommand{\cX}{\c X}
\newcommand{\cY}{\c Y}
\newcommand{\cZ}{\c Z}

\newcommand{\dA}{\d A}
\newcommand{\dB}{\d B}
\newcommand{\dC}{\d C}

\newcommand{\dF}{\d F}

\newcommand{\dN}{\d N}

\newcommand{\dQ}{\d Q}
\newcommand{\dR}{\d R}

\newcommand{\dZ}{\d Z}

\newcommand{\fG}{\f G}

\newcommand{\fM}{\f M}

\newcommand{\fP}{\f P}

\newcommand{\fR}{\f R}

\newcommand{\fX}{\f X}
\newcommand{\fY}{\f Y}

\newcommand{\fc}{\f c}

\newcommand{\fm}{\f m}

\newcommand{\fp}{\f p}

\newcommand{\rD}{\r D}

\newcommand{\rH}{\r H}
\newcommand{\rI}{\r I}
\newcommand{\rJ}{\r J}

\newcommand{\rL}{\r L}

\newcommand{\rN}{\r N}

\newcommand{\rR}{\r R}

\newcommand{\rT}{\r T}
\newcommand{\rU}{\r U}
\newcommand{\rV}{\r V}

\newcommand{\rd}{\r d}

\newcommand{\rn}{\r n}

\newcommand{\sA}{\s A}
\newcommand{\sB}{\s B}
\newcommand{\sC}{\s C}
\newcommand{\sD}{\s D}

\newcommand{\sF}{\s F}

\newcommand{\sH}{\s H}
\newcommand{\sI}{\s I}

\newcommand{\sL}{\s L}
\newcommand{\sM}{\s M}
\newcommand{\sN}{\s N}
\newcommand{\sO}{\s O}
\newcommand{\sP}{\s P}
\newcommand{\sQ}{\s Q}

\newcommand{\sU}{\s U}
\newcommand{\sV}{\s V}
\newcommand{\sW}{\s W}
\newcommand{\sX}{\s X}

\newcommand{\sfe}{\sf e}

\newcommand{\ab}{\r{ab}}
\newcommand{\ac}{\r{ac}}
\newcommand{\Ad}{\r{Ad}}
\newcommand{\anti}{\r{anti}}
\newcommand{\BT}{\sf{BT}}
\newcommand{\can}{\r{can}}
\newcommand{\cl}{\r{cl}}
\newcommand{\cusp}{\r{cusp}}
\newcommand{\DR}{\r{dR}}
\newcommand{\et}{\acute{\r{e}}\r{t}}
\newcommand{\Fr}{\sf{F}}
\newcommand{\Gm}[1]{\bG_{m,#1}}
\newcommand{\Gaf}{\widehat{\bG}_{a}}
\newcommand{\Gmf}{\widehat{\bG}_{m}}
\newcommand{\id}{\r{id}}

\newcommand{\KS}{\r{KS}}
\newcommand{\loc}{\r{loc}}
\newcommand{\lt}{\r{lt}}
\newcommand{\LT}{\cL\!\cT}
\newcommand{\Nilp}{\sf{Nilp}}
\newcommand{\nr}{\r{nr}}
\newcommand{\ord}{\r{ord}}
\newcommand{\p}{\fp}
\newcommand{\RES}{\r{res}}
\newcommand{\rig}{\r{rig}}
\newcommand{\rrd}{\:\rd}
\newcommand{\tame}{\r{tame}}
\newcommand{\Tp}{\rT_p}
\newcommand{\U}[1]{\prescript{\prime}{}{#1}}
\newcommand{\un}{\r{univ}}
\newcommand{\Ver}{\sf{V}}
\newcommand{\wtimes}{\widehat{\otimes}}

\DeclareMathOperator{\AV}{AV}

\DeclareMathOperator{\Char}{ch}

\DeclareMathOperator{\End}{End}
\DeclareMathOperator{\Ext}{Ext}

\DeclareMathOperator{\Frob}{Frob}

\DeclareMathOperator{\Gal}{Gal}
\DeclareMathOperator{\GL}{GL}
\DeclareMathOperator{\Hom}{Hom}
\DeclareMathOperator{\IM}{Im}

\DeclareMathOperator{\Ker}{Ker}
\DeclareMathOperator{\Lie}{Lie}
\DeclareMathOperator{\Map}{Map}
\DeclareMathOperator{\Mat}{Mat}

\DeclareMathOperator{\Nm}{Nm}

\DeclareMathOperator{\RE}{Re}

\DeclareMathOperator{\Sp}{Sp}
\DeclareMathOperator{\Span}{Span}
\DeclareMathOperator{\Spec}{Spec}
\DeclareMathOperator{\Spf}{Spf}
\DeclareMathOperator{\std}{std}

\DeclareMathOperator{\Sym}{Sym}

\DeclareMathOperator{\vol}{vol}

\begin{document}

\title
{On $p$-adic Waldspurger formula}

\author{Yifeng Liu}
\address{Department of Mathematics, Massachusetts Institute of Technology, Cambridge, MA 02139}
\email{liuyf@math.mit.edu}

\author{Shouwu Zhang}
\address{Department of Mathematics, Princeton University, Princeton, NJ 08544}
\email{shouwu@math.princeton.edu}

\author{Wei Zhang}
\address{Department of Mathematics, Columbia University, New York, NY 10027}
\email{wzhang@math.columbia.edu}

\date{October 19, 2014}
\subjclass[2010]{Primary: 11G40; Secondary: 11F67, 11G18, 11J95}

\begin{abstract}
In this article, we study $p$-adic torus periods for certain $p$-adic valued
functions on Shimura curves coming from classical origin. We prove a $p$-adic
Waldspurger formula for these periods, generalizing the recent work of
Bertolini, Darmon, and Prasanna. In pursuing such a formula, we construct a
new anti-cyclotomic $p$-adic $L$-function of Rankin--Selberg type. At a
character of positive weight, the $p$-adic $L$-function interpolates the
central critical value of the complex Rankin--Selberg $L$-function. Its value
at a Dirichlet character, which is outside the range of interpolation,
essentially computes the corresponding $p$-adic torus period.
\end{abstract}

\maketitle

\tableofcontents

\section{Introduction}
\label{s:introduction}

In this article, we study \emph{$p$-adic torus periods} for certain
$\dC_p$-valued functions on Shimura curves coming from classical origin. We
prove a \emph{$p$-adic Waldspurger formula} for these periods, generalizing
the recent work of Bertolini--Darmon--Prasanna \cite{BDP13}. We may view it
as the counterpart of the classical Waldspurger formula for (complex
automorphic) torus periods. In pursuing such a formula, we construct a new
anti-cyclotomic $p$-adic $L$-function of Rankin--Selberg type. At a character
of positive weight, the $p$-adic $L$-function interpolates the central
critical value of the complex Rankin--Selberg $L$-function. Its value at a
Dirichlet character, which is outside the range of interpolation, essentially
computes the corresponding $p$-adic torus period.

The nonvanishing of such period provides a new criterion for the
nontriviality of Heegner points on modular abelian varieties of $\GL(2)$-type
over totally real number fields. This sort of result, in a slightly different
form, was first obtained by Rubin in \cite{Rub92}, for CM elliptic curves
over the rationals. The generalization to other elliptic curves or abelian
varieties parameterized by modular curves is due to
Bertolini--Darmon--Prasanna assuming the Heegner condition and other control
of ramification (see \cite{BDP13}*{Assumption 5.12} for a list of
conditions). Also, recently we learn that Brooks, in his PhD thesis
\cite{Bro}, obtains a result similar to \cite{BDP13} for classical new forms
under certain control of ramification without assuming the Heegner condition,
which is a special case of our formula when $F=\dQ$. Their method uses
$p$-adic differential operators, traced back to the work of Katz
\cite{Kat78}, which is different from Rubin's. Our result generalizes all
known results and is placed in the framework of Waldspurger formula
(\cite{Wal85}, for the central value of the complex $L$-function), or
Yuan--Zhang--Zhang's general version of Gross--Zagier formula (\cite{YZZ13},
for the central derivative of the complex $L$-function). The method we use is
a global Mellin transform for the Lubin--Tate action on the Igusa tower of
the Shimura curve at infinite level, which is closely related to $p$-adic
differential operators in the classical situation.

We remark that the $p$-adic $L$-function we construct is a distribution, or
equivalently, a rigid analytic function on certain ``space of (twisted)
anti-cyclotomic characters''. Note that the rigid analyticity is crucial for
developing the corresponding Iwasawa theory. Also, our construction assumes
neither the Heegner condition (which is somehow apparent since we work over
totally real fields), nor any control of ramification on either
representations, characters, or test vectors. We only need one assumption --
the prime $\p$ in consideration of the totally real field is split in the CM
extension we use to define torus periods/Heegner points.

\subsection{$p$-adic Maass functions and $p$-adic torus periods}
\label{ss:maass_functions}

Throughout the article, we fix a prime $p$, a totally imaginary number field
$E\subset\dC_p$ with $F$ the maximal totally real subfield whose degree is
$g$. Thus we obtain a distinguished place $\fp$ (resp.\ $\fP$) of $F$ (resp.\
$E$) above $p$. Denote by $\dA$ (resp.\ $\dA^\infty$) the ring of ad\`{e}les
(resp.\ finite ad\`{e}les) of $F$. Let $\eta=\prod \eta_v\colon
F^\times\backslash\dA^\times\to\{\pm1\}$ be the quadratic character
associated to $E/F$. In particular, we have the $L$-function
$L(s,\eta)=\prod_{v<\infty}L(s,\eta_v)$. Let $\delta_E\in\dZ_{>0}$ be the
absolute value of the discriminant of $E$.

Recall from \cite{YZZ13}*{\Sec 1.2.1} that a quaternion algebra $\dB$ over
$\dA$ is \emph{incoherent} if $\Sigma_{\dB}$, the set of places $v$ of $F$
where $\dB$ is ramified, is a finite set of odd cardinality; $\dB$ is
\emph{totally definite} if $\Sigma_{\dB}$ contains all archimedean places of
$F$. Let $\dB$ be a totally definite incoherent quaternion algebra over
$\dA$, which gives rise to a projective system of Shimura curves
$\{X(\dB)_U\}_U$ defined over $F$, indexed by compact open subgroups $U$ of
$\dB^{\infty\times}\colonequals(\dB\otimes_\dA\dA^\infty)^\times$. Put
$X(\dB)=\varprojlim_U X(\dB)_U$. For every $g\in\dB^{\infty\times}$, we
denote by $\rT_g\colon X(\dB)\to X(\dB)$ the induced Hecke morphism by right
translation.

\begin{definition}[$p$-adic Maass function]
We introduce the following objects and notation.
\begin{enumerate}
  \item We say a function $X(\dB)(\dC_p)\to\dC_p$ is a \emph{$p$-adic
      Maass function} if it is the pullback of some locally analytic
      $\dC_p$-valued function on $X(\dB)_U\otimes_F\dC_p$ for some $U$.
      Denote by $\sA_{\dC_p}(\dB^\times)$ the $\dC_p$-vector space
      spanned by all $p$-adic Maass functions on $X(\dB)(\dC_p)$, which
      is a representation of $\dB^{\infty\times}$ such that
      $g\in\dB^{\infty\times}$ acts by $\rT_g^*$. For simplicity, we will
      write $g^*$ instead of $\rT_g^*$ in what follows.

  \item Let $A$ be an abelian variety over $F$, $f\colon X\to A$ be a
      morphism, and $\omega\in\rH^0(A,\Omega^1_A)$ be a differential
      form. Then we have the function $f^*\log_\omega$ on $X(\dC_p)$,
      where $\log_\omega\colon A(\dC_p)\to\dC_p$ is the $p$-adic
      logarithm on $A$ along $\omega$ \cite{Bour}, which is an element in
      $\sA_{\dC_p}(\dB^\times)$. A $p$-adic Maass function is
      \emph{(cuspidal) classical} if it is a finite linear combination of
      functions of the form $f^*\log_\omega$ for different
      $(A,f,\omega)$.

  \item An irreducible $\dB^{\infty\times}$-subrepresentation $\pi$ of
      $\sA_{\dC_p}(\dB^\times)$ is \emph{classical} if $\pi$ contains a
      nonzero classical $p$-adic Maass function.

  \item If $\pi$ is classical, then there exists a simple $F$-abelian
      variety $A$, unique up to isogeny, and an embedding $i\colon
      M\colonequals\End^0(A)\hookrightarrow\dC_p$ such that $\pi$ is the
      space generated by $f^*\log_\omega^i$ for
      $\omega\in\rH^0(A,\Omega_A^1)$ and $f\colon X(\dB)\to A$ a modular
      parametrization (see Notation \ref{no:modular_parametrization}).
      Here, $\log_\omega^i\colon A(\dC_p)\to
      M\otimes_\dQ\dC_p\xrightarrow{i}\dC_p$ is the $M$-linear logarithm.
      In particular, every function in $\pi$ is classical. Moreover, we
      have a decomposition $\pi=\bigotimes'_v\pi_v$ such that $\pi_v$ is
      a representation of $\dB^\times_v$, which is unramified for all but
      finitely many $v$.

  \item Suppose $\pi$ is classical, and we may replace $A$ by $A^\vee$ in
      (4). Then we obtain another classical representation $\pi^\vee$,
      called the \emph{dual} of $\pi$, which is isomorphic to
      $\pi\otimes\omega_\pi^{-1}$ as a representation. Here, $\omega_\pi$
      denotes the central character of $\pi$.
\end{enumerate}
\index{$p$-adic Maass function, $\sA_{\dC_p}(\dB^\times)$}\index{$p$-adic
Maass function, $\sA_{\dC_p}(\dB^\times)$!classical}
\end{definition}

\begin{remark}
It is an interesting question to give a function-theoretical criterion for a
$p$-adic Maass function to be classical.
\end{remark}

\begin{definition}\label{de:e_embedding}
An \emph{$E$-embedding} of $\dB$ is an embedding
\begin{align}\label{eq:e_embedding}
\sfe={\prod}'_v\sfe_v\colon \dA_E^\infty
={\prod}'_{v<\infty}E\otimes_FF_v\hookrightarrow\dB^{\infty}
\end{align}
of $\dA^\infty$-algebras. We say $\dB$ is \emph{$E$-embeddable} if there
exists an $E$-embedding of $\dB$.\index{$E$-embedding}
\end{definition}

Now we take a quaternion algebra $\dB$ together with an $E$-embedding. Put
$X=X(\dB)$ for simplicity.

\begin{definition}[CM-subscheme]\label{de:cm_subscheme}
We define the \emph{CM-subscheme} $Y$ to be $X^{E^\times}$, the subscheme of
$X$ fixed by the action of $\sfe(E^\times)$. In fact, we have $Y=Y^+\coprod
Y^-$ such that $E^\times$ acts on the tangent space of points in $Y^\pm$ via
the character $t\mapsto(t/t^c)^{\pm1}$, where $c$ denotes the nontrivial
element in $\Gal(E/F)$. Both $Y^+(\dC_p)$ and $Y^-(\dC_p)$ are equipped with
the natural profinite topology, and admit a transitive action of
$\dA^{\infty\times}_E$ via Hecke morphisms.
\end{definition}

For a $p$-adic Maass function $\phi\in\sA_{\dC_p}(\dB^\times)$ and locally
constant functions $\varphi_\pm\colon Y^\pm(\dC_p)\to\dC_p$, we define the
\emph{$p$-adic torus periods} to be
\[\sP_{Y^\pm}(\phi,\varphi_\pm)=\int_{Y^\pm(\dC_p)}\phi(y)\varphi_\pm(y)\rrd y,\]
where the Haar measures $\rd y$ have total volume $1$, and the integrals can
be expressed as a finite sums.\index{$p$-adic torus period, $\sP_{Y^\pm}$}

\subsection{A $p$-adic Rankin--Selberg $L$-function}
\label{ss:l_function}

From now on, we will assume $\fp$ is split in $E$. We fix a classical
irreducible representation $\pi$ contained in $\sA_{\dC_p}(\dB^\times)$. Put
$\pi^+=\pi$ and $\pi^-=\pi^\vee$, both as subspaces of
$\sA_{\dC_p}(\dB^\times)$.

\begin{definition}
Let $\chi\colon E^\times\backslash \dA^{\infty\times}_E\to\dC_p^\times$ be a
character.
\begin{enumerate}
  \item We say $\chi$ is a \emph{$p$-adic character of weight $w\in\dZ$}
      if there exists a compact open subgroup $V$ of
      $\dA^{\infty\times}_E$ such that $\chi(t)=(t_{\fP}/t_{\fP^c})^w$
      for $t\in V$.

    \item Let $\iota\colon\dC_p\xrightarrow{\sim}\dC$ be an isomorphism.
        For a locally algebraic character $\chi$ of weight $w$, we attach
        following local characters
      \begin{itemize}
      \item $\chi^{(\iota)}_v=1$ if $v|\infty$ but not equal to
          $\iota\res_F$;

      \item $\chi^{(\iota)}_v(z)=(z/z^c)^w$ for $v=\iota\res_F$,
          where $z\in
          E\otimes_{F,\iota}\dR\xrightarrow{\iota\res_E}\dC$;

      \item $\chi^{(\iota)}_v=\iota\circ\chi_v$ for $v<\infty$ but
          $v\neq\p$;

      \item
          $\chi^{(\iota)}_\p(t)=\iota\((t_{\fP}/t_{\fP^c})^{-w}\chi_\p(t)\)$
          for $t\in E_\p^\times$.
      \end{itemize}
      In particular,
      $\chi^{(\iota)}\colonequals\otimes_v\chi^{(\iota)}_v\colon
      \dA^\times_E\to\dC^\times$ is an automorphic character, which is
      called the \emph{$\iota$-avatar} of $\chi$.

  \item We say a $p$-adic character of weight $w$ is \emph{$\pi$-related}
      if $\omega_\pi\cdot\chi\res_{\dA^{\infty\times}}=1$, and
      \[\epsilon(1/2,\pi_v,\chi_v)=\chi_v(-1)\eta_v(-1)\epsilon(\dB_v)\]
      holds for every finite place $v\neq\p$ of $F$, where
      $\epsilon(1/2,\pi_v,\chi_v)$ is the local Rankin--Selberg
      $\epsilon$-factor and $\epsilon(\dB_v)\in\{\pm1\}$ is the Hasse
      invariant. Denote by $\Xi(\pi)_w$ be the set of all $\pi$-related
      $p$-adic characters of weight $w$.

  \item Define $\Omega_{X,Y^\pm}$ to be the restriction
      $\Omega^1_X\res_{Y^\pm}$, which is an
      $\dA^{\infty\times}_E$-equivariant sheaf on $Y^\pm$. For
      $\chi\in\Xi(\pi)_k$, we define $\sigma_\chi^\pm$ to be the subspace
      of $\rH^0(Y^\pm,\Omega_{X,Y^\pm}^{\otimes-k})\otimes_F\dC_p$
      consisting of $\varphi$ such that $t^*\varphi=\chi(t)^{\pm
      1}\varphi$ for all $t\in\dA^{\infty\times}_E$.
\end{enumerate}
\index{$p$-adic character of weight $w$}\index{$p$-adic character of weight
$w$!$\pi$-related, $\Xi(\pi)_w$}\index{$p$-adic character of weight
$w$!$\iota$-avatar of}
\end{definition}

There is an (ind-)proper rigid analytic curve over $\dC_p$, which
parameterizes locally analytic characters of $\dA^{\infty\times}_E$ and
contains $\bigcup_\dZ\Xi(\pi)_w$ as a Zariski dense subset. We denote by
$\sD(\pi)$ the coordinate ring of such curve, which we call the
\emph{$\pi$-related distribution algebra}. It is a complete $\dC_p$-algebra;
see Definition \ref{de:character_space} for its rigorous definition.
\index{distribution algebra!$\pi$-related, $\sD(\pi)$}

To state our result for $p$-adic $L$-function, we need to fix some $p$-adic
pairing and archimedean pairing. For $p$-adic pairing, we need to make three
choices:
\begin{itemize}
  \item a \emph{$p$-adic Petersson inner product} for $\pi$, which
      amounts to a nonzero $\dB^{\infty\times}$-invariant bilinear
      pairing
      \[(\;,\;)_\pi\colon\pi^+\times\pi^-\to\dC_p.\]
      Such pairing exists uniquely up to a scalar in $\dC_p^\times$;

  \item an \emph{abstract conjugation}, that is, an
      $\dA^{\infty\times}_E$-equivariant isomorphism
      \[\bc\colon Y^+(\dC_p)\xrightarrow{\sim} Y^-(\dC_p);\]

  \item an additive character $\psi\colon F_\p\to\dC_p^\times$ of level
      $0$, that is, the kernel of $\psi$ contains $O_\p$ but not
      $\p^{-1}$, where $O_\p$ is the ring of integers in $F_\p$.
\end{itemize}
\index{abstract conjugation, $\bc$}The abstract conjugation $\bc$ and $\psi$
induce a $\dA^{\infty\times}_E$-invariant pairing
\[(\;,\;)_\chi\colon\sigma_\chi^+\times\sigma_\chi^-\to\dC_p\]
by the formula
$(\varphi_+,\varphi_-)_\chi=(\varphi_+\otimes\omega_{\psi+}^k)\cdot
\bc^*(\varphi_-\otimes\omega_{\psi-}^k)$, where the right-hand side is a
constant function on $Y^+$, hence an element in $\dC_p$. Here,
$\omega_{\psi\pm}$ is a section of $\Omega_{X,Y^\pm}\otimes_F\dC_p$
determined by $\psi$ (see \eqref{eq:omega_lt} for details).

For archimedean pairing, we assume $k\geq1$. For every
$\iota\colon\dC_p\xrightarrow{\sim}\dC$, we have a
$\dB^{\infty\times}\times\dA^{\infty\times}_E$-invariant $\iota$-linear
pairing
\[(\;,\;)_{\pi,\chi}^{(\iota)}\colon
(\pi^+\otimes\sigma_\chi^+)\times(\pi^-\otimes\sigma_\chi^-)\to\dC,\]
such that for $\phi_\pm\colonequals f_\pm^*\log_{\omega_\pm}\in\pi^\pm$ and
$\varphi_\pm\in\sigma_\chi^\pm$,
\[(\phi_+\otimes\varphi_+,\phi_-\otimes\varphi_-)_{\pi,\chi}^{(\iota)}
=(\iota\varphi_+\otimes\fc_\iota^*\iota\varphi_-\otimes\mu^k)\times
\int_{X_\iota(\dC)}\frac{\Theta_\iota^{k-1}f_+^*\omega_+\otimes\fc_\iota^*\Theta_\iota^{k-1}f_-^*\omega_-}{\mu^k}\rrd
x,\]
where
\begin{itemize}
  \item $\fc_\iota$ is the complex conjugation on (the underlying real
      analytic space of) $X_\iota\colonequals X\otimes_{F,\iota}\dC$;

  \item $\mu$ is an arbitrary Hecke invariant hyperbolic metric on
      $X_\iota(\dC)$;

  \item $\iota\varphi_+\otimes\fc_\iota^*\iota\varphi_-\otimes\mu^k$ is
      regarded as a complex number since it is a constant function on
      $(Y^+\otimes_{F,\iota}\dC)(\dC)$;

  \item $\Theta_\iota$ is the Shimura--Maass operator on $X_\iota$ (see
      \Sec\ref{ss:comparison_differential} for details); and

  \item $\rd x$ is the Tamagawa measure on $X_\iota(\dC)$.
\end{itemize}
Then we define the \emph{period ratio} $\Omega_\iota(\chi)$ to be the unique
element in $\dC^\times$ such that\index{period ratio, $\Omega_\iota(\chi)$}
\[(\;,\;)_{\pi,\chi}^{(\iota)}=\Omega_\iota(\chi)\cdot\iota(\;,\;)_\pi\otimes\iota(\;,\;)_\chi.\]

\begin{theorem}[$p$-adic $L$-function]\label{th:l_function_maass}
There is a unique element $\sL(\pi)\in\sD(\pi)$ such that for every
$\pi$-related $p$-adic character $\chi\in\Xi(\pi)_k$ of weight $k\geq 1$, and
every isomorphism $\iota\colon\dC_p\xrightarrow{\sim}\dC$,
\begin{align*}
\iota\sL(\pi)(\chi)=L(1/2,\pi^{(\iota)},\chi^{(\iota)})\cdot
\frac{2^{g-3}\delta_E^{1/2}\zeta_F(2)\Omega_\iota(\chi)}{L(1,\eta)^2L(1,\pi^{(\iota)},\Ad)}
\frac{\epsilon(1/2,\psi,\pi^{(\iota)}_\fp\otimes\chi^{(\iota)}_{{\fP}^c})}
{L(1/2,\pi^{(\iota)}_\fp\otimes\chi^{(\iota)}_{{\fP}^c})^2},
\end{align*}
where $\pi^{(\iota)}\colonequals\pi\otimes_{\dC_p,\iota}\dC$ and
$\pi_\fp^{(\iota)}\colonequals\pi_\fp\otimes_{\dC_p,\iota}\dC$; and global
$L$-functions do not include archimedean factors.
\end{theorem}

\begin{remark}
The element $\sL(\pi)$ is our \emph{anti-cyclotomic $p$-adic $L$-function}
for $\pi$. It depends only on the choices of a $p$-adic Petersson inner
product $(\;,\;)_\pi$, an abstract conjugation $\bc$, and an additive
character $\psi$ of $F_\p$ of level $0$. More precisely,
\begin{enumerate}
  \item if we change $(\;,\;)_\pi$ to $(\;,\;)'_\pi=c(\;,\;)_\pi$ for
      some $c\in\dC_p^\times$, then $\sL(\pi)$ is multiplied by $c^{-1}$;

  \item if we change $\bc$ to $\bc'=\rT_t\circ\bc$ for some
      $t\in\dA^{\infty\times}_E$, then $\sL(\pi)$ is multiplied by $[t]$,
      the Dirac distribution at $t$;

  \item if we change $\psi$ to $\psi_a$ for some $a\in O_\p^\times$,
      where $\psi_a(x)=\psi(ax)$ for $x\in F_\p$, then $\sL(\pi)$ is
      multiplied by $\omega_{\pi_\p}(a)\cdot [a]^2$, where $a$ is
      regarded at the place $\fP^c$ in $[a]$.
\end{enumerate}
\index{anti-cyclotomic $p$-adic $L$-function, $\sL(\pi)$}
\end{remark}

In \Sec\ref{ss:l_function_abelian}, we will state a version of Theorem
\ref{th:l_function_maass} in terms of Heegner cycles on abelian varieties,
which implies Theorem \ref{th:l_function_maass} by Lemma \ref{le:l_function}.

\subsection{A $p$-adic Waldspurger formula}
\label{ss:waldspurger_formula}

Note that the set of interpolation of $\sL(\pi)$ is $\bigcup_{k\geq
1}\Xi(\pi)_k$. Thus, a natural question would be seeking the value of
$\sL(\pi)(\chi)$ at a locally constant character $\chi$, that is,
$\chi\in\Xi(\pi)_0$. The main theorem stated in this section answers this
question, through a so called \emph{$p$-adic Waldspurger formula}.

Having chosen a $p$-adic Petersson inner product $(\;,\;)_\pi$ for $\pi$ and
an abstract conjugation $\bc$, we may define a nonzero element
$\alpha^\natural$, called the \emph{local period}, in the space
\[\Hom_{\dA^{\infty\times}_E}(\pi^+\otimes\sigma_\chi^+,\dC_p)\otimes\Hom_{\dA^{\infty\times}_E}(\pi^-\otimes\sigma_\chi^-,\dC_p)\]
such that for $\phi_\pm\in\pi^\pm$, $\varphi_\pm\in\sigma_\chi^\pm$ and
$\iota\colon\dC_p\xrightarrow{\sim}\dC$, the complex number
$\iota\alpha^\natural(\phi_+,\phi_-;\varphi_+,\varphi_-)$ is a regularization
of the following (formal) matrix coefficient integral
\[\int_{\dA^{\infty\times}\backslash
\dA^{\infty\times}_E}\iota(t^*\phi_+,\phi_-)_\pi\cdot\iota(t^*\varphi_+,\varphi_-)_\chi\rrd
t.\] See Definition \ref{de:matrix_integral} for details. \index{local period, $\sQ$}

\begin{theorem}[$p$-adic Waldspurger formula]\label{th:waldspurger_maass}
Let $\chi\in\Xi(\pi)_0$ be a $p$-adic character of weight $0$, that is, a
locally constant character. We have for every $\phi_\pm\in\pi^\pm$ and
$\varphi_\pm\in\sigma_\chi^\pm$,
\[\sP_{Y^+}(\phi_+,\varphi_+)\sP_{Y^-}(\phi_-,\varphi_-)=\sL(\pi)(\chi)\cdot
\frac{L(1/2,\pi_\fp\otimes\chi_{\fP^c})^2}{\epsilon(1/2,\psi,\pi_\fp\otimes\chi_{\fP^c})}\cdot
\alpha^\natural(\phi_+,\phi_-;\varphi_+,\varphi_-).\]
\end{theorem}

We note that the ratio
$\sL(\pi)(\chi)/\epsilon(1/2,\psi,\pi_\fp\otimes\chi_{\fP^c})$ does not
depend on $\psi$. In \Sec\ref{ss:heegner_waldspurger}, we will state a
version of Theorem \ref{th:waldspurger_maass} in terms of Heegner cycles on
abelian varieties, which implies Theorem \ref{th:waldspurger_maass} by Lemma
\ref{le:waldspurger}.

\subsection{Notation and conventions}
\label{ss:notation_conventions}

\begin{itemize}
  \item Denote by $F_\p^\nr$ (resp.\ $F_\p^\ab$) the completion of the
      maximal unramified (resp.\ abelian) extension of $F_\p$ in $\dC_p$
      and $O_\p^\nr$ (resp.\ $O_\p^\ab$) its ring of integers. We denote
      by $\kappa$ the residue field of $O_\p^\nr$, which is isomorphic to
      $\dF_p^\ac$.

  \item Denote by $F^\times_\cl$ (resp.\ $E^\times_\cl$) the closure of
      $F^\times$ (resp.\ $E^\times$) in $\dA^{\infty\times}$ (resp.\
      $\dA^{\infty\times}_E$).

  \item Put $O_\p^\anti=O_{E_\p}^\times/O_\p^\times$. We will write
      elements $t\in E_\p^\times$ in the form $(t_\bullet,t_\circ)$ where
      $t_\bullet\in F_\p^\times$ (resp.\ $t_\circ\in F_\p^\times$) is the
      component at $\fP^c$ (resp.\ $\fP$).

  \item We fix an identification between $\dB_\p$ and $\Mat_2(F_\p)$ such
      that in \eqref{eq:e_embedding},
      \[\sfe_\p(E\otimes_FF_\p)=\left(
                                \begin{array}{cc}
                                  E_{{\fP}^c} &  \\
                                   & E_\fP \\
                                \end{array}
                              \right).\]

  \item Denote by $\rJ$ the matrix $\left(
                                   \begin{array}{cc}
                                      & 1 \\
                                     -1 &  \\
                                   \end{array}
                                 \right)$ in $\dB_\p=\Mat_2(F_\p)$.

  \item For $m\in\dZ$, define the \emph{$\p$-Iwahori subgroup of level
      $m$} to be
      \begin{align*}
      U_{\p,m}&=\left\{g\in\GL_2(O_\p)\res g\equiv\left(
                                                    \begin{array}{cc}
                                                      1 & * \\
                                                      0 & 1 \\
                                                    \end{array}
                                                  \right) \mod \p^m\right\}
      \quad\text{if $m\geq0$},\\
      U_{\p,m}&=\left\{g\in\GL_2(O_\p)\res g\equiv\left(
                                                    \begin{array}{cc}
                                                      1 & 0 \\
                                                      * & 1 \\
                                                    \end{array}
                                                  \right) \mod \p^{-m}\right\}
      \quad\text{if $m<0$}.
      \end{align*}

  \item $\dN=\{m\in\dZ\res m\geq 0\}$. We write elements in $A^{\oplus
      m}$ in columns for an object $A$ (with a well-defined underlying
      set) in an abelian category.

  \item The local or global Artin reciprocity maps send uniformizers to
      geometric Frobenii.

  \item If $G$ is a reductive group over $F$, we always take the Tamagawa
      measure when integrating on the ad\`{e}lic group $G(\dA)$. In
      particular, the total volume of $\dA^\times E^\times\backslash
      \dA^\times_E$ is $2$ under such measure.

  \item For a relative (formal) scheme $X/S$, we will simply write
      $\Omega^1_X$ instead of $\Omega^1_{X/S}$ for the sheaf of relative
      differentials if the base is clear in the context.

  \item Denote by $\Gmf$ (resp.\ $\Gaf$) the multiplicative (resp.\
      additive) formal group. We also regard $\Gmf$ as a formal
      $p$-divisible group. They have the coordinate $T$.

  \item Denote by $\LT$ the Lubin--Tate group over $O_\p^\nr$, which is
      unique up to isomorphism.

  \item Denote by $F_\p^\lt$ the field extension of $F_\p^\nr$ by adding
      the ``period'' of the Lubin--Tate group $\LT$ (see
      \cite{ST01}*{page 460}). It is not discrete unless $F_\p=\dQ_p$ by
      \cite{ST01}*{Lemma 3.9}.

  \item In this article, we will only use basic knowledge about rigid
      analytic varieties over complete $p$-adic fields in the sense of
      Tate. The readers may use the book \cite{BGR84} for a reference. If
      $\sX$ is an $L$-rigid analytic variety for some complete
      non-archimedean field $L$, we denote by $\sO(\sX,K)$ the complete
      $K$-algebra of $K$-valued rigid analytic functions on $\sX$ for any
      complete field extension $K/L$.
\end{itemize}

We fix the Haar measure $\rd t_v$ on $F_v^\times\backslash E_v^\times$ for
every place $v$ of $F$ determined by the following conditions:
\begin{itemize}
  \item When $v$ is archimedean, the total volume of
      $F_v^\times\backslash
      E_v^\times\simeq\dR^\times\backslash\dC^\times$ is $1$;

  \item When $v$ is split, the volume of the maximal compact subgroup of
      $F_v^\times\backslash E_v^\times\simeq F_v^\times$ is $1$;

  \item When $v$ is nonsplit and unramified, the total volume is $1$;

  \item When $v$ is ramified, the total volume is $2$.
\end{itemize}
Then the product measure $\prod_v\rd t_v$ is $2^{-g}\delta_E^{-1/2}L(1,\eta)$
times the Tamagawa measure (compare with \cite{YZZ13}*{1.6}).

In the main part of the article, we will fix an additive character
$\psi\colon F_\p\to\dC_p^\times$ of level $0$. We choose a generator
$\nu\colon\LT\to\Gmf$ in the free $O_\p$-module $\Hom(\LT,\Gmf)$ of rank $1$.
Then there are unique isomorphisms
\begin{align}\label{eq:lt_data}
\nu_\pm\colon F_\p/O_\p\xrightarrow{\sim}\LT[\p^\infty]
\end{align}
such that the induced composite map
\[\nu[\p^\infty]\circ\nu_\pm\colon F_\p/O_\p\to\Gmf[p^\infty]\subset\dC_p^\times\]
coincides with $\psi^\pm$, where $\psi^+=\psi$ and $\psi^-=\psi^{-1}$.

\subsubsection*{Acknowledgements}

We would like to thank Daniel Disegni for his careful reading of the draft
and useful comments. Y.~L. is supported by NSF grant DMS \#1302000; S.~Z. is
supported by NSF grant DMS \#0970100 and \#1065839; W.~Z. is supported by NSF
grant DMS \#1301848, and a Sloan research fellowship.

\section{Arithmetic of quaternionic Shimura curves}
\label{s:arithmetic_quaternionic}

In this chapter, we study some $p$-adic arithmetic properties of quaternionic
Shimura curves over a totally real field. We start from the local theory of
some $p$-adic Fourier analysis on Lubin--Tate groups, following the work of
\cite{ST01}, in \Sec\ref{ss:fourier_theory}. In \Sec\ref{ss:kodaira_spencer},
we study the Gauss--Manin connection and the Kodaira--Spencer isomorphism for
quaternionic Shimura curves, followed by the discussion of universal
convergent modular forms in \Sec\ref{ss:universal_convergent}. In particular,
we prove Theorem \ref{th:family}, which is one of the most crucial technical
results of the article. In \Sec\ref{ss:comparison_differential}, we prove
some results involving comparison with transcendental constructions under a
given complex uniformization. The last one \Sec\ref{ss:proof_claims} contains
the proof of six claims in the previous ones, which requires an auxiliary use
of unitary Shimura curves. In particular, no representations will show up in
this chapter.

\subsection{Fourier theory on Lubin--Tate group}
\label{ss:fourier_theory}

Let $G$ be a topologically finitely generated abelian locally $F_\p$-analytic
group. For a complete field $K$ containing $F_\p$, denote by $C(G,K)$ the
locally convex $K$-vector space of locally analytic $K$-valued functions on
$G$, and $D(G,K)$ its strong dual which is a topological $K$-algebra by
convolution. Recall that the strong dual topology coincides with topology of
uniform convergence on bounded sets in $C(G,K)$. We have a natural continuous
injective homomorphism
\[[\;]\colon G\to D(G,K)^\times\]
by taking Dirac distributions. Moreover, we have $D(G,K)\wtimes_KK'\simeq
D(G,K')$ for a complete field extension $K'/K$.

\begin{definition}[Stable function]\label{de:heartsuit_local}
Let $\sB$ be the generic fiber of (the underlying formal scheme of) $\LT$,
which is isomorphic to the open unit disc over $F_\p^\nr$. We have a map
\[\alpha\colon\sB\times_{\Spf F_\p^\nr}\sB\to\sB\]
induced by the formal group law. A function $\phi\in\sO(\sB,K)$ is
\emph{stable} if
\[\sum_{\Ker[\p]}\phi(\alpha(\;,z))=0.\]
We denoted by $\sO(\sB,K)^\heartsuit$ the subspace of $\sO(\sB,K)$ of stable
functions. The restricted map
$\bM_\loc\colonequals\alpha^*\res\sO(\sB,K)^\heartsuit$ is called the
\emph{local Mellin transform}, whose image is contained in
$\sO(\sB,K)^\heartsuit\wtimes_K\sO(\sB,K)^\heartsuit$.\index{stable function,
$\sO(\sB,K)^\heartsuit$}\index{Mellin transform!local, $\bM_\loc$}
\end{definition}

From now on, we will assume $K$ contains $F_\p^\lt$. By \cite{ST01}*{Theorems
2.3 \& 3.6} (together with the remark after \cite{ST01}*{Corollary 3.7}), we
have a natural \emph{Fourier transform} isomorphism
\begin{align*}
\lambda\colon \sO(\sB,K)\xrightarrow{\sim}D(O_\p,K)
\end{align*}
of topological $K$-algebras, using the homomorphism $\nu\colon\LT\to\Gmf$.
For $z\in\sB(K)$, the assignment $\delta\mapsto(\lambda^{-1}\delta)(z)$
defines an element $\kappa_z$ in the strong dual of $D(O_\p,K)$, that is,
$C(O_\p,K)$. In fact, $\kappa_z$ is a locally analytic character of $O_\p$
(see \cite{ST01}*{\Sec 3}) satisfying $\kappa_z(a)=\nu(a.z)$ for every $a\in
O_\p$.

\begin{remark}\label{re:lt_action}
We have an action of $O_\p$ on $\sB$ coming from the Lubin--Tate group, and
hence on $D(O_\p,K)$ via $\lambda$. More precisely, $t\in O_\p$ acts on
$D(O_\p,K)$ by multiplying $[t]$.
\end{remark}

If we identify $D(O_\p^\times,K)$ as the closed subspace of $D(O_\p,K)$
consisting of distributions supported on $O_\p^\times$, then by
\cite{ST01}*{Lemma 4.6.5}, $\lambda$ induces an isomorphism between
$\sO(\sB,K)^\heartsuit$ and $D(O_\p^\times,K)$. The following diagram
commutes
\[\xymatrix{
\sO(\sB,K)^\heartsuit \ar[rr]^-{\bM_\loc} \ar[d]^-{\lambda}_-{\simeq} &&
\sO(\sB,K)^\heartsuit\wtimes_K\sO(\sB,K)^\heartsuit  \ar[d]^-{\lambda}_-{\simeq} \\
D(O_\p^\times,K) \ar[rr] && D(O_\p^\times,K)\wtimes_K D(O_\p^\times,K),
 }\] where
the bottom arrow is the pushforward of distributions along the diagonal
embedding $O_\p^\times\to O_\p^\times\times O_\p^\times$.

We define the \emph{Lubin--Tate differential operator} $\Theta$ on
$\sO(\sB,K)$ by the formula
\begin{align}\label{eq:theta_local}
\Theta\phi=\frac{\rd\phi}{\nu^*\rd T/T}.
\end{align}
\index{Lubin--Tate differential operator, $\Theta$}

Consider the compact abelian locally $F_\p$-analytic group $O_\p^\anti$,
which will be identified with $O_\p^\times$ via $t\mapsto t/t^c$. We regard
the range of $\bM_\loc$ as
$\sO(\sB,K)^\heartsuit\wtimes_{F_\p}D(O_\p^\anti,F_\p)$. For each $w\in\dZ$,
we have a locally analytic $F_\p$-valued character $t\mapsto t^w$ of
$O_\p^\times$ and hence $O_\p^\anti$, denoted by $\ang{w}\in
C(O_\p^\anti,F_\p)$.

\begin{lem}\label{le:family_local}
Let $\phi\in\sO(\sB,K)^\heartsuit$ be a stable function. We have for $k\geq
0$,
\[\bM_\loc(\phi)(\ang{k})=\Theta^k\phi,\] and $\Theta\bM_\loc(\phi)(\ang{-1})=\phi$.
\end{lem}

\begin{proof}
This follows from \cite{ST01}*{Lemma 4.6 5\&8}.
\end{proof}

\begin{definition}[Admissible function]\label{de:admissible_local}
For $n\in\dN$, a stable function $\phi\in\sO(\sB,K)^\heartsuit$ is
\emph{$n$-admissible} if $\phi(\alpha(\;,\nu_+(x)))=\psi(x)\phi$ for every
$x\in \p^{-n}/O_\p$, where $\nu_+$ is introduced in
\eqref{eq:lt_data}.\index{stable function,
$\sO(\sB,K)^\heartsuit$!admissible}
\end{definition}

\begin{lem}\label{le:heartsuit_local}
Let $\phi\in\sO(\sB,K)^\heartsuit$ be an $n$-admissible stable function. Then
$\lambda(\phi)$ is supported on $1+\p^n$ if $n\geq 1$; in particular,
$\bM_\loc(\phi)(\ang{k})=\bM_\loc(\phi)(\chi\ang{k})$ for any $k\in\dZ$ and
any (locally constant) character $\chi\colon O_\p^\anti\to K^\times$ that is
trivial on $(1+\p^n)^\times$.
\end{lem}

\begin{proof}
This again follows from \cite{ST01}*{Lemma 4.6.5} and the relation among
$\nu$, $\nu_+$ and $\psi$.
\end{proof}

\subsection{Kodaira--Spencer isomorphism}
\label{ss:kodaira_spencer}

Let $\dB$ be a totally definition incoherent quaternion algebra over $\dA$.
Denote by $\Gamma$ the set of all compact open subgroups $U^\p$ of
$\dB^{\infty\p\times}=(\dB\otimes_\dA\dA^{\infty\p})^\times$, which is a
filtered partially ordered set under inclusion. Recall that we have the
system of Shimura curves $\{X_U\}_U$ over $\Spec F$.

\begin{definition}[Shimura pro-curve of Iwahori level]
For a tame level $U^\p\in\Gamma$ and $m\in\dZ$, put $X(m,U^\p)=X_{U^\p
U_{\p,m}}\otimes_FF_\p^\nr$ where we recall that $U_{\p,m}$ is the
$\p$-Iwahori subgroup of level $m$. If we take the inverse limit with respect
to $m$, we obtain
\[X(\pm\infty,U^\p)=\varprojlim_{m\to\infty}X(\pm m,U^\p).\]
For $m\in\dN\cup\{\infty\}$, if we take the inverse limit with respect to
$U^\p$, we obtain
\[X(\pm m)=\varprojlim_{U^\p\in \Gamma}X(\pm m,U^\p),\]
which we call the \emph{Shimura pro-curve of $\p$-Iwahori level $m$}.
\end{definition}

We have successive surjective morphisms
\[X(\pm\infty)\to\cdots\to
X(\pm1)\to X(0),\] which are equivariant under the Hecke actions of
$\dB^{\infty\p\times}$. By Carayol \cite{Car86}*{\Sec 6}, $X(0)$ admits a
canonical smooth model (see \cite{Mil92}*{Definition 2.2} for its meaning)
$\cX$ over $\Spec O_\p^\nr$. Strictly speaking, Carayol assumed that
$F\neq\dQ$. But when $F=\dQ$, one may take $\cX$ to be the model defined by
modular interpretation which is well-known.

We recall the construction in \cite{Car86}*{1.4} of an $O_\p$-divisible group
$\cG$ on $\cX$. For $m\geq 1$, denote the principal congruence subgroup of
level $\p^m$ by
\[U_{\p,m}^{\r{pr}}=\{g\in U_{\p,0}\res g\equiv 1\mod \p^m\}\]
and $X(m)^{\r{pr}}$ the corresponding covering of $X(0)$. Consider the right
action of $U_{\p,m}^{\r{pr}}/U_{\p,0}\simeq\GL_2(O_\p/\p^m)$ on
$(\p^{-m}/O_\p)^{\oplus 2}$ sending $v\in(\p^{-m}/O_\p)^{\oplus 2}$ to
$g^{-1}v$. Then the quotient
\begin{align*}
\(X(m)^{\r{pr}}\times(\p^{-m}/O_\p)^{\oplus 2}\)/(U_{\p,m}^{\r{pr}}/U_{\p,0})
\end{align*}
defines a finite flat group scheme $G_m$, with strict $O_\p$-action, over
$X(0)$. The inductive system $\{G_m\}_{m\geq 1}$ defines an $O_\p$-divisible
group $G$ over $X(0)$ (which is however denoted by $E_\infty$ in
\cite{Car86}). In particular, over $X(+\infty)$ (resp.\ $X(-\infty)$), we
have an exact sequence
\begin{align}\label{eq:g_split}
\xymatrix{
0 \ar[r] & F_\p/O_\p \ar[r] & G \ar[r] & F_\p/O_\p \ar[r] & 0
}
\end{align}
such that the second arrow is the inclusion into the first (resp.\ second)
factor and the third arrow is the projection onto the second (resp.\ first)
factor.

By \cite{Car86}*{6.4}, the $O_\p$-divisible group $G$ extends uniquely to an
$O_\p$-divisible group $\cG$ of height $2$ over $\cX$, together with an
action by $\dB^{\infty\p\times}$ that is compatible with the Hecke action on
the base.

For $m\geq 1$, put $\cX^{(m)}=\cX\otimes_{O_\p}O_\p/\p^m$ and
$\cG^{(m)}=\cG\res_{\cX^{(m)}}$. We have the following exact sequence
\begin{align}\label{eq:hodge_pre}
\xymatrix{
0 \ar[r] & \ul\omega^{\bullet(m)}_p \ar[r] & \cL_p^{(m)} \ar[r] &
(\ul\omega^{\circ(m)}_p)^\vee \ar[r] & 0, }
\end{align}
where if we put $h=[F_\p:\dQ_p]$,
\begin{itemize}
  \item $\cL_p^{(m)}$ is the Dieudonn\'{e} crystal of $\cG^{(m)}$
      evaluated at $\cX^{(m)}$, which is a locally free sheaf of rank
      $2h$;
  \item $\ul\omega^{\bullet(m)}_p$ is the sheaf of invariant
      differentials of $\cG^{(m)}/\cX^{(m)}$, which is a locally free
      sheaf of rank $1$;
  \item $\ul\omega^{\circ(m)}_p$ is the sheaf of invariant differentials
      of $(\cG^{(m)})^\vee/\cX^{(m)}$, which is a locally free sheaf of
      rank $2h-1$.
\end{itemize}
They are equipped with actions of $O_\p$ under which \eqref{eq:hodge_pre} is
equivariant. The projective system of \eqref{eq:hodge_pre} for all $m\geq1$
induces the following $O_\p$-equivariant exact sequence
\begin{align*}
\xymatrix{
0 \ar[r] & \ul\omega^\bullet_p \ar[r] & \cL_p \ar[r] & (\ul\omega^\circ_p)^\vee \ar[r] & 0, }
\end{align*}
of locally free sheaves over $\cX^\wedge$, the formal completion of $\cX$
along its special fiber. Let $\cL$ (resp.\ $\ul\omega^{\circ\vee}$) be the
maximal sub-sheaf of $\cL_p$ (resp.\ $(\ul\omega^\circ_p)^\vee$) where $O_\p$
acts via the structure map. Then we have the following
$\dB^{\infty\p\times}$-equivariant exact sequence
\begin{align}\label{eq:hodge}
\xymatrix{
0 \ar[r] & \ul\omega^\bullet \ar[r] & \cL \ar[r] & \ul\omega^{\circ\vee} \ar[r] & 0, }
\end{align}
where $\ul\omega^\bullet=\ul\omega^\bullet_p$. We call \eqref{eq:hodge} the
\emph{formal Hodge exact sequence}.

We have the \emph{Gauss--Manin connection}
\begin{align}\label{eq:gm_formal}
\nabla\colon\cL\to\cL\otimes\Omega^1_{\cX^\wedge},
\end{align}
which is equivariant under the Hecke action of $\dB^{\infty\p\times}$. We
have the following two lemmas whose proofs will be given in
\Sec\ref{ss:proof_claims}.

\begin{lem}\label{le:algebraizable}
The formal Hodge exact sequence \eqref{eq:hodge} is algebraizable, that is,
it is the formal completion of an exact sequence
\begin{align}\label{eq:hodge_sequence}
\xymatrix{
0 \ar[r] & \ul\omega^\bullet \ar[r] & \cL \ar[r] & \ul\omega^{\circ\vee} \ar[r] & 0, }
\end{align}
on $\cX$. Here, we abuse of notation by adopting the same symbols for these
sheaves. Moreover, the Gauss--Manin connection \eqref{eq:gm_formal} is
algebraizable.
\end{lem}

We simply call \eqref{eq:hodge_sequence} the \emph{Hodge exact sequence}.

\begin{remark}\label{re:local_system}
For $m\geq1$, one may consider the right action of
$U_{\p,m}^{\r{pr}}/U_{\p,0}\simeq\GL_2(O_\p/\p^m)$ on $(O_\p/\p^m)^{\oplus
2}$ sending $v\in(O_\p/\p^m)^{\oplus 2}$ to $g^tv$. Then the quotient
\[\(X(m)^{\r{pr}}\times(O_\p/\p^m)^{\oplus 2}\)/(U_{\p,m}^{\r{pr}}/U_{\p,0})\]
defines an $O_\p/\p^m$-local system $\rL_m$ on $X(0)$ of rank $2$. Denote by
$\rL$ the $O_\p$-local system over $X(0)$ defined by $(\rL_m)_{m\geq 1}$.
Then $\rL$ is canonically isomorphic to the restriction of $\cL$ on the
generic fiber.
\end{remark}

\begin{proposition}[Kodaira--Spencer isomorphism]\label{pr:ks_isomorphism}
The composition of
\begin{align}\label{eq:ks_pre}
\ul\omega^\bullet\to\cL\xrightarrow{\nabla}\cL\otimes\Omega^1_\cX\to\ul\omega^{\circ\vee}\otimes\Omega^1_\cX
\end{align}
is an isomorphism of quasi-coherent sheaves, where $\ul\omega^\circ$ is the
dual sheaf of $\ul\omega^{\circ\vee}$ and the tensor product is over the
structure sheaf $\sO_\cX$.
\end{proposition}

The isomorphism \eqref{eq:ks_pre} induces the following
($\dB^{\infty\p\times}$-equivariant) \emph{Kodaira--Spencer isomorphism}
\begin{align}\label{eq:ks}
\KS\colon\ul\omega^\bullet\otimes\ul\omega^\circ\xrightarrow{\sim}\Omega^1_\cX.
\end{align}

For $w\in\dN$, put $\cL^{[w]}=\Sym^w\cL\otimes\Sym^w\cL^\vee$. The
Gauss--Manin connection $\nabla^\vee$ on the dual sheaf $\cL^\vee$ and the
original one $\nabla$ induce a connection
$\nabla^{[w]}\colon\cL^{[w]}\to\cL^{[w]}\otimes\Omega^1_\cX$. Define
$\Theta^{[w]}$ to be the composite map
\begin{align}\label{eq:theta}
(\Omega^1_\cX)^{\otimes w}\xrightarrow{\KS^{-1}} (\ul\omega^\bullet)^{\otimes
w}\otimes(\ul\omega^\circ)^{\otimes w}\to \cL^{[w]} \xrightarrow{\nabla^{[w]}}\c
L^{[w]}\otimes\Omega^1_\cX.
\end{align}

Denote by $\cX(0)$ the (dense) open subscheme of $\cX$ with all points on the
special fiber where $\cG$ is supersingular removed. For $m\in\dN$, denote by
$\cX(m)$ the functor classifying $O_\p/\p^m$-equivariant \emph{frames}, that
is, exact sequences
\[\xymatrix{0\ar[r]& \LT[\p^m] \ar[r]& \cG[\p^m] \ar[r]& \p^{-m}/O_\p \ar[r]&
0}\] over $\cX(0)$ with terms fixed. Then $\cX(m)$ is representable by a
scheme \'{e}tale over $\cX(0)$, which we again denote by $\cX(m)$. Put
$\cX(\infty)=\varprojlim_{m\to\infty}\cX(m)$. We define $\cG$,
$\nabla^{[w]}$, $\KS$, $\Theta^{[w]}$, and the sequence
\eqref{eq:hodge_sequence} for $\cX(m)$ ($m\in\dN\cup\{\infty\}$) via
restriction and denote them by the same notation. Over $\cX(\infty)$, we have
the \emph{universal frame}
\begin{align}\label{eq:igusa}
\xymatrix{0\ar[r]& \LT \ar[r]^-{\varrho_\bullet^\un}& \cG \ar[r]^-{\varrho_\circ^\un} & F_\p/O_\p \ar[r]& 0}.
\end{align}
There is a $\dB^{\infty\p\times}$-equivariant action of $O_{E_\p}^\times$ on
the morphism $\cX(\infty)\to\cX(0)$. More precisely, for
$(t_\bullet,t_\circ)\in O_{E_\p}^\times$, the pullback of \eqref{eq:igusa} is
the frame
\begin{align}\label{eq:action}
\xymatrix{0\ar[r]& \LT \ar[rr]^-{\varrho_\bullet^\un\circ t_\bullet^{-1}}&&
\cG \ar[rr]^-{\varrho^\un_\circ\circ t_\circ} && F_\p/O_\p \ar[r]& 0}.
\end{align}

The trivialization \eqref{eq:lt_data} induces \emph{transition isomorphisms}
\begin{align}\label{eq:infinity}
\Upsilon_\pm\colon X(\pm\infty)\otimes_{F_\p^\nr}F_\p^\ab\xrightarrow{\sim}
\cX(\infty)\otimes_{O_\p^\nr}F_\p^\ab
\end{align}\index{transition isomorphism, $\Upsilon_\pm$}
such that the pullback of \eqref{eq:igusa} coincide with \eqref{eq:g_split}.
It is $\dB^{\infty\p\times}$-equivariant and $O_{E_\p}^\times$-equivariant
(resp.\ $O_{E_\p}^\times$-anti-equivariant) for $X(+\infty)$ (resp.\
$X(-\infty)$). We have the following commutative diagram
\[\xymatrix{
X\otimes_FF_\p^\ab \ar[rr]^-{\rT_\rJ}\ar[d] && X\otimes_FF_\p^\ab \ar[d] \\
X(+\infty)\otimes_{F_\p^\nr}F_\p^\ab \ar[rd]^-{\Upsilon_+}\ar[rr]^-{\rT_\rJ} && X(-\infty)\otimes_{F_\p^\nr}F_\p^\ab \ar[ld]_-{\Upsilon_-} \\
& \cX(\infty)\otimes_{O_\p^\nr}F_\p^\ab, }\] where $\rJ$ is introduced in \Sec\ref{ss:notation_conventions}.

\subsection{Universal convergent modular forms}
\label{ss:universal_convergent}

For $m\in\dN\cup\{\infty\}$, denote by $\fX(m)$ the formal completion of
$\cX(m)$ along its special fiber, which is an affine formal scheme over
$O_\p^\nr$, equipped with an $O_\p$-divisible group $\fG$ induced from $\cG$.
The action of $O_{E_\p}^\times$ \eqref{eq:action} makes it the Galois group
of the $\dB^{\infty\p\times}$-equivariant pro-\'{e}tale Galois cover
$\fX(\infty)\to\fX(0)$. Denote by $O_{E_\p,m}^\times$ the subgroup of
$O_{E_\p}^\times$ that fixes the sub-cover $\fX(m)\to\fX(0)$ for $m\in\dN$.

The following lemma will be proved in \Sec\ref{ss:proof_claims}.

\begin{lem}\label{le:unit_root}
There is a unique quasi-coherent formal sub-sheaf $\cL^\circ$ of $\cL$
(viewed as the formal sheaf induced from $\cX(0)$) such that
\begin{enumerate}
  \item we have the following \emph{unit-root decomposition}
      \[\cL=\ul\omega^\bullet\oplus\cL^\circ;\]
  \item $\nabla\cL^\circ\subset\cL^\circ\otimes\Omega^1_{\fX(0)}$;
  \item for every closed point $x\in\fX(0)(\kappa)$, the restriction of
      $\cL^\circ$ to $\fX(0)_{/x}$, the formal completion at $x$, is free
      of rank $1$.
\end{enumerate}
\end{lem}

We restrict the above decomposition to $\fX(m)$ for $m\in\dN\cup\{\infty\}$.
The splitting in the above lemma induces a map
\[\theta^{[w]}_\ord\colon\cL^{[w]}\to(\ul\omega^\bullet)^{\otimes w}\otimes(\ul\omega^\circ)^{\otimes w}
\xrightarrow{\KS}(\Omega^1_{\fX(m)})^{\otimes w}\] for all $w\in\dN$.

\begin{definition}[Atkin--Serre operator]\label{de:theta_ordinary}
For $m\in\dN\cup\{\infty\}$ and $w\in\dN$, define the \emph{Atkin--Serre
operator} to be
\[\Theta^{[w]}_\ord\colon(\Omega^1_{\fX(m)})^{\otimes w}
\xrightarrow{\Theta^{[w]}\res_{\fX(m)}}\cL^{[w]}\otimes\Omega^1_{\fX(m)}
\xrightarrow{\theta^{[w]}_\ord}(\Omega^1_{\fX(m)})^{\otimes w+1},\] where
$\Theta^{[w]}$ is defined in \eqref{eq:theta}. For $k\in\dN$, define the
\emph{Atkin--Serre operator of degree $k$} to be
\[\Theta^{[w,k]}_\ord=\Theta^{[w+k-1]}_\ord\circ\cdots\circ\Theta^{[w]}_\ord\colon
(\Omega^1_{\fX(m)})^{\otimes w}\to(\Omega^1_{\fX(m)})^{\otimes w+k}.\] In
what follows, $w$ will always be clear from the text, and hence we will
suppress $w$ from notation; in other words, we simply write $\Theta_\ord$
(resp.\ $\Theta_\ord^k$) for $\Theta^{[w]}_\ord$ (resp.\
$\Theta^{[w,k]}_\ord$) for all $w\in\dN$.\index{Atkin--Serre operator,
$\Theta_\ord$}
\end{definition}

Using Serre--Tate coordinates (Proposition \ref{th:serre_tate}), the formal
deformation space of $\LT\oplus F_\p/O_\p$ is canonically isomorphic to
$\LT$. Thus, we have the classifying morphism
\[c\colon\fX(\infty)\to\LT\]
of $O_\p^\nr$-formal schemes. It induces a morphism
\[c_{/x}\colon\fX(\infty)_{/x}\to\LT\]
for every closed point $x\in\fX(\infty)(\kappa)$, where $\fX(\infty)_{/x}$
denotes the formal completion of $\fX(\infty)$ at $x$. The following lemma
and proposition will be proved in \Sec\ref{ss:proof_claims}.

\begin{lem}\label{le:deformation}
The morphism $c_{/x}$ is an isomorphism for every $x$.
\end{lem}

\begin{proposition}\label{pr:lt_action}
There is a morphism $\beta\colon\LT\times_{\Spf
O_\p^\nr}\fX(\infty)\to\fX(\infty)$ such that
\begin{enumerate}
  \item for every $x\in\fX(\infty)(\kappa)$, it preserves
      $\fX(\infty)_{/x}$ and the induced morphism
      \[\beta_{/x}\colon\LT\times_{\Spf O_\p^\nr}\fX(\infty)_{/x}\to\fX(\infty)_{/x}\]
      is simply the formal group law after identifying $\fX(\infty)_{/x}$
      with $\LT$ via $c_{/x}$;

  \item if we equip $\LT$ with the action of
      $O_{E_\p}^\times\times\dB^{\infty\p\times}$ via the inflation
      $O_{E_\p}^\times\to O_\p^\times$ by $t\mapsto t/t^c$ and trivially
      on the second factor, then $\beta$ is
      $O_{E_\p}^\times\times\dB^{\infty\p\times}$-equivariant;

  \item for every $x\in F_\p/O_\p$, the following diagram
      \[\xymatrix{
      \fX(\infty)\widehat\otimes_{O_\p^\nr}F_\p^\ab  \ar[rr]^-{\beta_{\nu_\pm(x)}} \ar[d]_-{\Upsilon_\pm^{-1}}
      &&  \fX(\infty)\widehat\otimes_{O_\p^\nr}F_\p^\ab \ar[d]^-{\Upsilon_\pm^{-1}} \\
      X(\pm\infty)\otimes_{F_\p^\nr}F_\p^\ab \ar[rr]^-{\rT_{\rn^\pm(x)}} &&  X(\pm\infty)\otimes_{F_\p^\nr}F_\p^\ab
      }\] commutes, where
      \[\rn^+(x)=\left(
                 \begin{array}{cc}
                   1 & x \\
                   0 & 1 \\
                 \end{array}
               \right),\quad
      \rn^-(x)=\left(
                 \begin{array}{cc}
                   1 & 0 \\
                   x & 1 \\
                 \end{array}
               \right);
      \] and $\beta_z$ is the restriction of $\beta$ to a point $z$ of $\sB$.
\end{enumerate}
In particular, $\LT$ acts trivially on the special fiber of $\fX(\infty)$ and
such $\beta$ is unique.
\end{proposition}

\begin{definition}\label{de:global_differential}
We put $\omega_\nu=c^*\nu^*\rd T/T$, which is a nowhere zero global
differential form on $\fX(\infty)$, that is, an element in
$\rH^0(\fX(\infty),\Omega^1_{\fX(\infty)})$. We call it a \emph{global
Lubin--Tate differential}. It is clear that the pullback
$\Upsilon_\pm^*\omega_\nu$ depends only $\psi$ (or rather
$\psi^\pm$).\index{global Lubin--Tate differential, $\omega_\nu$}
\end{definition}

For $m\in\dN\cup\{\infty\}$, $w\in\dZ$ and a complete extension $K/F_\p^\nr$,
define the space of \emph{$K$-valued convergent modular forms of weight $w$
and $\p$-Iwahori level $m$} to be
\[\sM^w(m,K)=\rH^0(\fX(m),(\Omega^1_{\fX(m)})^{\otimes
w})\wtimes_{O_\p^\nr}K,\] which is naturally a complete $K$-vector space. It
has a natural action by $O_{E_\p}^\times\times\dB^{\infty\p\times}$. A
convergent modular form of weight $0$ is simply called a \emph{convergent
modular function}. In particular, $\sM^0(m,K)$ is the complete tensor product
of the coordinate ring of $\fX(m)$ and the field $K$, and thus $\sM^w(m,K)$
is naturally a topological $\sM^0(m,K)$-module.\index{convergent modular
form, $\sM^w(m,K)$}\index{convergent modular form, $\sM^w(m,K)$!of finite
tame level, $\sM^w_\flat(m,K)$}

Define the subspace of \emph{$K$-valued convergent modular forms of weight
$w$ and $\p$-Iwahori level $m$ and finite tame level} to be
\[\sM_\flat^w(m,K)=\bigcup_{U^\p\in\Gamma}\sM^w(m,K)^{U^\p}\subset\sM^w(m,K),\]
which is stable under the action of
$O_{E_\p}^\times\times\dB^{\infty\p\times}$. The global Lubin--Tate
differential $\omega_\nu$ provides a canonical
$\dB^{\infty\p\times}$-equivariant isomorphism
$\sM_\flat^w(\infty,K)\simeq\sM_\flat^0(\infty,K)$ for $w\in\dZ$. We will
view $\sM_\flat^w(m,K)$ as a subspace of $\sM_\flat^0(\infty,K)$ for any
$m\in\dN\cup\{\infty\}$ and $w\in\dZ$ via restriction.  In particular for
$w\in\dN$, we have the \emph{Atkin--Serre operator}
\[\Theta_\ord\colon \sM_\flat^w(m,K)\to\sM_\flat^{w+1}(m,K),\]
which is $O_{E_\p}^\times\times\dB^{\infty\p\times}$-equivariant. Recall that
both $O_{E_\p}^\times$ and $\dB^{\infty\p\times}$ act on $\fX(\infty)$ and
thus on $\sM^0(\infty,K)$ for any complete extension $K/F_\p^\nr$.

The morphism $\beta$ induces a \emph{translation map}
\[\beta_z^*\colon\sM^0(\infty,K)\to\sM^0(\infty,K)\]
for every $z\in\sB(K)$.

\begin{definition}[Stable convergent modular forms]\label{de:heartsuit_formal}
A convergent modular function $f\in\sM^0(\infty,K)$ is \emph{stable} if
\[\sum_{z\in\sB[\p]}\beta_z^*f=0.\]
Denote by $\sM^0(\infty,K)^\heartsuit$ the subspace of $\sM^0(\infty,K)$ of
stable convergent modular functions. For every closed point
$x\in\fX(\infty)(\kappa)$, we have the restriction map
\[\RES_x\colon\sM^0(\infty,K)\to\sO(\sB,K)\] by Lemma
\ref{le:deformation}. By Proposition \ref{pr:lt_action}, $f$ is stable if and
only if $\RES_xf$ is stable (Definition \ref{de:heartsuit_local}) for all
$x$.

For $m\in\dN\cup\{\infty\}$ and $w\in\dZ$, recall that we have viewed
$\sM^w(m,K)$ as a subspace of $\sM^0(\infty,K)$ and thus we may define the
space of \emph{$K$-valued stable convergent modular forms of weight $w$,
$\p$-Iwahori level $m$ and finite tame level} to be\index{convergent modular
form, $\sM^w(m,K)$!stable, $\sM_\flat^w(m,K)^\heartsuit$}
\[\sM_\flat^w(m,K)^\heartsuit=\sM_\flat^w(m,K)\cap\sM^0(\infty,K)^\heartsuit.\]
\end{definition}

\begin{remark}
For $m\in\dN$, the space $\sM_\flat^w(m,K)$ (resp.\
$\sM_\flat^w(m,K)^\heartsuit$) generalizes the notation of (the space of)
convergent modular forms (resp.\ convergent modular forms of infinite slope)
of weight $2w$ from modular curves to Shimura curves. Moreover, the space
$\sM_\flat^w(m,K)^\heartsuit$ does not depend on $\psi$ or $\nu$.
\end{remark}

\begin{definition}[Admissible convergent modular forms]\label{de:admissible_formal}
For $n\in\dN$, a stable convergent modular function
$f\in\sM^0(\infty,K)^\heartsuit$ is \emph{$n$-admissible} if
$\beta_{\nu_+(x)}^*f=\psi(x)f$ for all $x\in\p^{-n}/O_\p$. A stable
convergent modular form $f\in\sM_\flat^w(m,K)^\heartsuit$ is $n$-admissible
if it is so, when regarded as an element in $\sM^0(\infty,K)^\heartsuit$. By
Proposition \ref{pr:lt_action} (1), $f$ is $n$-admissible if and only if
$\RES_xf$ is $n$-admissible (in the sense of Definition
\ref{de:admissible_local}) for all $x$.\index{convergent modular form,
$\sM^w(m,K)$!admissible}
\end{definition}

The following lemma is a comparison between the Atkin--Serre operator and the
Lubin--Tate differential operator.

\begin{lem}\label{le:theta}
For a convergent modular form of weight $w$, $\p$-Iwahori level $m$ and
finite tame level $f\in\sM_\flat^w(m,K)$ for some $w,m\in\dN$, we have
\[\RES_x(\Theta_\ord f)=\Theta(\RES_x f)\]
for every $x\in\fX(\infty)(\kappa)$.
\end{lem}

\begin{proof}
It follows from Lemma \ref{le:deformation}, Theorem \ref{th:o_main_a}, and
the definition of $\Theta$ \eqref{eq:theta_local}.
\end{proof}

\begin{definition}[Universal convergent modular form]\label{de:universal_form}
A \emph{universal convergent modular form} of depth $m\in\dN$ and tame level
$U^\p\in\Gamma$ is an element
$\bM\in\sM^0(\infty,K)\wtimes_{F_\p}D(O_\p^\anti,F_\p)$ such that $\bM$ is
$U^\p$-invariant and
\begin{align}\label{eq:universal_form}
t^*\bM=[t]^{-1}\cdot\bM
\end{align}
for $t\in O_{E_\p,m}^\times$.\index{convergent modular form,
$\sM^w(m,K)$!universal}
\end{definition}

\begin{theorem}\label{th:family}
Suppose $K$ contains $F_\p^\lt$. Let $f\in\sM_\flat^w(m,K)^\heartsuit$ be a
stable convergent modular form of weight $w$ and $\p$-Iwahori level $m$, for
some $w,m\in\dN$. Then there is a unique element
$\bM(f)\in\sM^0(\infty,K)\wtimes_{F_\p}D(O_\p^\anti,F_\p)$ such that for
every $k\in\dN$,
\begin{align}\label{eq:family}
\bM(f)(\ang{w+k})=\Theta_\ord^kf,
\end{align}
where in the target, we identify $\sM_\flat^{w+k}(m,K)$ as a subspace of
$\sM^0(\infty,K)$. Moreover, we have
\begin{enumerate}
  \item if $f$ is fixed by $U^\p\in\Gamma$, so is $\bM(f)$;

  \item $\bM(f)$ is a universal convergent modular form of depth $m$
      (Definition \ref{de:universal_form});

  \item if $w\geq 1$, we have
  \[\Theta_\ord \bM(f)(\ang{w-1})=f;\]

  \item Suppose $f$ is $n$-admissible. Then
      $\bM(f)(\ang{k})=\bM(f)(\chi\ang{k})$ any $k\in\dZ$ and any
      (locally constant) character $\chi\colon O_\p^\anti\to K^\times$
      that is trivial on $(1+\p^n)^\times$.
\end{enumerate}
\end{theorem}

We call $\bM(f)$ the \emph{global Mellin transform} of $f$.\index{Mellin
transform!global, $\bM$}

\begin{proof}
The uniqueness is clear since the set $\{\langle w+k\rangle\res k\geq 0\}$
spans a dense subspace of $C(O_\p^\anti,F_\p)$. Regard $f$ as an element in
$\sM^0(\infty,K)^\heartsuit$. Note that the map
\[\beta^*\colon\sM^0(\infty,K)\to\sM^0(\infty,K)\wtimes_K\sO(\sB,K)\] sends
$\sM^0(\infty,K)^\heartsuit$ into
$\sM^0(\infty,K)^\heartsuit\wtimes_K\sO(\sB,K)^\heartsuit$. The element
$\beta^*f$ belongs to the space
$\sM^0(\infty,K)^\heartsuit\wtimes_K\sO(\sB,K)^\heartsuit$, which is
isomorphic to $\sM^0(\infty,K)^\heartsuit\wtimes_{F_\p}D(O_\p^\anti,F_\p)$
via $\lambda$ since $K$ contains $F_\p^\lt$. Define a (continuous
$F_\p$-linear) translation map
\[\tau_w\colon D(O_\p^\anti,F_\p)\to D(O_\p^\anti,F_\p)\]
such that $(\tau_w\phi)(g)=\phi(g\cdot\ang{-w})$ for every $g\in
C(O_\p^\anti,F_\p)$. We take
\[\bM(f)=\tau_w(\beta^*f).\]
The formula \eqref{eq:family} follows from Lemma \ref{le:family_local} and
Lemma \ref{le:theta}.

Property (1) follows from Proposition \ref{pr:lt_action} (2). Property (3)
and (4) follow from Lemma \ref{le:family_local} and Lemma
\ref{le:heartsuit_local}, respectively. For property (2), we only need to
show that \eqref{eq:universal_form} holds for $\bM=\bM(f)$ and $t\in
O_{E_\p,m}^\times$. In fact, since $f$ is fixed by $O_{E_\p,m}^\times$,
$\bM(f)=\bM(t^*f)$ which equals $[t]\cdot t^*\bM(f)$ by Proposition
\ref{pr:lt_action} (2) and Remark \ref{re:lt_action}.
\end{proof}

\subsection{Comparison of differential operators at archimedean places}
\label{ss:comparison_differential}

We equip $\dB$ with an $E$-embedding (Definition \ref{de:e_embedding}). In
particular, we have the CM-subscheme $Y^\pm$ of $X$ (Definition
\ref{de:cm_subscheme}). The projection $X\to X(\pm\infty)$ restricts to an
isomorphic from $Y^\pm$ to its image. Thus we may regard $Y^\pm$ as a closed
subscheme of $X(\pm\infty)$. Via the transition isomorphism
\eqref{eq:infinity}, we obtain a closed subscheme $Y^\pm(\infty)=\Upsilon_\pm
Y^\pm$ of $\cX(\infty)\otimes_{O_\p^\nr}F_\p^\ab$, which is in fact a closed
subscheme of $\cX(\infty)\otimes_{O_\p^\nr}F_\p^\nr$. Finally for $m\in\dN$,
define $Y^\pm(m)$ to be the image of $Y^\pm(\infty)$ in
$\cX(m)\otimes_{O_\p^\nr}F_\p^\nr$.

Let $S$ be a complex scheme locally of finite type. We denote by $\breve{S}$
the underlying real analytic space with the complex conjugation automorphism
$\fc_S\colon\breve{S}\to\breve{S}$. In what follows, we will sometimes deal
with a complex scheme $S$ that is of the form $\varprojlim_IS_i$ where $I$ is
a filtered partially ordered set and each $S_i$ is a smooth complex scheme,
with a sheaf $\sF$ that is the restriction of a quasi-coherent sheaf $\sF_0$
on some $S_0$. Then we will write $\breve{S}=\{\breve{S}_i\}_{i\in I}$ for
the projective system of the underlying real analytic spaces together with
the complex conjugation $\fc_S$, and $\breve\sF=\{\breve\sF_i\}_{i\geq 0}$
the projective system of real analytification of the restricted sheaf $\sF_i$
for $i\geq 0$. Moreover, we define
\[\rH^0(\breve{S},\breve\sF)\colonequals\varinjlim_{i\geq 0}\rH^0(\breve{S}_i,\breve\sF_i).\]

For $\iota\colon\dC_p\xrightarrow{\sim}\dC$, put
$X_\iota=X\otimes_{F,\iota}\dC$ and
$\fc_\iota\colon\breve{X}_\iota\to\breve{X}_\iota$ the complex conjugation.
Denote by $(\cL_\iota,\nabla_\iota)$ the restriction of the pair
$(\cL,\nabla)$ in Lemma \ref{le:algebraizable} along $\pi_\iota\colon
X_\iota\to\cX\otimes_{O_\p^\nr,\iota}\dC$. Denote the restriction of the
sequence \eqref{eq:hodge_sequence} along $\pi_\iota$ by
\begin{align}\label{eq:hodge_iota}
\xymatrix{
0 \ar[r] & \ul\omega^\bullet_\iota \ar[r] & \cL_\iota \ar[r] & \ul\omega^{\circ\vee}_\iota \ar[r] & 0. }
\end{align}
The following lemma will be proved in \Sec\ref{ss:proof_claims}.

\begin{lem}\label{le:hodge_filtration}
The sequence \eqref{eq:hodge_iota} coincides with the Hodge filtration on
$\cL_\iota$.
\end{lem}

Let $\ul\omega^\circ_\iota$ be the restriction of $\ul\omega^\circ$ along
$\pi_\iota$, which is the dual sheaf of $\ul\omega^{\circ\vee}_\iota$. Then
we have the Kodaira--Spencer isomorphism
$\KS\colon\ul\omega^\bullet_\iota\otimes\ul\omega^\circ_\iota\xrightarrow{\sim}\Omega^1_{X_\iota}$.
The Hodge decomposition
\begin{align}\label{eq:hodge_decomposition}
\cL_\iota=\ul{\breve\omega}^\bullet_\iota\oplus\fc_\iota^*\ul{\breve\omega}^\bullet_\iota
\end{align}
on $X_\iota$ induces a map
\[\theta^{[w]}_\iota\colon\breve\cL_\iota^{[w]}\to(\ul{\breve\omega}^\bullet_\iota)^{\otimes w}
\otimes(\ul{\breve\omega}^\circ_\iota)^{\otimes w}
\xrightarrow{\KS}(\breve\Omega^1_{X_\iota})^{\otimes w}\] for all $w\in\dN$.
Similar to Definition \ref{de:theta_ordinary}, define the
\emph{Shimura--Maass operator} to be\index{Shimura--Maass operator,
$\Theta_\iota$}
\[\Theta^{[w]}_\iota\colon(\breve\Omega^1_{X_\iota})^{\otimes w}\xrightarrow{\eqref{eq:theta}}
\breve\cL^{[w]}\otimes\breve\Omega^1_{X_\iota}\xrightarrow{\theta^{[w]}_\iota}(\breve\Omega^1_{X_\iota})^{\otimes
w+1}.\] For $k\in\dN$, define the \emph{Shimura--Maass operator of degree $k$} to be
\[\Theta^{[w,k]}_\iota=\Theta^{[w+k-1]}_\iota\circ\cdots\circ\Theta^{[w]}_\iota\colon
(\breve\Omega^1_{X_\iota})^{\otimes w}\to(\breve\Omega^1_{X_\iota})^{\otimes
w+k}.\] As for $\Theta_\ord$, we will suppress $w$ from notation and write
$\Theta_\iota$ (resp.\ $\Theta_\iota^k$) for $\Theta_\iota^{[w]}$ (resp.
$\Theta_\iota^{[w,k]}$). In particular, we have
\[\Theta_\iota\colon\rH^0(\breve{X}_\iota,(\breve\Omega^1_{X_\iota})^{\otimes w})\to
\rH^0(\breve{X}_\iota,(\breve\Omega^1_{X_\iota})^{\otimes w+1}).\]

Put
\begin{align*}
X(m)_\iota&=X(m)\otimes_{F_\p^\nr,\iota}\dC,\quad m\in\dZ\cup\{\pm\infty\},\\
Y^\pm(m)_\iota&=Y^\pm(m)\otimes_{F_\p^\nr,\iota}\dC,\quad m\in\dN\cup\{\infty\}.
\end{align*}
Then $Y^\pm(m)_\iota$ is a closed subscheme of $X(\pm m)_\iota$ via the
transition isomorphism \eqref{eq:infinity}. Denote by $\cY^\pm(m)$ the
Zariski closure of $Y^\pm(m)$ in $\cX(m)$.

\begin{lem}
For $m\in\dN\cup\{\infty\}$, every morphism $\Spec F_\p^\nr\to\cY^\pm(m)$
over $\Spec O_\p^\nr$ extends uniquely to a section $\Spec
O_\p^\nr\to\cY^\pm(m)$.
\end{lem}

\begin{proof}
It suffices to show for $m=\infty$. Let $x^\pm\colon\Spec
F_\p^\nr\to\cY^\pm(\infty)$ be a morphism. It induces a unique morphism
$\Spec O_\p^\nr\to\cX$ and we will regard $x^\pm$ as the latter one. Since
$x^\pm$ is fixed by $E^\times$, there are actions of $E^\times\cap
O_{E_\p}^\times$ and hence $O_{E_\p}$ on the $O_\p$-divisible group
$\cG_{x^\pm}$. Therefore, the reduction of $\cG_{x^\pm}$ is ordinary. To
conclude, we only need to show that $x^\pm$ lifts the canonical subgroup of
$\cG_{x^\pm}$. This follows from the fact that $E^\times$ acts on the tangent
space of $x^\pm$ via the character $t\mapsto(t/t^c)^{\pm1}$.
\end{proof}

For $m\in\dN\cup\{\infty\}$, denote by $\fY^\pm(m)$ the formal completion of
$\cY^\pm(m)$ along its special fiber, which is a closed (affine) formal
subscheme of $\fX(m)$ by the above lemma.

 et $F_\p^\ab\subset K\subset\dC_p$ be a complete intermediate field. Let
$f$ be an element in $\rH^0(X(m),(\Omega^1_{X(m)})^{\otimes w})\otimes_FK$
with $m\in\dZ\cup\{\pm\infty\}$ and $w\in\dN$. Then by the transition
isomorphism \eqref{eq:infinity} and restriction to ordinary locus, we have an
element
\begin{align}\label{eq:restriction_ordinary}
f_\ord=\Upsilon_{\pm*}f\in\sM^w_\flat(m,K).
\end{align}
On the other hand, we have the projection map $X_\iota\to X(m)_\iota$ for
which $\Theta_\iota$ descends. Thus, $f$ induces another element
\begin{align}\label{eq:restriction}
f_\iota\in\rH^0(X(m)_\iota,(\Omega^1_{X(m)_\iota})^{\otimes w}).
\end{align}
We will freely regard $f_\iota$ as an element in
$\rH^0(X_\iota,(\Omega^1_{X_\iota})^{\otimes w})$ according to the context.
The following lemma shows that the Atkin--Serre operator and the
Shimura--Maass operator coincide on CM points.

\begin{lem}\label{le:theta_comparison}
Let the situation be as above. We have for $k\in\dN$,
\[\iota(\Theta^k_\ord f_\ord)\res_{\fY^\pm(m)}=(\Theta^k_\iota f_\iota)\res_{Y^\pm(m)_{\iota}}\]
as functions on $Y^\pm(m)_{\iota}$.
\end{lem}

\begin{proof}
Generally, once we restrict to stalks, we can not apply differential
operators anymore. Therefore, we need variant definitions of
$\Theta^{[w,k]}_\ord$ and $\Theta^{[w,k]}_\iota$ (Here, we retrieve the
original notation in order to be clear). In fact, since
$\nabla\cL^\circ\subset\cL^\circ\otimes\Omega^1_\fX$ by Lemma
\ref{le:unit_root}, $\Theta^{[w,k]}_\ord$ is equal to the composition of the
following sequence of maps
\begin{multline*}
(\Omega^1_{\fX(n)})^{\otimes w}\xrightarrow{\KS^{-1}}
(\ul\omega^\bullet)^{\otimes w}\otimes(\ul\omega^\circ)^{\otimes w}
\to\cL^{[w]}\xrightarrow{\Theta^{[w]}}\cL^{[w]}\otimes\Omega^1_{\fX(n)}
\xrightarrow{\KS^{-1}}\cL^{[w]}\otimes(\ul\omega^\bullet\otimes\ul\omega^\circ)\\
\to\cL^{[w+1]}\xrightarrow{\Theta^{[w+1]}}\cL^{[w+1]}\otimes\Omega^1_{\fX(n)}
\to\cdots\to\cL^{[w+k]}\xrightarrow{\theta_\ord^{[w+k]}}(\Omega^1_{\fX(n)})^{\otimes
w+k}.
\end{multline*}
Similarly, since
$\nabla_\iota(\fc_\iota^*\ul{\breve\omega}^\bullet_\iota)\subset(\fc_\iota^*\ul{\breve\omega}^\bullet_\iota)\otimes\breve\Omega^1_{X_\iota}$,
we have a similar description of $\Theta^{[w,k]}_\iota$. Therefore, to prove
the lemma, we only need to show that the splitting
\eqref{eq:hodge_decomposition} coincides with the restriction of the
splitting $\cL=\ul\omega^\bullet\oplus\cL^\circ$ on $Y^\pm_\iota$. Pick up
any point $y\in Y^\pm(m)_\iota(\dC)$. We have an action of $E^\times$ on both
the splitting
$\ul{\breve\omega}^\bullet_\iota\res_y\oplus\fc_\iota^*\ul{\breve\omega}^\bullet_\iota\res_y$
and $\ul\omega^\bullet\res_y\oplus\cL^\circ\res_y$. By Lemma
\ref{le:hodge_filtration}, $\ul{\breve\omega}^\bullet_\iota\res_y$ and
$\ul\omega^\bullet\res_y$ coincide, which is one complex eigen-line of
$E^\times$. We have that $\fc_\iota^*\ul{\breve\omega}^\bullet_\iota\res_y$
and $\cL^\circ\res_y$ must also coincide, which is the other complex
eigen-line.
\end{proof}

\begin{definition}[$\iota$-nearby data]\label{de:nearby_data}
For $\iota\colon\dC_p\xrightarrow{\sim}\dC$, an \emph{$\iota$-nearby data}
for $\dB$ consists of
\begin{itemize}
  \item a quaternion algebra $B(\iota)$ over $F$ such that $B(\iota)_v$
      is definite for archimedean places $v$ other than $\iota\res_F$,

  \item an isomorphism $B(\iota)_v\simeq\dB_v$ for every finite place $v$
      other than $\p$,

  \item an isomorphism
      $B(\iota)_\iota=B(\iota)\otimes_{F,\iota}\dR\simeq\Mat_2(\dR)$,

  \item a uniformization
      \[
      X_\iota(\dC)\simeq B(\iota)^\times\backslash\cH\times\dB^{\infty\times}/F^\times_\cl,
      \]
      where $\cH=\dC\setminus\dR$ denotes the union of Poincar\'{e} upper
      and lower half-planes,

  \item an embedding $\sfe(\iota)\colon E\hookrightarrow B(\iota)$ of
      $F$-algebras such that $\sfe(\iota)_v$ coincides with $\sfe_v$
      under the isomorphism $B(\iota)_v\simeq\dB_v$ for every finite
      place $v$ other than $\p$, and $\cH^{E^\times}=\{\pm i\}$.
\end{itemize}\index{$\iota$-nearby data}
\end{definition}

We now choose an $\iota$-nearby data for $\dB$. For every $w\in\dZ$, denote
by $\sA^{(2w)}(B(\iota)^\times)$ (resp.\
$\sA^{(2w)}_{\cusp}(B(\iota)^\times)$) the space of real analytic (resp.\ and
cuspidal) automorphic forms on $B(\iota)^\times(\dA)$ of weight $2w$ at
$\iota\res_F$ and trivial at other archimedean places. There is a natural
$\dB^{\infty\times}$-equivariant map
\begin{align}\label{eq:automorphic_form}
\phi_\iota\colon\rH^0(\breve{X}_\iota,(\breve\Omega^1_{X_\iota})^{\otimes w})\to\sA^{(2w)}(B(\iota)^\times)
\end{align}
such that for $g_\iota\in B(\iota)^\times_\iota=\GL_2(\dR)$,
\[\phi_\iota(f)([g_\iota,1])j(g_\iota,i)^w=f(g_\iota(i))\otimes\rd z^{\otimes -w},\]
where $j(g_\iota,i)=(\det g_\iota)^{-1}\cdot(ci+d)^2$ is the square of the
usual $j$-factor. We denote by
$\rH^0_\cusp(\breve{X}_\iota,(\breve\Omega^1_{X_\iota})^{\otimes
w})\subset\rH^0(\breve{X}_\iota,(\breve\Omega^1_{X_\iota})^{\otimes w})$ the
inverse image of $\sA^{(2w)}_\cusp(B(\iota)^\times)$ under $\phi_\iota$.

We may recover the pair $(\cL_\iota,\nabla_\iota)$ in the following way.
Denote by $\rL_\iota$ the $\dC$-local system on $X_\iota$ defined by the
quotient $B(\iota)^\times\backslash\dC^{\oplus 2}\times\cH\times
\dB^{\infty\times}/F^\times_\cl$ where the action of $B(\iota)^\times$ is
given by
\[\gamma[(a_1,a_2)^t,z,g]=[((a_1,a_2)\iota(\gamma)^{-1})^t,\iota(\gamma)(z),\gamma^\infty g].\]
Then $\rL_\iota$ is canonically isomorphic to the restriction of
$\rL\otimes_{O_\p,\iota}\dC$ along the natural morphism $\pi_\iota$, where
$\rL$ is the $O_\p$-local system on $\cX$ defined in Remark
\ref{re:local_system}. Thus, $\cL_\iota=\sO_{X_\iota}\otimes_{\dC}\rL_\iota$
and $\nabla_\iota\colon\cL_\iota\to\cL_\iota\otimes\Omega^1_{X_\iota}$ is the
induced connection.

The following lemma shows our definition of Shimura--Maass operators coincide
with the classical one.

\begin{lem}\label{le:theta_iota}
For every $f\in\rH^0(\breve{X}_\iota,(\breve\Omega^1_{X_\iota})^{\otimes w})$
with some $w\in\dN$, we have
\[\Theta_\iota f\otimes\rd z^{\otimes -w-1}=\(\frac{\partial}{\partial z}+\frac{2w}{z-\ol{z}}\)f\otimes\rd z^{\otimes-w}.\]
\end{lem}

\begin{proof}
We may pass to the universal cover $\cH\times\dB^{\infty\times}/F^\times_\cl$
and suppress the part $\dB^{\infty\times}/F^\times_\cl$ in what follows. Over
$\cH$, the sheaf $\cL_\iota$ is trivialized as $\dC^{\oplus2}$ and the
sub-sheaf $\ul\omega^\bullet_\iota$ is generated by the section
$\omega^\bullet_\iota$ whose value at $z$ is $(z,1)^t$. Dually, the sheaf
$\cL_\iota^\vee$ is trivialized as two-dimensional complex row vectors and
the sub-sheaf $\ul\omega^\circ_\iota$ is generated by the section
$\omega^\circ_\iota$ whose value at $z$ is $(1,-z)$. Then
$\KS(\omega^\bullet_\iota\otimes\omega^\circ_\iota)=\rd z$.

It is easy to see that
\[\Theta_\iota\((\omega^\bullet_\iota)^{\otimes w}\otimes(\omega^\circ_\iota)^{\otimes w}\)
=\frac{2w}{z-\ol{z}}\((\omega^\bullet_\iota)^{\otimes
w}\otimes(\omega^\circ_\iota)^{\otimes w}\)\otimes\rd z\] since
$\fc_\iota^*\ul\omega^\bullet_\iota$ (resp.\
$\fc_\iota^*\ul\omega^\circ_\iota$) is generated by the section
$(\ol{z},1)^t$ (resp.\ $(1,-\ol{z})$). The lemma follows.
\end{proof}

Denote by $\Delta_{\pm,\iota}$ the element
\begin{align*}
\Delta_\pm\colonequals\frac{1}{4i}\left(
               \begin{array}{cc}
                 1 & \pm i \\
                 \pm i & 1 \\
               \end{array}
             \right)
\end{align*}
in $\f{gl}_{2,\dC}=\Mat_2(\dC)=\Lie_{\dC}(B(\iota)\otimes_{F,\iota}\dC)$, and
$\Delta_{\pm,\iota}^k=\Delta_{\pm,\iota}\circ\cdots\circ\Delta_{\pm,\iota}$
the $k$-fold composition.

\begin{lem}\label{le:shimura_maass}
For every
$f\in\rH^0_\cusp(\breve{X}_\iota,(\breve\Omega^1_{X_\iota})^{\otimes w})$ and
$k\in\dN$, we have
\[\phi_\iota(\Theta_\iota^k f)=\Delta_{+,\iota}^k\phi_\iota(f).\]
\end{lem}

\begin{proof}
This follows from Lemma \ref{le:theta_iota}, \cite{Bum97} p.130, p.143, and
Proposition 2.2.5 on p.155.
\end{proof}

\subsection{Proof of claims}
\label{ss:proof_claims}

In this section, we prove six claims (\ref{le:algebraizable},
\ref{pr:ks_isomorphism}, \ref{le:unit_root}, \ref{le:deformation},
\ref{pr:lt_action}, and \ref{le:hodge_filtration}) left in previous sections.
All these claims are natural extensions from their versions on the modular
curve. The reader may skip this section for the first reading.

Our strategy is to use the unitary Shimura curves considered by Carayol in
\cite{Car86}. Thus we will fix an isomorphism
$\iota\colon\dC_p\xrightarrow{\sim}\dC$. In particular, $F_\p^\nr$ is a
subfield of $\dC$. We also fix an $\iota$-nearby data for $\dB$ (Definition
\ref{de:nearby_data}) and put $B=B(\iota)$ for short.

Note that when $F=\dQ$ there is no need to change the Shimura curve. In order
to unify the argument, we will choose to do so in this case as well. We will
also assume that we are not in the case of classical modular curves where all
these statements are well-known.

Fix an element $\lambda\in\dC$ such that $\IM\lambda>0$, $-\lambda^2\in\dN$,
$p\neq\lambda^2$, $p$ splits in $\dQ(\lambda)\subset\dC$, and $\dQ(\lambda)$
is not contained in $E$. Put $\U{F}=F(\lambda)$ and $\U{E}=E(\lambda)$ both
as subfields $\dC$. We identify the completion of $\U{F}$ and $\U{E}$ inside
$\dC\simeq\dC_p$ with $F_\p$. In \cite{Car86}*{\Sec 2} (see also
\cite{Kas04}*{\Sec 2}), a reductive group $\U{G}$ over $\dQ$ is defined such
that
\[\U{G}(\dQ_p)=\dQ_p^\times\times\GL_2(F_\p)\times(\dB_{\p_2}^\times\times\cdots\times\dB_{\p_m}^\times),\]
where $\p_2,\dots,\p_m$ are primes of $F$ over $p$ over than $\p$. Let
$\U{G}^\p=\prod'_{q\neq
p}\U{G}(\dQ_q)\times(\dB_{\p_2}^\times\times\cdots\times\dB_{\p_m}^\times)$
and $\U{\Gamma}$ the set of all (sufficiently small) compact open subgroups
$\U{U}^\p$ of $\U{G}^\p$. Then for each $\U{U}^\p\in\U{\Gamma}$, there is a
unitary Shimura curve $\U{X}_{\U{U}^\p}$ over $\Spec F_\p$ of the level
structure $\dZ_p^\times\times\GL_2(O_\p^\times)\times\U{U}^\p$, which is
smooth and projective. It has a canonical smooth model $\U{\cX}_{\U{U}^\p}$
over $\Spec O_\p^\nr$ defined via a moduli problem \cite{Car86}*{\Sec 6}. In
particular, there is a universal abelian variety
$\pi\colon\cA_{\U{U}^\p}\to\U{\cX}_{\U{U}^\p}$ with a specific $p$-divisible
subgroup $\U{\cG}_{\U{U}^\p}\subset\cA_{\U{U}^\p}[p^\infty]$ that is
naturally an $O_\p$-divisible group of dimension $1$ and height $2$. Denote
$\U{\cX}(0)_{\U{U}^\p}$ the (dense) open subscheme of $\U{\cX}_{\U{U}^\p}$
with all points on the special fiber where $\U{\cG}$ is supersingular
removed. For $n\in\dN$, define $\U{\cX}(n)_{\U{U}^\p}$ to be the functor
classifying $O_\p$-equivariant extensions
\[\xymatrix{0\ar[r]& \LT[\p^n] \ar[r]& \U{\cG}[\p^n] \ar[r]& \p^{-n}/O_\p \ar[r]& 0}\]
of $\U{\cG}$ over $\U{\cX}(0)_{\U{U}^\p}$, which is a scheme \'{e}tale over
$\U{\cX}(0)_{\U{U}^\p}$.

Then the construction of Carayol amounts to saying that for every
sufficiently small $U^\p\in\Gamma$ and a connected component $\cX'_{U^\p}$ of
$\cX_{U^\p}$, there exists a $\U{U}^\p\in\U{\Gamma}$ such that
\begin{itemize}
  \item
      $\cX(n)'_{U^\p}\colonequals\cX'_{U^\p}\times_{\cX_{U^\p}}\cX(n)_{U^\p}$
      is isomorphic to the identity connected component of
      $\U{\cX}(n)_{\U{U}^\p}$ for $n\in\dN$;
  \item under the above isomorphism $\cG_{U^\p}\res_{\cX(n)'_{U^\p}}$ is
      isomorphic to the restriction of $\U{\cG}_{\U{U}^\p}$ to
      $\cX(n)'_{U^\p}$.
\end{itemize}

In what follows we may and will fix a sufficiently small $U^\p\in\Gamma$, a
connected component $\cX'_{U^\p}$ of $\cX_{U^\p}$, and a corresponding
$\U{U}^\p\in\U{\Gamma}$. The tame levels $U^\p$ and $\U{U}^\p$ will be
suppressed from the notation.

Consider the Hodge sequence
\begin{align*}
\xymatrix{
0\ar[r]&  \pi_*\Omega^1_{\cA/\U{\cX}} \ar[r]& \sH^1_{\DR}(\cA/\U{\cX})
\ar[r]& \rR^1\pi_*\sO_{\cA} \ar[r]& 0.
}
\end{align*}
It has a direct summand
\begin{align}\label{eq:hodge_unitary}
\xymatrix{
0\ar[r]&  (\pi_*\Omega^1_{\cA/\U{\cX}})^{2,1}_1 \ar[r]& \sH^1_{\DR}(\cA/\U{\cX})^{2,1}_1
\ar[r]& (\rR^1\pi_*\sO_{\cA})^{2,1}_1 \ar[r]& 0
}
\end{align}
which is $O_\p$-equivariant (see \cite{Car86}*{4.2} for the meaning of
$(-)^{2,1}_1$), in which the three sheaves are locally constant of rank $1$,
$2h$, and $2h-1$, respectively, where $h=[F_\p\colon \dQ_p]$.

If $M$ is a projective $O_\p$-module or a locally free sheaf on an
$O_\p$-scheme, equipped with an $O_\p$-action $O_\p\to\End M$, then we denote
by $M^{O_\p}$ the maximal submodule or subsheaf on which $O_\p$ acts via the
structure homomorphism. Then by construction, if we apply the functor
$(-)^{O_\p}$ and take formal completion (after restriction to $\cX$) to
\eqref{eq:hodge_unitary}, we will recover the the exact sequence
\eqref{eq:hodge_sequence} in Lemma \ref{le:algebraizable}. For later use, we
denote the sequence \eqref{eq:hodge_unitary} after $(-)^{O_\p}$ by
\begin{align*}
\xymatrix{
0\ar[r]&  \U{\ul\omega}^\bullet \ar[r]& \U{\cL} \ar[r]& \U{\ul\omega}^{\circ\vee} \ar[r]& 0.
}
\end{align*}

Moreover, we have the Gauss--Manin connection
\[\U{\nabla}_p\colon \sH^1_{\DR}(\cA/\U{\cX})\to\sH^1_{\DR}(\cA/\U{\cX})\otimes\Omega^1_{\U{\cX}}.\]
By the functoriality of the Gauss--Manin connection, we have an induced
connection
\[\U{\nabla}\colon \U{\cL}\to\U{\cL}\otimes\Omega^1_{\U{\cX}},\]
whose formal completion (after restriction to $\cX'$) coincides with
\eqref{eq:gm_formal}. Therefore, Lemma \ref{le:algebraizable} is proved.

Denote by
$\U{\KS}\colon\U{\ul\omega}^\bullet\otimes\U{\ul\omega}^\circ\to\Omega^1_{\U{\cX}}$
the induced Kodaira--Spencer map, where $\U{\ul\omega}^\circ$ is the dual
sheaf of $\U{\ul\omega}^{\circ\vee}$. Proposition \ref{pr:ks_isomorphism}
follows from the following analogous one for $\U{\cX}$.

\begin{lem}
The Kodaira--Spencer map
$\U{\KS}\colon\U{\ul\omega}^\bullet\otimes\U{\ul\omega}^\circ\to\Omega^1_{\U{\cX}}$
is an isomorphism.
\end{lem}

\begin{proof}
The proof is similar to \cite{DT94}*{Lemma 7}. Denote by $\cA^\vee$ the dual
abelian variety of $\cA$. Then $\U{\ul\omega}^{\circ\vee}$ is canonically
isomorphic to $(\ul\Lie(\cA^\vee/\U{\cX})^{2,1}_1)^{O_\p}$. We only need to
show that for every closed point $t\colon\Spec k(t)\to\U{\cX}$, the induced
map
\begin{align}\label{eq:ks_unitary}
\U{\ul\omega}^\bullet\otimes k(t)\to
(\ul\Lie(\cA^\vee/\U{\cX})^{2,1}_1)^{O_\p}\otimes\Omega^1_{\U{\cX}}\otimes k(t)
\end{align}
is surjective, where $\ul\Lie$ denotes the sheaf of tangent vectors.

Let $A/\Spec k(t)$ be the abelian variety classified by $t$. Put $T=\Spec
k(t)[\varepsilon]/(\varepsilon^2)$. The lifts $A_\phi$ of $A$ (with other PEL
structures) to $T$ correspond to homomorphisms
\[\phi\colon t^*\U{\ul\omega}^\bullet\to(\ul\Lie(\cA^\vee/\U{\cX})^{2,1}_1)^{O_\p}\otimes k(t).\]
Since both sides are $k(t)$-vector spaces of dimension $1$, we may choose a
$\phi$ that is surjective. Let $t_\phi\colon T\to\U{\cX}$ be the morphism
that classifies $A_\phi/T$. Compose the isomorphism
$t_\phi^*\U{\ul\omega}^\bullet\otimes k(t)\to t^*\U{\ul\omega}^\bullet$ and
the surjective map $\phi$. By the isomorphism
\[(\ul\Lie(\cA^\vee/\U{\cX})^{2,1}_1)^{O_\p}\otimes k(t)\simeq
t_\phi^*(\ul\Lie(\cA^\vee/\U{\cX})^{2,1}_1)^{O_\p}\otimes\Omega^1_{T/k(t)}\otimes
k(t),\] we obtain a surjective map
\[t_\phi^*\U{\ul\omega}^\bullet\otimes k(t)\to t_\phi^*(\ul\Lie(\cA^\vee/\U{\cX})^{2,1}_1)^{O_\p}\otimes\Omega^1_{T/k(t)}\otimes k(t),\]
which is the pullback of \eqref{eq:ks_unitary} under $t_\phi$. Therefore,
\eqref{eq:ks_unitary} is surjective.
\end{proof}

For $n\in\dN$, denote by $\U{\fX}(n)$ the formal completion along its special
fiber, which is equipped with an $O_\p$-divisible group $\U{\fG}$ induced
from $\U{\cG}$. Let $\U{\fG}_\can\subset\U{\fG}$ be the canonical subgroup.
We have a morphism $\U{\Phi}\colon\U{\fX}\to\U{\fX}$ defined by ``dividing
$\U{\fG}_\can[\p]$'' which lifts the relative Frobenius on the special fiber.
In fact, the induced map on the coordinate ring is simply the operator
$\r{Frob}$ defined in \cite{Kas04}*{Definition 11.1}.

\begin{proof}[Proof of Lemma \ref{le:unit_root}]
We only need to prove the same statement for $\U{\cX}$. Then we define
$\U{\cL}^\circ$ to be the sub-sheaf of $\U{\cL}$ where $\U{\Phi}$ acts by a
$p$-adic unit. To show that it glues to a formal quasi-coherent sheaf, we may
adopt the proof of \cite{Kat73}*{Theorem 4.1} in the case where $\dZ_p$ is
replaced by $O_\p$ and $p$ is replaced by a uniformizer $\varpi$ of $F$. The
assumptions are satisfied because the Newton polygon of the underlying
$p$-divisible group of $\U{\fG}\res_x$ for any $x\in\U{\fX}(\kappa)$ is the
one starting with $(0,0)$, ending with $(2h,1)$ and having the unique
breaking point at $(h,0)$. The similar proof also confirms (1) and (2). For
(3), we use the local calculation in Lemma \ref{le:rationality}.
\end{proof}

\begin{proof}[Proof of Lemma \ref{le:deformation} and Lemma \ref{le:hodge_filtration}]
We only need to prove the similar statements for $\U{\fX}$, which follow from
the moduli interpretation of $\U{\fX}$ and the Serre--Tate theorem for Lemma
\ref{le:deformation}, and the existence of the universal abelian variety
$\cA$ for Lemma \ref{le:hodge_filtration}, respectively.
\end{proof}

\begin{proof}[Proof of Proposition \ref{pr:lt_action}]
We may similarly define $\U{\fX}(\infty)$ over $\Spf O_\p^\nr$ and only need
to construct the morphism $\U{\beta}\colon\LT\times_{\Spf
O_\p^\nr}\U{\fX}(\infty)\to\U{\fX}(\infty)$ with similar properties, since
the action of $\LT$ is supposed to preserve the special fiber. We use moduli
interpretation. For a scheme $S$ over $\Spec O_\p^\nr$ where $p$ is locally
nilpotent, $\U{\fX}(\infty)(S)$ is the set of isomorphism classes of
quintuples $(A,\iota,\theta,k^\p,\kappa_\p)$, where $(A,\iota,\theta,k^\p)$
is the same data in \cite{Car86}*{\Sec 5.2} but $k^\p$ is an isomorphism
instead of a class, and $\kappa_\p$ is an exact sequence
\[\xymatrix{
0\ar[r]& \LT \ar[r]& (A_{p^\infty})^{2,1}_1 \ar[r]& F_\p/O_\p \ar[r]& 0. }\]
On the other hand, $\LT(S)$ is the set of isomorphism classes of $(G,k_G)$
where $k_G$ is an exact sequence
\[\xymatrix{
0\ar[r]& \LT \ar[r]& G \ar[r]& F_\p/O_\p \ar[r]& 0. }\] Using the group
structure on $\LT$, we may add the above two exact sequences to a new one
$\alpha(k_\p,k_G)$ as
\[\xymatrix{
0\ar[r]& \LT \ar[r]& \alpha((A_{p^\infty})^{2,1}_1,G) \ar[r]& F_\p/O_\p
\ar[r]& 0. }\] By the theorem of Serre--Tate and the fact that \'{e}tale
level structures are determined on the special fiber, we associate
canonically a quintuple $(A',\iota',\theta',k'^\p,\kappa'_\p)$ with
$\kappa'_\p=\alpha(k_\p,k_G)$. This defines $\U{\beta}$. All these properties
follow from the above construction and Theorem \ref{th:serre_tate}.
\end{proof}

\section{Heegner cycles on abelian varieties}

In this chapter, we reformulate our main theorems about $p$-adic
$L$-functions and $p$-adic Waldspurger formula in terms of Heegner cycles on
abelian varieties. We start from recalling some background about
representations of incoherent algebras and abelian varieties of $\GL(2)$-type
in \Sec\ref{ss:representations_incoherent}. In
\Sec\ref{ss:l_function_abelian}, we state the main theorem about $p$-adic
$L$-functions in terms of Heegner cycles and show that it implies Theorem
\ref{th:l_function_maass}. In \Sec\ref{ss:heegner_waldspurger}, we state the
main theorem about $p$-adic Waldspurger formula in terms of Heegner cycles
and show that it implies Theorem \ref{th:waldspurger_maass}.

\subsection{Representations of incoherent algebras}
\label{ss:representations_incoherent}

We recall some materials from \cite{YZZ13}*{\Sec 3.2}. Let
$\iota_1,\dots,\iota_g$ be all archimedean places of $F$. Let $\dB$ be a
totally definite incoherent quaternion algebra over $\dA$, to which there is
an associated projective system of Shimura curves $\{X_U\}_U$. Put
$X=\varprojlim_U X_U$. We recall the following definition in
\cite{YZZ13}*{\Sec 3.2.2}.

\begin{definition}
Let $L$ be a field admitting embeddings into $\dC$. Denote by
$\cA(\dB^\times,L)$ the set of isomorphism classes of irreducible
representations $\Pi$ of $\dB^{\infty\times}$ over $L$ such that for some and
hence all embeddings $L\hookrightarrow\dC$, the Jacquet--Langlands transfer
of $\Pi\otimes_L\dC$ to $\GL_2(\dA^\infty)$ is a finite direct sum of (finite
components of) irreducible cuspidal automorphic representations $\GL_2(\dA)$
of parallel weight $2$.
\end{definition}

\begin{remark}\label{re:l_value}
When $L$ is algebraically closed, we have for every finite place $v$ of $F$
the local $L$-function $L(s,\Pi_v)$, local $\epsilon$-factor
$\epsilon(1/2,\psi,\Pi_v)$, local adjoint $L$-function $L(s,\Pi_v,\Ad)$,
local Rankin--Selberg $L$-function $L(s,\Pi_v,\chi_v)$ and $\epsilon$-factor
$\epsilon(1/2,\Pi_v,\chi_v)$ for a locally constant character $\chi_v\colon
E_v^\times\to L^\times$. When $L=\dC$, we have the global versions, which are
products of local ones over all \emph{finite} places of $F$.
\end{remark}

We say an abelian variety $A$ can be \emph{parameterized by $\dB$} if there
is a non-constant morphism from $X=X(\dB)$ to $A$. Denote by $\AV^0(\dB)$ the
set of simple abelian varieties over $F$ that can be parameterized by $\dB$
up to isogeny, which is stable under duality. From $A\in\AV^0(\dB)$, we
obtain a rational representation $\Pi_A$ of $\dB^{\infty\times}$ which is an
element in $\cA(\dB^\times,\dQ)$. The assignment $A\mapsto\Pi_A$ induces a
bijection between $\AV^0(\dB)$ and $\Pi\in\cA(\d B^\times,\dQ)$.

\begin{notation}\label{no:modular_parametrization}
Recall from \cite{YZZ13}*{\Sec 3.2.3} the following notation
\[\Pi_A=\varinjlim_U\Hom_{\xi_U}(X_U^*,A),\] where
\begin{itemize}
  \item the colimit is taken over all compact open subgroups $U$ of
      $\dB^{\infty\times}$;
  \item $X_U^*$ is simply $X_U$ (resp.\ $X_U$ plus cusps) if $X_U$ is
      proper (resp.\ not proper which happens exactly when it is the
      classical modular curve);
  \item $\xi_U$ is the normalized Hodge class on $X_U^*$
      \cite{YZZ13}*{\Sec 3.1.3}; and
  \item $\Hom_{\xi_U}(X_U^*,A)$ denotes the $\dQ$-vector space of
      \emph{modular parameterizations}, that is, (quasi-)morphisms from
      $X_U^*$ to $A$ that send $\xi_U$ to torsion.
\end{itemize}
\end{notation}

If we denote by $J_U$ the Jacobian of $X_U^*$, then $\Hom_{\xi_U}(X_U^*,A)$
is canonically identified with $\Hom^0(J_U,A)$. Moreover,
$M_A\colonequals\End^0(A)$ is a field of degree equal to the dimension of $A$
and $M_A$ acts on the representation $\Pi$. Denote by $A^\vee$ the dual
abelian variety (up to isogeny) of $A$ and we have $\Pi_{A^\vee}$ similarly.
Then the rosati involution induces a canonical isomorphism $M_{A^\vee}\simeq
M_A$.

\begin{definition}[Canonical pairing, \cite{YZZ13}*{\Sec 3.2.4}]\label{de:pairing_canonical}
We have a canonical pairing
\[(\;,\;)_A\colon \Pi_A\times\Pi_{A^\vee}\to M_A\]
induced by maps
\[(\;,\;)_U\colon \Hom^0(J_U,A)\times \Hom^0(J_U,A^\vee)\to M_A\] defined
by $(f_+,f_-)\mapsto \vol(X_U)^{-1}\circ f_+\circ f_-^\vee\in\End^0(A)=M_A$
for all levels $U$.\index{canonical pairing, $(\;,\;)_A$}
\end{definition}

Now we take a simple abelian variety $A$ over $F$ of $\GL(2)$-type. For a
finite place $v$ of $F$, choose a rational prime $\ell$ that does not divide
$v$. We have a Galois representation $\rho_{A,v}$ of $\rD_v$, the
decomposition group at $v$, on the $\ell$-adic Tate module $\rV_\ell(A)$ of
$A$, which is a free module over $M_{A,\ell}\colonequals
M_A\otimes_\dQ\dQ_\ell$ of rank $2$. It is well-known that the characteristic
polynomial
\[P_v(T)=\r{det}_{M_{A,\ell}}(1-\Frob_v\res\rV_\ell(A)^{\rI_v})\]
belongs to $M_A[T]$ and is independent of $\ell$, where $\rI_v\subset\rD_v$
is the inertia subgroup and $\Frob_v\in\rD_v/\rI_v$ is the geometric
Frobenius.

\begin{definition}[$L$-function and $\epsilon$-factor]\label{de:l_value}
Let $K$ be a field containing $M_A$.
\begin{enumerate}
  \item Define $L(s,\rho_{A,v})=P_v(N_v^{-s-1/2})^{-1}$ to be the local
      $L$-function. In a similar manner, we define the local adjoint
      $L$-function $L(s,\rho_{A,v},\Ad)$; in particular,
      $L(1,\rho_{A,v},\Ad)\in M_A$.

  \item For a locally constant character $\chi_v\colon F_v^\times\to
      K^\times$, we have the twisted local $L$-function
      $L(s,\rho_{A,v}\otimes\chi_v)$ and the $\epsilon$-factor
      $\epsilon(1/2,\psi,\rho_{A,v}\otimes\chi_v)$.

    \item For a locally constant character $\chi_v\colon E_v^\times\to
        K^\times$, we have the local Rankin--Selberg $L$-function
        $L(s,\rho_{A,v},\chi_v)$ and the $\epsilon$-factor
        $\epsilon(1/2,\rho_{A,v},\chi_v)$.

    \item Let $\iota\colon K\hookrightarrow\dC$ be an embedding.

    We define the global $L$-function
        \[L(s,\rho_A^{(\iota)})=\prod_{v<\infty}\iota L(s,\rho_{A,v}),\]
        which is absolutely convergent for $\RE s>1$. We say that $A$ is
        \emph{automorphic} if $L(s,\rho_A^{(\iota)})$, for some and hence
        all $\iota$, is (the finite component of) the $L$-function of an
        irreducible cuspidal automorphic representation of $\GL_2(\dA)$.
        We have global versions for the other $L$-functions and
        $\epsilon$-factors.
\end{enumerate}
\end{definition}

\begin{remark}
It is conjectured that every abelian variety of $\GL(2)$-type is automorphic.
In particular, when $F=\dQ$, every abelian variety of $\GL(2)$-type is
parameterized by modular curves. This follows from Serre's modularity
conjecture (for $\dQ$) \cite{Rib92}*{Theorem 4.4}, where the latter has been
proved by Khare and Wintenberger \cite{KW09}.
\end{remark}

\subsection{$p$-adic $L$-function in terms of abelian varieties}
\label{ss:l_function_abelian}

Let $A$ be the simple abelian variety up to isogeny (of $\GL(2)$-type) over
$F$ that gives rise to the classical representation $\pi$ in
\Sec\ref{ss:l_function}. In particular, it is automorphic. We put
$M=M_A=M_{A^\vee}$ which is regarded as a subfield of $\dC_p$. Denote by
$\omega_A\colon F^\times\backslash\dA^{\infty\times}\to M^\times$ the central
character associated to $A$. For simplicity, we also put $F^M=F\otimes_\dQ M$
equipped with a natural map to $\dC_p$.

\begin{definition}[Distribution algebra]\label{de:character_space}
Denote by $\Gamma_E$ the set of compact open subgroups $V^\p$ of
$\dA^{\infty\p\times}_E$, which is a filtered partially ordered set under
inclusion. Let $K/F_\p$ be a complete field extension.
\begin{enumerate}
  \item A ($K$-valued) character
      \[\chi\colon E^\times\backslash
      \dA^{\infty\times}_E\to K^\times\] is a \emph{character of
      weight $w\in\dZ$ and tame level $V^\p\in\Gamma_E$} if
      \begin{itemize}
      \item $\chi$ is invariant under some $V^\p\in\Gamma_E$;

      \item there is a compact open subgroup $V_\p$ of $E_\p^\times$
          and $w\in\dZ$ such that $\chi(t)=(t_{\fP}/t_{\fP^c})^w$ for
          $t\in V_\p$.
      \end{itemize}
      We suppress the word \emph{tame level} if $V^\p$ is not specified.

  \item For a $K$-valued character $\chi$ of weight $w$ as above, we
      define two characters $\check\chi_\fP$ and $\check\chi_{\fP^c}$ of
      $F_\p^\times$ by the formula $\check\chi_\fP(t)=t^{-w}\chi_\fP(t)$
      and $\check\chi_{\fP^c}(t)=t^w\chi_{\fP^c}(t)$.

  \item Suppose $K$ is contained in $\dC_p$. Let $\chi$ be a locally
      algebraic character of weight $w$. Given an isomorphism
      $\iota\colon\dC_p\xrightarrow{\sim}\dC$, we define the following
      local characters
      \begin{itemize}
      \item $\chi^{(\iota)}_v=1$ if $v|\infty$ but not equal to
          $\iota\res_F$;

      \item $\chi^{(\iota)}_v(z)=(z/z^c)^w$ for $v=\iota\res_F$,
          where $z\in
          E\otimes_{F,\iota}\dR\xrightarrow{\iota\res_E}\dC$;

      \item $\chi^{(\iota)}_v=\iota\chi_v$ for $v<\infty$ but
          $v\neq\p$;

      \item
          $\chi^{(\iota)}_\p(t)=\iota\(\check\chi_\fP(t_\circ)\check\chi_{\fP^c}(t_\bullet)\)$
          for $t\in E_\p^\times$.
      \end{itemize}
      In particular,
      $\chi^{(\iota)}\colonequals\otimes_v\chi^{(\iota)}_v\colon
      \dA^\times_E\to\dC^\times$ is an automorphic character, which is
      called the \emph{$\iota$-avatar} of $\chi$.

  \item Suppose $K$ contains $M$. A $K$-valued characters $\chi$ of
      weight $w$ is \emph{$A$-related} if
      \begin{itemize}
        \item $\omega_A\cdot \chi\res_{\dA^{\infty\times}}=1$;

        \item
            $\#\{v<\infty,v\neq\p\res\epsilon(1/2,\rho_{A,v},\chi_v)=-1\}\equiv
            g+1\mod 2$.
      \end{itemize}
      Denote by $\Xi(A,K)_w$ the set of all $A$-related $K$-valued
      characters of weight $w$. Put $\Xi(A,K)=\bigcup_\dZ\Xi(A,K)_w$. For
      $\chi\in\Xi(A,K)$, there is a unique up to isomorphism totally
      definite incoherent quaternion algebra $\dB_\chi$ over $\dA$,
      unramified at $\p$, such that
      $\epsilon(1/2,\rho_{A,v},\chi_v)=\chi_v(-1)\eta_v(-1)\epsilon(\dB_{\chi,v})$
      for every finite place $v\neq\p$ of $F$. The algebra $\dB_\chi$ is
      $E$-embeddable and by which $A$ can be
      parameterized.
      \index{$p$-adic character of weight $w$!$A$-related, $\Xi(A,K)_w$}

  \item Suppose $K$ contains $M$. For a locally constant character
      $\omega\colon F^\times\backslash\dA^{\infty\times}\to M^\times$,
      denote by $\sC(\omega,K)$ the set of locally analytic $K$-valued
      functions $f$ on the locally $F_\p$-analytic group
      $E^\times\backslash\dA^{\infty\times}_E$ satisfying
      \begin{itemize}
        \item $f$ is invariant under translation by some
            $V^\p\in\Gamma_E$;
        \item $f(xt)=\omega(t)^{-1}f(x)$ for all $x\in
            E^\times\backslash\dA^{\infty\times}_E$ and $t\in
            F^\times\backslash\dA^{\infty\times}$.
      \end{itemize}
      Then $\sC(\omega,K)$ is a locally convex $K$-vector space. Let
      $\sD(\omega,K)$ be the strong dual of $\sC(\omega,K)$, which we
      call the \emph{$K$-valued $\omega$-related distribution algebra}.
      It is a commutative $K$-algebra by convolution.
      \index{distribution algebra!$\omega$-related, $\sD(\omega,K)$}

      For a complete field extension $K'/K$, we have
      $\sD(\omega,K)\wtimes_KK'\simeq \sD(\omega,K')$. In fact, if $K$ is
      discretely valued, $\sD(\omega,K)$ may be written as a projective
      limit, indexed by tame levels $V^\p\in\Gamma_E$, of nuclear
      Fr\'{e}chet--Stein $K$-algebras with finite \'{e}tale transition
      homomorphisms (see Remark \ref{re:nuclear}), and thus complete. We
      have a continuous homomorphism
      \[[\;]\colon E^\times\backslash\dA^{\infty\times}_E\to\sD(\omega,K)^\times\]
      given by Dirac distributions.

  \item Suppose $K$ contains $M$. Define $\sD(A,K)$ to be the quotient
      $K$-algebra of $\sD(\omega_A,K)$ by the closed ideal generated by
      elements that vanish on $\Xi(A,K)\subset\sC(\omega_A,K)$, which we
      call the \emph{$K$-valued $A$-related distribution algebra}. For a
      complete field extension $K'/K$, we have $\sD(A,K)\wtimes_KK'\simeq
      \sD(A,K')$. Similar to $\sD(\omega_A,K)$, if $K$ is discretely
      valued, $\sD(A,K)$ may be written as a projective limit of nuclear
      Fr\'{e}chet--Stein $K$-algebras.
      \index{distribution algebra!$A$-related, $\sD(A,K)$}

  \item For $\pi\in\sA_{\dC_p}(\dB^\times)$ as in
      \Sec\ref{ss:l_function}, we define $\sD(\pi)$ to be the quotient of
      $\sD(\omega_A,\dC_p)$ by the closed ideal generated by elements
      that vanish on $\chi\in\Xi(A,\dC_p)$ with $\dB_\chi\simeq\dB$,
      which we call the \emph{$\pi$-related distribution algebra}. In
      particular, we have a quotient map
      $\varsigma\colon\sD(A,\dC_p)\to\sD(\pi)$.
\end{enumerate}
\end{definition}

Let $K$ be a complete field containing $MF_\p^\lt$. Consider a character
$\chi\in\Xi(A,K)_k$. Take $\dB=\dB_\chi$, and choose an $E$-embedding. Then
we have the $F$-scheme $X$ and its closed subscheme $Y=Y^+\coprod Y^-$. Put
$A^+=A$ and $A^-=A^\vee$, and $\Pi^\pm=\Pi_{A^\pm}\in\cA(\dB^\times,\dQ)$. We
have the canonical pairing $(\;,\;)_A\colon\Pi^+\times\Pi^-\to M$.

Define $\sigma_\chi^\pm$ to be the $K$-subspace of
$\rH^0(Y^\pm,\Omega_{X,Y^\pm}^{\otimes-k})\otimes_FK$ consisting of $\varphi$
such that $t^*\varphi=\chi(t)^{\pm 1}\varphi$, where
$\Omega_{X,Y^\pm}=\Omega^1_X\res_{Y^\pm}$. The abstract conjugation $\bc$
induces a $\dA^{\infty\times}_E$-invariant pairing
\[(\;,\;)_\chi\colon\sigma_\chi^+\times\sigma_\chi^-\to K\]
by the formula
$(\varphi_+,\varphi_-)_\chi=(\varphi_+\otimes\omega_{\psi+}^k)\cdot
\bc^*(\varphi_-\otimes\omega_{\psi-}^k)$, where the right-hand side is a
$K$-valued constant function on $Y^+$, hence regarded as an element in $K$.
Here,
\begin{align}\label{eq:omega_lt}
\omega_{\psi\pm}=\Upsilon_\pm^*\omega_\nu\res_{Y^\pm}
\end{align}
is the a section of $\Omega_{X,Y^\pm}$, where $\omega_\nu$ is the global
Lubin--Tate differential in Definition \ref{de:global_differential}. It is
determined the additive character $\psi$.

Assume $K$ is contained in $\dC_p$ and $k\geq 1$. For every
$\iota\colon\dC_p\xrightarrow{\sim}\dC$, we have a
$\dB^{\infty\times}\times\dA^{\infty\times}_E$-invariant pairing
\[(\;,\;)_{A,\chi}^{(\iota)}\colon
(\Pi^+\otimes_{F^M}\sigma_\chi^+)\times(\Pi^-\otimes_{F^M}\sigma_\chi^-)\to
(\Lie A^+\otimes_{F^M}\Lie A^-)\otimes_{F^M,\iota}\dC,\] such that for
$f_\pm\in\Pi^\pm$, $\varphi_\pm\in\sigma_\chi^\pm$ and
$\omega_\pm\in\rH^0(A^\pm,\Omega^1_{A^\pm})$,
\begin{multline*}
\langle\omega_+\otimes\omega_-,(f_+\otimes\varphi_+,f_-\otimes\varphi_-)_{A,\chi}^{(\iota)}\rangle=\\
\times\(\iota\varphi_+\otimes\fc_\iota^*\iota\varphi_-\otimes\mu^k\)\times
\int_{X_\iota(\dC)}\frac{\Theta_\iota^{k-1}f_+^*\omega_+\otimes\fc_\iota^*\Theta^{k-1}_\iota
f_-^*\omega_-}{\mu^k}\rrd x,
\end{multline*}
where
\begin{itemize}
  \item $\langle\;,\;\rangle$ is the canonical pairing between
      $\rH^0(A^\pm,\Omega^1_{A^\pm})$ and $\Lie A^\pm$;
  \item $\mu$ is an arbitrary Hecke invariant hyperbolic metric on
      $X_\iota(\dC)$;
  \item $\iota\varphi_+\otimes\fc_\iota^*\iota\varphi_-\otimes\mu^k$ is a
      constant function on $Y^+_\iota(\dC)$, hence viewed as a complex
      number; and
  \item $\rd x$ is the Tamagawa measure on $X_\iota(\dC)$.
\end{itemize}
Define $\bP_\iota(\chi)$ to be the unique element in $\dC^\times$ such that
\[(\;,\;)_{A,\chi}^{(\iota)}=\bP_\iota(\chi)\cdot\iota(\;,\;)_A\otimes\iota(\;,\;)_\chi,\]
which we call the \emph{period ratio (at $\iota$)}, as a function on
$\bigcup_{k\geq 1}\Xi(A,K)_k$.

\begin{theorem}\label{th:l_function_abelian}
There is a unique element
\[\sL(A)\in(\Lie A^+\otimes_{F^M}\Lie A^-)\otimes_{F^M}\sD(A,MF_\p^\lt)\]
such that for every character $\chi\in\Xi(A,K)_k$ with $k\geq 1$ and
$MF_\p^\lt\subset K\subset\dC_p$ a complete intermediate field, and every
$\iota\colon\dC_p\xrightarrow{\sim}\dC$,
\begin{align}\label{eq:interpolation}
\iota\sL(A)(\chi)=L(1/2,\rho_A^{(\iota)},\chi^{(\iota)})\cdot
\frac{2^{g-3}\delta_E^{1/2}\zeta_F(2)\bP_\iota(\chi)}{L(1,\eta)^2L(1,\rho_A^{(\iota)},\Ad)}
\frac{\iota\epsilon(1/2,\psi,\rho_{A,\p}\otimes\check\chi_{{\fP}^c})}
{\iota L(1/2,\rho_{A,\p}\otimes\check\chi_{{\fP}^c})^2}.
\end{align}
\end{theorem}

\begin{lem}\label{le:l_function}
Theorem \ref{th:l_function_abelian} implies Theorem
\ref{th:l_function_maass}.
\end{lem}

\begin{proof}
Apparently, we only need to prove Theorem \ref{th:l_function_maass} for one
$p$-adic Petersson inner product $(\;,\;)_\pi$. Choose a basis element
$\omega_\pm$ of the rank-$1$ free $F^M$-module $\Lie A^\pm$. They together
defines a pairing $(\;,\;)_\pi$ such that
\[(f_+^*\log_{\omega^+},f_-^*\log_{\omega^-})_\pi=(f_+,f_-)_A\]
for every $f_\pm\in\Pi^\pm$. Define $\sL(\pi)$ to be
$\langle\omega_+\otimes\omega_-,\varsigma\sL(A)\rangle$, where $\varsigma$ is
introduced in \ref{de:character_space} (7). Then by Lemma
\ref{le:shimura_maass}, $\sL(\pi)$ satisfies the requirement in Theorem
\ref{th:l_function_maass} for the above $p$-adic Petersson inner product.
\end{proof}

\subsection{Heegner cycles and $p$-adic Waldspurger formula}
\label{ss:heegner_waldspurger}

Let $K$ be a complete field containing $M$. Consider an element
$\chi\in\Xi(A,K)_0$, regarded as a locally constant character of
$\Gal(E^\ab/E)$ via global class field theory. Take $\dB=\dB_\chi$, and
choose an $E$-embedding. We choose a CM point $P^+\in Y^+(E^\ab)=Y^+(\dC_p)$
and put $P^-=\bc P^+$. For each $f_\pm\in\Pi^\pm$, we have a Heegner cycle
$P^\pm_\chi(f_\pm)$ on $A^\pm$ defined by the formula
\begin{align}\label{eq:heegner_cycle}
P^\pm_\chi(f_\pm)=\int_{\Gal(E^\ab/E)}f_\pm(\tau P^\pm)\otimes_M\chi(\tau)^{\pm1}\rrd\tau,
\end{align}
where $\rd\tau$ is the Haar measure on $\Gal(E^\ab/E)$ of total volume $1$.

Suppose $K$ contains $MF_\p^\ab$. We have a $K$-linear map
\[\log_{A^\pm}\colon A^\pm(K)_{\dQ}\otimes_M K\to\Lie
A^\pm\otimes_{F^M}K.\] As a functional on $\Pi^+\times\Pi^-$, the product
$\log_{A^+}P^+_\chi(f_+)\cdot\log_{A^-}P^-_\chi(f_-)$ defines an element in
the following one dimensional $K$-vector space
\[\Hom_{\dA^{\infty\times}_E}(\Pi^+\otimes\chi,K)
\otimes_K\Hom_{\dA^{\infty\times}_E}(\Pi^-\otimes\chi^{-1},K)
\otimes_{F^M}(\Lie A^+\otimes_{F^M}\Lie A^-).\] On the other hand, using
matrix coefficient integral, we construct a basis $\alpha_\chi(\;,\;)$ of the
space $\Hom_{\dA^{\infty\times}_E}(\Pi^+\otimes\chi,K)
\otimes_K\Hom_{\dA^{\infty\times}_E}(\Pi^-\otimes\chi^{-1},K)$. It satisfies
that for every $\iota\colon\dC_p\xrightarrow{\sim}\dC$,
\[\iota\alpha_\chi(f_+,f_-)=\alpha^\natural(f_+,f_-;\chi^{(\iota)}),\]
where the last term is introduced in Definition \ref{de:matrix_integral}.

\begin{theorem}[$p$-adic Waldspurger formula]\label{th:waldspurger_abelian}
Let the notation be as above. For $\chi\in\Xi(A,K)_0$, we have
\[\log_{A^+}P^+_\chi(f_+)\cdot\log_{A^-}P^-_\chi(f_-)
=\sL(A)(\chi)\cdot\frac{L(1/2,\rho_{A,\p}\otimes\check\chi_{\fP^c})^2}
{\epsilon(1/2,\psi,\rho_{A,\p}\otimes\check\chi_{\fP^c})}\cdot
\alpha_\chi(f_+,f_-).\]
\end{theorem}

\begin{lem}\label{le:waldspurger}
Theorem \ref{th:waldspurger_abelian} implies Theorem
\ref{th:waldspurger_maass}.
\end{lem}

\begin{proof}
We only need to prove Theorem \ref{th:waldspurger_maass} for one
$(\;,\;)_\pi$ and one nonzero element $\varphi_\pm\in\sigma_\chi^\pm$. We use
the $p$-adic Petersson inner product in the proof of Lemma
\ref{le:l_function}, and choose $\varphi_\pm$ such that
$\varphi_\pm(P^\pm)=1$. Then the lemma follows by definition.
\end{proof}

\section{Construction of $p$-adic $L$-function}

This chapter is dedicated to the proofs of Theorems
\ref{th:l_function_abelian} and \ref{th:waldspurger_abelian}. In
\Sec\ref{ss:distribution_matrix}, we construct the distribution interpolating
matrix coefficient integrals appearing in the classical Waldspurger formula.
We construct the universal torus period in \Sec\ref{ss:universal_torus},
which is a crucial construction toward the $p$-adic $L$-function. In
\Sec\ref{ss:interpolation_universal}, we study the relation between universal
torus periods and classical torus periods, based on which we accomplish the
proofs of our main theorems in \Sec\ref{ss:proof_theorems}.

\subsection{Distribution of matrix coefficient integrals}
\label{ss:distribution_matrix}

Let $K/MF_\p$ be a complete field extension.

\begin{definition}\label{de:character_ramification}
For $w\in\dZ$, $n\in\dN$, and a locally constant character $\omega\colon
F^\times\backslash\dA^{\infty\times}\to M^\times$, we say a $K$-valued
character $\chi\colon E^\times\backslash\dA^{\infty\times}_E\to K^\times$ of
weight $w$ is \emph{of central type $\omega$ and depth $n$} if
\begin{itemize}
    \item $\omega\cdot \chi\res_{\dA^{\infty\times}}=1$, and

    \item $\chi_{\fP^c}(t)=t^{-w}$ for all $t\in(1+\p^n)^\times$.
\end{itemize}
We denote by $\Xi(\omega,K)^n_w$ the set of all $K$-valued characters of
weight $w$, central type $\omega$ and depth $n$. Moreover, put
$\Xi(\omega,K)^n=\bigcup_\dZ\Xi(\omega,K)^n_w$. We have a natural pairing
$\sD(\omega,K)\times\Xi(\omega,K)^n\to K$.\index{$p$-adic character of weight
$w$!of central type $\omega$ and depth $n$, $\Xi(\omega,K)^n_w$}
\end{definition}

Choose a $\dB$ by which $A$ can be parameterized, together with an
$E$-embedding. Recall that we put $\Pi^\pm=\Pi(\dB)_{A^\pm}$.

\begin{definition}[Stable vector]\label{de:heartsuit_representation}
An element $f_\pm$ in $\Pi^\pm\otimes_MK$ (resp.\ $\Pi^\pm_\p\otimes_MK$) is
a \emph{stable vector} if
\begin{enumerate}
  \item $f_\pm$ is fixed by $\rN^\pm(O_\p)$;
  \item $f_\pm$ satisfies the equation
      \begin{align*}
      \sum_{\rN^\pm(\p^{-1})/\rN^\pm(O_\p)}\Pi^\pm_\p(g)f_\pm=0.
      \end{align*}
\end{enumerate}
We denote by $(\Pi^\pm)^\heartsuit_K$ (resp.\ $(\Pi^\pm_\p)^\heartsuit_K$)
the subset of $\Pi^\pm\otimes_MK$ (resp.\ $\Pi^\pm_\p\otimes_MK$) of stable
vectors.\index{stable vector, $(\Pi^\pm)^\heartsuit_K$}

For $n\in\dN$, we say a stable vector $f_\pm\in(\Pi^\pm)^\heartsuit_K$ (or
$(\Pi^\pm_\p)^\heartsuit_K$) is \emph{$n$-admissible} if
\[\Pi^\pm_\p(\rn^\pm(x))f_\pm=\psi^\pm(x)f_\pm\]
for every $x\in\p^{-n}/O_\p$, where $\rn^\pm(x)$ is same as in Proposition
\ref{pr:lt_action}.\index{stable vector, $(\Pi^\pm)^\heartsuit_K$!admissible}
\end{definition}

\begin{remark}\label{re:kirillov}
If we realize $\Pi_\p^\pm$ in the Kirillov model with respect to the pair
$(\rN^+,\psi^\pm)$, then $f_{\pm\p}$ belongs to $(\Pi^\pm_\p)^\heartsuit_K$
if and only if $f_{+\p}$ (resp.\ $\Pi_\p^-(\rJ)f_{-\p}$) is supported on
$O_\p^\times$, and is $n$-admissible if and only if $f_{+\p}$ (resp.\
$\Pi_\p^-(\rJ)f_{-\p}$) is supported on $(1+\p^n)^\times$.
\end{remark}

We recall the definition of the classical (normalized) matrix coefficient
integral. Suppose $K$ is contained in $\dC_p$. We take a character
$\chi\in\Xi(A,K)$.

\begin{definition}[Regularized matrix coefficient integral]\label{de:matrix_integral}
Let $\iota\colon\dC_p\xrightarrow{\sim}\dC$ be an isomorphism. We recall the
regularization of the following matrix coefficient integral
\[\alpha^\natural(f_+,f_-;\chi^{(\iota)})\text{ ``$=$'' }
\int_{\dA^{\infty\times}\backslash\dA_E^{\infty\times}}\iota(\Pi(t)f_+,f_-)_A\cdot\chi^{(\iota)}(t)\rrd
t.\] To do this, we take any decomposition
$\iota(\;,\;)_A=\prod_{v<\infty}(\;,\;)_{\iota,v}$ where
$(\;,\;)_{\iota,v}\colon\Pi^+_v\times\Pi^-_v\to\dC$ is a
$\dB^\times_v$-invariant bilinear pairing. For
$f_\pm=\otimes_{v<\infty}f_{\pm v}$ such that $(f_{+v},f_{-v})_{\iota,v}=1$
for all but finitely many $v$, we put
\begin{align*}
\alpha(f_{+v},f_{-v};\chi_v^{(\iota)})&=
\int_{F_v^\times\backslash E_v^\times}(\Pi_v(t)f_{+v},f_{-v})_{\iota,v}\chi_v^{(\iota)}(t)\rrd t;\\
\alpha^\natural(f_{+v},f_{-v};\chi^{(\iota)}_v)&=
\(\frac{\zeta_{F_v}(2)L(1/2,\rho_{A,v}^{(\iota)},\chi^{(\iota)}_v)}{L(1,\eta_v)L(1,\rho_{A,v}^{(\iota)},\Ad)}\)^{-1}
\alpha(f_{+v},f_{-v};\chi_v^{(\iota)}).
\end{align*}
Here, $\rd t$ is the measure on $F_v^\times\backslash E_v^\times$ given
determined in \Sec\ref{ss:notation_conventions}, and $\rho_{A,v}^{(\iota)}$
is the corresponding admissible complex representation of $\dB_v^\times$ via
$\iota$. Then we have $\alpha^\natural(f_{+v},f_{-v};\chi^{(\iota)}_v)=1$ for
all but finitely many $v$, and the product
\[\alpha^\natural(f_+,f_-;\chi^{(\iota)})\colonequals\prod_{v<\infty}\alpha^\natural(f_{+v},f_{-v};\chi^{(\iota)}_v),\]
which is well-defined, does not depend on the choice of the decomposition of
$\iota(\;,\;)_A$. We extend the functional
$\alpha^\natural(\;,\;;\chi^{(\iota)})$ to all $f_+,f_-$ by linearity.

The regularization of the matrix coefficient integral for
$\iota\alpha^\natural(\phi_+,\phi_-;\varphi_+,\varphi_-)$ in
\Sec\ref{ss:waldspurger_formula} is defined similarly as above.
\end{definition}

The following proposition is our main result.

\begin{proposition}\label{pr:matrix_integral}
Let $MF_p\subset K\subset\dC_p$ be a complete intermediate field. Let
$f_\pm\in(\Pi^\pm)^\heartsuit_K$ be an $n$-admissible stable vector for some
$n\in\dN$. Then there is a unique element $\sQ(f_+,f_-)\in\sD(\omega_A,K)$
such that for all $K$-valued characters $\chi\in\Xi(\omega_A,K)^n$ of central
type $\omega_A$ and depth $n$, and $\iota\colon\dC_p\xrightarrow{\sim}\dC$,
\begin{align*}
\iota\sQ(f_+,f_-)(\chi)=\iota\(\frac{L(1/2,\rho_{A,\p}\otimes\check\chi_{\fP^c})^2}
{\epsilon(1/2,\psi,\rho_{A,\p}\otimes\check\chi_{\fP^c})}\)\cdot
\alpha^\natural(f_+,f_-;\chi^{(\iota)}).
\end{align*}
\end{proposition}

The element $\sQ(f_+,f_-)$ is called the \emph{($K$-valued) local period
distribution}.

Before giving the proof, we make a convenient choice of a decomposition of
$(\;,\;)_A$. Realize the representation $\Pi_\p^\pm$ in the Kirillov model as
in Remark \ref{re:kirillov}. We may assume that $f_\pm=\otimes f_{\pm v}$,
with $f_{\pm v}\in\Pi^\pm_v\otimes_MK$, is decomposable and is fixed by some
$V^\p\in\Gamma_E$ sufficiently small. Choose a decomposition
$(\;,\;)_A=\prod_{v<\infty}(\;,\;)_v$ such that
\begin{enumerate}
  \item $(f_{+v},f_{-v})_v=1$ for all but finitely many $v$;
  \item $(f'_{+v},f'_{-v})_v\in K$ for all $f'_{\pm
      v}\in\Pi_v^\pm\otimes_MK$;
  \item for $f'_{\pm\p}\in\Pi_\p^\pm\otimes_MK$ that is compactly
      supported on $F_\p^\times$,
      \[(f'_{+\p},f'_{-\p})_\p=\int_{F_\p^\times}f'_{+\p}(a)f'_{-\p}(a)\rrd
      a,\] where $\rd a$ is the Haar measure on $F_\p^\times$ such that
      the volume of $O_\p^\times$ is $1$.
\end{enumerate}

We need two lemmas for the proof of the proposition, where for simplicity we
write $\omega=\omega_A$. For each finite place $v\neq\p$, let
$\sD(\omega_v,K,V_v^\p)$ be the quotient of $D(E_v^\times/V_v^\p,K)$ by the
closed ideal generated by $\{\omega_v(t)[t]-1\res t\in F_v^\times\}$. Put
\[\sD(\omega_v,K)=\varprojlim_{V_v^\p}\sD(\omega_v,K,V_v^\p),\]
where the limit runs over all compact open subgroups $V_v^\p$ of
$E_v^\times$. Let $\sD(\omega_\p,K)$ be the quotient of $D(E_\p^\times,K)$ by
the closed ideal generated by $\{\omega_\p(t)[t]-1\res t\in F_\p^\times\}$.
We have natural homomorphisms $\sD(\omega_v,K)\to\sD(\omega,K)$ for all
finite places $v$.

\begin{lem}\label{le:matrix_tame}
Let $v\neq\p$ be a finite place of $F$.
\begin{enumerate}
  \item There exists a unique element
      \[\cL^{-1}(\rho_{A,v})\in\sD(\omega_v,MF_\p)\] such that for every
      locally constant character $\chi_v\colon E_v^\times\to K^\times$
      satisfying $\omega_v\cdot\chi_v\res_{F_v^\times}=1$,
      \[\cL^{-1}(\rho_{A,v})(\chi_v)=L(1/2,\rho_{A,v},\chi_v)^{-1}.\]

  \item  For $f_{\pm v}\in\Pi_v^\pm\otimes_MK$, there exists a unique
      element
      \[\sQ(f_{+v},f_{-v})\in\sD(\omega_v,K)\] such that for every locally
      constant character $\chi_v\colon E_v^\times\to K^\times$ satisfying
      $\omega_v\cdot\chi_v\res_{F_v^\times}=1$, and
      $\iota\colon\dC_p\xrightarrow{\sim}\dC$,
      \[\iota\sQ(f_{+v},f_{-v})(\chi_v)=\alpha^\natural(f_{+v},f_{-v};\chi^{(\iota)}_v).\]
\end{enumerate}
\end{lem}

\begin{proof}
The uniqueness is clear. In the following proof, we suppress $v$ from the
notation and we will use the subscript $\iota$ for all changing of
coefficients of representations via $\iota$.

To prove (1), we first consider the following situation. Let $\tilde F$ be
either $F$ or $E$, and $\tilde\Pi$ be an irreducible admissible
$M$-representation of $\GL_2(\tilde F)$. We claim that there is a (unique)
element $\cL^{-1}_{\tilde F}(\tilde\Pi)\in D_\flat(\tilde F,MF_\p)$, where
$D_\flat(\tilde F,K)=\varprojlim_VD(\tilde F^\times/V,K)$ with $V$ running
over all compact open subgroups of $\tilde F^\times$, such that for every
locally constant character $\chi\colon\tilde F^\times\to K^\times$ and
$\iota\colon\dC_p\xrightarrow{\sim}\dC$,
\[\iota\cL^{-1}_{\tilde F}(\tilde\Pi)(\chi)=L(1/2,\tilde\Pi_\iota\otimes\chi_\iota)^{-1}.\]
In fact, for a locally constant character $\mu\colon\tilde F^\times\to
M^\times$, define $\cL^{-1}_{\tilde F}(\mu)\in D_\flat(\tilde
F^\times,MF_\p)$ by the formula
\[\cL^{-1}_{\tilde F}(\mu)(h)=1-\int_{O_{\tilde F}^\times}\mu(\tilde\varpi a)h(\tilde\varpi a)\rrd a\]
for $h\in\varinjlim_V C(\tilde F^\times/V,MF_\p)$. Here, $\tilde\varpi$ is an
arbitrary uniformizer of $\tilde F$, and $\rd a$ is the Haar measure on
$O_{\tilde F}^\times$ with total volume $1$. Then we have three cases:
\begin{itemize}
  \item If $\tilde\Pi$ is supercuspidal, put $\cL^{-1}_{\tilde
      F}(\tilde\Pi)=1$.

  \item If $\tilde\Pi$ is the unique irreducible subrepresentation of the
      non-normalized parabolic induction of $(\mu,\mu|\;|^{-2})$ for a
      character $\mu\colon \tilde F^\times\to M^\times$, put
      $\cL^{-1}_{\tilde F}(\tilde\Pi)=\cL^{-1}_{\tilde F}(\mu)$.

  \item If $\tilde\Pi$ is the irreducible non-normalized parabolic
      induction of $(\mu^1,\mu^2|\;|^{-1})$ for a pair of characters
      $\mu^i\colon\tilde F^\times\to M^\times$ ($i=1,2$), put
      $\cL^{-1}_{\tilde F}(\tilde\Pi)=\cL^{-1}_{\tilde
      F}(\mu^1)\cdot\cL^{-1}_{\tilde F}(\mu^2)$.
\end{itemize}

Go back to (1). First, assume $E/F$ is non-split. Then we define
$\cL^{-1}(\rho_A)$ to be the image of $\cL^{-1}_E(\Pi_E)$ in
$\sD(\omega,MF_\p)$ where $\Pi_E$ is the base change of $\Pi$ to $\GL_2(E)$,
which depends only on $\rho_A$. Second, assume $E=F_\bullet\times F_\circ$ is
split where $F_\bullet=F_\circ=F$. Then we define $\cL^{-1}(\rho_A)$ to be
the image of $\cL^{-1}_{F_\bullet}(\Pi)\otimes\cL^{-1}_{F_\circ}(\Pi)$ in
$\sD(\omega,MF_\p)$.

Now we consider (2). First, assume $E/F$ is non-split. Then the torus
$F^\times\backslash E^\times$ is compact and hence the matrix coefficient
$\Phi_{f_+,f_-}(g)\colonequals(\Pi^+(g)f_+,f_-)$ is finite under
$E^\times$-translation. We may assume the restriction
$\Phi_{f_+,f_-}|_{E^\times}=\sum_i a_i\chi_i$ is a finite $K$-linear
combination of $K$-valued (locally constant) characters $\chi_i$ of
$E^\times$ such that $\omega\cdot\chi_i\res_{F^\times}=1$. To every locally
constant function $h$ on $E^\times$ satisfying $\omega(t)h(at)=h(a)$ for all
$a\in E^\times$ and $t\in F^\times$, assigning the integral
\[\sum_ia_i\int_{F^\times\backslash E^\times}\chi_i(t)h(t)\rrd t,\]
which is a finite sum, defines an element $\alpha(f_+,f_-)$ in
$\sD(\omega,K)$. Put
\[\sQ(f_+,f_-)=\(\frac{\zeta_F(2)}{L(1,\rho_A,\Ad)L(1,\eta)}\)^{-1}\cL^{-1}(\rho_A)\alpha(f_+,f_-).\]

Second, assume $E=F_\bullet\times F_\circ$ is split. We assume the embedding
$E\to\Mat_2(F)$ is given by
\[(t_\bullet,t_\circ)\mapsto\left(
                              \begin{array}{cc}
                                t_\bullet &  \\
                                 & t_\circ \\
                              \end{array}
                            \right)
\] for $t_\bullet,t_\circ\in F$. Moreover, a character $\chi$ of $E^\times$ is given by a pair
$(\chi_\bullet,\chi_\circ)$ of characters of $F^\times$ such that
$\chi((t_\bullet,t_\circ))=\chi_\bullet(t_\bullet)\chi_\circ(t_\circ)$.

We realize $\Pi^\pm$ in the Kirillov model with respect to a (nontrivial)
additive character $\psi^\pm\colon F\to\dC^\times$ of conductor $0$ where
$\psi^-=(\psi^+)^{-1}$. Moreover, we may assume for
$f_\pm\in\Pi^\pm\otimes_MK$ that is compactly supported on $F^\times$,
\[(f_+,f_-)=\int_{F^\times}f_+(a)f_-(a)\rrd a,\] where $\rd a$ is the Haar
measure on $F^\times$ such that the volume of $O_F^\times$ is $c$ for some
$c\in M$. We have the following formula
\begin{multline}\label{eq:matrix_tame}
\alpha^\natural(f_+,f_-;\chi_\iota)\\
=\(\frac{\zeta_F(2)L(1/2,\rho_A^{(\iota)},\chi_\iota)}{L(1,\eta)L(1,\rho_A^{(\iota)},\Ad)}\)^{-1}
\int_{F^\times}\iota f_+(a)\cdot\chi_{\bullet\iota}(a)\rrd a
\int_{F^\times}\iota f_-(b)\cdot\chi_{\bullet\iota}(b^{-1})\rrd b  \\
=\(\frac{\zeta_F(2)}{L(1,\eta)}L(1,\rho_A^{(\iota)},\Ad)\)^{-1} Z(\iota
f_+,\chi_{\bullet\iota})Z(\iota f_-,\chi_{\bullet\iota}^{-1}),
\end{multline}
where
\[Z(\iota f_\pm,\chi_{\bullet\iota}^{\pm1})=L(1/2,\Pi_\iota^\pm\otimes\chi_{\bullet\iota}^{\pm1})^{-1}
\int_{F^\times}\iota f_\pm(a)\cdot\chi_{\bullet\iota}^{\pm1}(a)\rrd a.\] To
conclude, we only need to show that there exists an element $\cZ(f_\pm)\in
D_\flat(F^\times,K)$ such that for every locally constant character
$\chi\colon F^\times\to K^\times$ and
$\iota\colon\dC_p\xrightarrow{\sim}\dC$, $\iota\cZ(f_\pm)(\chi)=Z(\iota
f_\pm,\chi_\iota^{\pm1})$. Without loss of generality, we only construct
$\cZ(f_+)$.

Enlarging $M$ if necessary to include $l^{1/2}$ where $l$ is the cardinality
of the residue field of $F$, there is a subspace $\Pi^{+,c}$ of $\Pi^+$ such
that $\Pi^{+,c}\otimes_MK$ is the subspace of $\Pi^+\otimes_MK$ of functions
that are compactly supported on $F^\times$. For $f_+\in\Pi^{+,c}\otimes_MK$,
we may define $\cZ(f_+)$ such that for every locally constant function $h$ on
$F^\times$
\[\cZ(f_+)(h)=\cL^{-1}_F(\Pi^+)(h)\times
\int_{F^\times}f_+(a)h(a)\rrd a.\] Therefore, we may conclude if
$\dim\Pi^+/\Pi^{+,c}=0$. There are two cases remaining.

First, $\Pi^+$ is a special representation, that is, $\dim\Pi^+/\Pi^{+,c}=1$.
We may choose a representative $f_+=\mu(a)\cdot\Char_{O_F\setminus\{0\}}(a)$
for some character $\mu\colon F^\times\to M^\times$. Then $Z(\iota
f_\pm,\chi_\iota^{\pm1})=c$ (resp.\ $0$) if $\mu\cdot\chi$ is unramified
(resp.\ otherwise). Therefore, we may define $\cZ(f_+)$ such that
\[\cZ(f_+)(h)=\int_{O_F^\times}\mu(a)h(a)\rrd a\]
for every locally constant function $h$ on $F^\times$.

Second, $\Pi^+$ is a principal series, that is, $\dim\Pi^+/\Pi^{+,c}=2$.
There are two possibilities. In the first case, we may choose representatives
$f^i_+=\mu^i(a)\cdot\Char_{O_F\setminus\{0\}}(a)$ for two \emph{different}
characters $\mu^1,\mu^2\colon F^\times\to M^\times$. Without loss of
generality, we consider $f^1_+$. Then $Z(\iota
f^1_+,\chi_\iota)=L(1/2,\mu^2_\iota\cdot\chi_\iota)^{-1}$ (resp.\ $0$) if
$\mu^1\cdot\chi$ is unramified (resp.\ otherwise). Therefore, we may define
$\cZ(f^1_+)$ such that
\[\cZ(f^1_+)(h)=\cL^{-1}_F(\mu^1)(h)\times\int_{O_F^\times}\mu(a)h(a)\rrd a\]
for every locally constant function $h$ on $F^\times$. In the second case, we
may choose representatives $f^1_+=\mu(a)\cdot\Char_{O_F\setminus\{0\}}(a)$
and $f^2_+=(1-\log_l|a|)\mu(a)\cdot\Char_{O_F\setminus\{0\}}(a)$ for some
character $\mu\colon F^\times\to M^\times$. The function $f^1_+$ has been
treated above. For $f^2_+$, we have $Z(\iota f^2_+,\chi_\iota)=c$ (resp.\
$0$) if $\mu\cdot\chi$ is unramified (resp.\ otherwise). Therefore, we may
define $\cZ(f_+)$ such that
\[\cZ(f_+)(h)=\int_{O_F^\times}\mu(a)h(a)\rrd a\]
for every locally constant function $h$ on $F^\times$.
\end{proof}

\begin{lem}\label{le:matrix_wild}
Let $f_{\pm\p}\in(\Pi_\p^\pm)^\heartsuit_K$ be $n$-admissible stable vectors.
Then there exists a unique element
\[\sQ(f_{+\p},f_{-\p})\in\sD(\omega_\p,K)\] with the following property: for every character
$\chi_\p\colon E_\p^\times\to K^\times$ satisfying
$\omega_\p\cdot\chi_\p\res_{F_\p^\times}=1$ and $\chi_{\fP^c}(t)=t^{-w}$ for
$t\in(1+\p^n)^\times$ and some $w\in\dZ$, and
$\iota\colon\dC_p\xrightarrow{\sim}\dC$, we have
\[\iota\sQ(f_{+\p},f_{-\p})(\chi_\p)=
\iota\(\frac{L(1/2,\rho_{A,\p}\otimes\check\chi_{\fP^c})^2}
{\epsilon(1/2,\psi,\rho_{A,\p}\otimes\check\chi_{\fP^c})}\)
\alpha^\natural(f_{+\p},f_{-\p};\chi^{(\iota)}_\p).\] Here, $\check\chi$ is
defined similarly as in Definition \ref{de:character_space} (2).
\end{lem}

\begin{proof}
The uniqueness is clear. Note that the formula \eqref{eq:matrix_tame} also
works at $\p$. Moreover, we have the functional equation
\[Z(\iota f_{-\p},\chi_{\fP^c}^{(\iota)-1})=
\iota\epsilon(1/2,\psi,\rho_{A,\p}\otimes\check\chi_{\fP^c}) \cdot
Z(\iota(\Pi_\p^-(\rJ)f_{-\p}),\chi_{\fP^c}^{(\iota)}).\] By Remark
\ref{re:kirillov}, we only need to show that for $f\in\Pi_\p^\pm\otimes_MK$
that is supported on $(1+\p^n)^\times$, there exists
$\sQ'(f)\in\sD(\omega_\p,K)$ such that for $\chi_\p$ as in the statement and
$\iota$,
\[\iota\sQ'(f)(\chi_\p)=\int_{O_\p^\times}\iota f(a)\cdot\chi_{\fP^c}^{(\iota)}(a)\rrd
a.\] But in fact since $\chi_{\fP^c}^{(\iota)}$ restricts to the trivial
character on $(1+\p^n)^\times$,
\[\int_{O_\p^\times}\iota f(a)\cdot\chi_{\fP^c}^{(\iota)}(a)\rrd
a=\iota\int_{O_\p^\times}f(a)\rrd a.\] We may put
\[\sQ'(f)=\int_{O_\p^\times}f(a)\rrd a\in K\] which is a constant (depending only on
$f$).
\end{proof}

\begin{proof}[Proof of Proposition \ref{pr:matrix_integral}]
Let $f_\pm\in(\Pi^\pm)^\heartsuit_K$ be $n$-admissible stable vectors. It is
clear that $\sQ(f_{+v},f_{-v})$ constructed in Lemma \ref{le:matrix_tame} is
$1$ for almost all $v$. Therefore, we may simply define $\sQ(f_+,f_-)$ to be
the image of
\[\sQ(f_{+\p},f_{-\p})\otimes\bigotimes_{v\neq\p}\sQ(f_{+v},f_{-v})\]
in $\sD(\omega_A,K)$.
\end{proof}

\subsection{Universal torus periods}
\label{ss:universal_torus}

Let $\dB$ be as in the previous section and we choose a CM point $P^+\in
Y^+(E^\ab)$. Recall that for $m\in\dN\cup\{\infty\}$, we have the closed
formal subscheme $\fY^\pm(m)$ of $\fX(m)$ as in
\Sec\ref{ss:comparison_differential}. For a complete field extension
$K/F_\p^\nr$, put
\[\sN^\pm(m,K)=\rH^0(\fY^\pm(m),\sO_{\fY^\pm(m)})\wtimes_{O_\p^\nr}K.\]

The point $P^\pm$ identifies $\sN^\pm(\infty,K)$ with
$\Map_{\r{cont}}(E^\times_\cl\backslash\dA^{\infty\times},K)$, the
$K$-algebra of continuous $K$-valued functions on
$E^\times_\cl\backslash\dA^{\infty\times}$, under which for every
$f\in\sN^\pm(\infty,K)$ and $t\in E^\times_\cl\backslash\dA^{\infty\times}$,
$f(\rT_t P^\pm)=f(t^{\pm1})$. Moreover, the Galois group $O_{E_\p}^\times$ of
$\fX(\infty)/\fX(0)$ preserves $\fY^\pm(\infty)$, whose induced action on
$\sN^\pm(\infty,K)$ is given by
\[t^*f(x)=f(xt^{\pm1}),\quad x\in E^\times_\cl\backslash\dA^{\infty\times}.\]
The quotient of $\fY^\pm(\infty)$ by the subgroup $O_{E_\p,m}^\times$ is
simply $\fY^\pm(m)$.

\begin{definition}\label{de:distribution_algebra}
Consider a locally constant character
\[\omega\colon
F^\times\backslash\dA^{\infty\times}\to M^\times.\] Let $K/MF_\p$ be a
complete field extension. For every $V^\p\in\Gamma_E$ on which $\omega$ is
trivial, denote by $\sD(\omega,K,V^\p)$ the quotient of
$D(E^\times_\cl\backslash\dA^{\infty\times}/V^\p,K)$ by the closed ideal
generated by $\{\omega(t)[t]-1\res t\in\dA^{\infty\times}\}$. Then we have
the canonical isomorphism
\[\sD(\omega,K)\simeq\varprojlim_{\Gamma_E}\sD(\omega,K,V^\p),\]
where the former one is introduced in Definition \ref{de:character_space}
(5). The (unique) continuous homomorphism $D(O_{E_\p}^\times,K)\to
D(E^\times_\cl\backslash\dA^{\infty\times}/V^\p,K)$ sending $[t]$ to
$\omega_\p(t_\circ)[t]$ for $t=(t_\bullet,t_\circ)\in O_{E_\p}^\times$
descends to a continuous homomorphism $\bw\colon
D(O_\p^\anti,K)\to\sD(\omega,K,V^\p)$ of $K$-algebras, which is compatible
when varying $K$ and $V^\p$. In other words, we have a homomorphism
\begin{align}\label{eq:weight_map}
\bw\colon D(O_\p^\anti,K)\to\sD(\omega,K).
\end{align}
\end{definition}

\begin{definition}[Universal character]
We define the \emph{$\pm$-universal character} to be
\[\chi^\pm_\un\colon E^\times_\cl\backslash\dA^{\infty\times}\xrightarrow{[\;]^{\pm
1}}\sD(\omega,MF_\p)^\times.\]
Then $\chi^\pm_\un$ is an element in
$\sN^\pm(\infty,F_\p^\nr)\wtimes_{F_\p}\sD(\omega,MF_\p)$ satisfying
\begin{align}\label{eq:character_family}
t^*\chi^\pm_\un=[t]\cdot \chi^\pm_\un
\end{align}
for $t\in O_{E_\p,m}^\times$, the group on which $\omega_\p$ is
trivial.\index{universal character, $\chi^\pm_\un$}
\end{definition}

\begin{definition}[Universal torus period]\label{de:character_universal}
Suppose $K$ contains $MF_\p^\lt$. Given a stable convergent modular form
$f\in\sM_\flat^w(m,K)^\heartsuit$ for some $w,m\in\dN$, we have the global
Mellin transform $\bM(f)$ by Theorem \ref{th:family}. By restriction, we
obtain elements
\[\bM(f)\res_{\fY^\pm(\infty)}\in\sN^\pm(\infty,K)\wtimes_{F_\p}D(O_\p^\anti,F_\p),\]
and hence by \eqref{eq:weight_map},
\[\bw(\bM(f)\res_{\fY^\pm(\infty)})\in\sN^\pm(\infty,K)\wtimes_K\sD(\omega,K).\]
By Theorem \ref{th:family} (2) and \eqref{eq:character_family}, the product
$\bw(\bM(f)\res_{\fY^\pm(\infty)})\cdot\chi^\pm_\un$ descends to an element
in $\sN^\pm(m,K)\wtimes_K\sD(\omega,K)$ for some (different) $m\in\dN$
depending on $f$ (and $\omega_\p$). For any $V^\p\in\Gamma$ under which $f$
is invariant, we regard $\bw(\bM(f)\res_{\fY^\pm(\infty)})\cdot\chi^\pm_\un$
as an element in $\sN^\pm(m,K)\wtimes_K\sD(\omega,K,V^\p)$, which is then
invariant under $V^\p$. In particular,
\begin{align}\label{eq:universal_period}
\sP^\pm_\omega(f)\colonequals\frac{1}{|E^\times\backslash\dA^{\infty\times}_E/V^\p O_{E_\p,m}^\times|}
\sum_{E^\times\backslash\dA^{\infty\times}_E/V^\p O_{E_\p,m}^\times}
(\bw(\bM(f)\res_{\fY^\pm(\infty)})\cdot\chi^\pm_\un)(t)
\end{align}
is an element in $\sD(\omega,K,V^\p)$ independent of $m$, and is compatible
with respect to $V^\p$. In other words, $\sP^\pm_\omega(f)$ is an element in
$\sD(\omega,K)$, which we call the \emph{universal torus
period}.\index{universal torus period, $\sP^\pm_\omega$}
\end{definition}

\subsection{Interpolation of universal torus periods}
\label{ss:interpolation_universal}

We remain the setting in the previous section. Let $MF_\p^\lt F_\p^\ab\subset
K\subset\dC_p$ be a complete intermediate field.

By Definition \ref{de:heartsuit_representation}, an element
$f^\pm\in\Pi^\pm\otimes_FK$ can be realized as $K$-linear combination of
morphisms from $X_{U^\p U_{\p,\pm m}}$ to $A$ for some $U^\p\in\Gamma_E$ and
$m\in\dN$, which we will now assume. Take
$\omega_\pm\in\rH^0(A^\pm,\Omega^1_{A^\pm})$. Using the notation in
\eqref{eq:restriction_ordinary} and by Proposition \ref{pr:lt_action} (3), we
have $(f_\pm^*\omega_\pm)_\ord\in\sM^0_\flat(\infty,K)$. It is stable (resp.\
$n$-admissible (in the sense of Definition \ref{de:heartsuit_formal})) if and
only if $f^\pm$ is stable (resp.\ $n$-admissible (in the sense of Definition
\ref{de:heartsuit_representation})).

For stable vectors $f^\pm\in(\Pi^\pm)^\heartsuit_K$, define the element
\[\sP^\pm_\un(f_\pm)\in\Lie A^\pm\otimes_{F^M}\sD(\omega_A,K)\]
by the formula
\[\langle\omega_\pm,\sP^\pm_\un(f_\pm)\rangle=\sP^\pm_{\omega_A}((f_\pm^*\omega_\pm)_\ord).\]

In this section, we study the relation of
\[\iota\sP^+_\un(f_+)(\chi)\cdot\iota\sP^-_\un(f_-)(\chi)\in(\Lie A^+\otimes_{F^M}\Lie A^-)\otimes_{F^M,\iota}\dC\]
for a given $\iota\colon\dC_p\xrightarrow{\sim}\dC$, with classical torus
periods, for $n$-admissible stable vectors $f^\pm\in(\Pi^\pm)^\heartsuit_K$
and a character $\chi\in\Xi(\omega_A,K)^n_k$ of weight $k\geq 1$ and depth
$n$ (see Definition \ref{de:character_ramification}). For this purpose, we
choose an $\iota$-nearby data for $\dB$ (Definition \ref{de:nearby_data}). In
particular, we have
\[Y^\pm_\iota(\dC)=E^\times_\cl\backslash\{\pm i\}\times\dA^{\infty\times}_E\subset X_\iota(\dC).\]
Choose an element $t_\pm\in\dA^{\infty\times}_E$ such that $\iota P^\pm$ is
represented by $[\pm i,t_\pm]$. Define $\zeta_\iota^\pm\in\dC^\times$ such
that
\begin{align}\label{eq:ratio_zeta}
\rd z([\pm i,t_\pm])=\zeta_\iota^\pm\cdot\iota\omega_{\psi\pm}\res_{P^\pm},
\end{align}
where $\omega_{\psi\pm}$ is defined in \eqref{eq:omega_lt}. We also introduce
matrices
\[j_\iota^+=\left(
               \begin{array}{cc}
                 1 &  \\
                  & 1 \\
               \end{array}
             \right),\quad
j_\iota^-=\left(
            \begin{array}{cc}
               & 1 \\
              1 &  \\
            \end{array}
          \right)
\] in $\Mat_2(\dR)=B(\iota)\otimes_{F,\iota}R$.

\begin{lem}\label{le:interpolation}
Let the notation be as above. We have
\begin{multline*}
\iota\langle\omega_+,\sP^+_\un(f_+)(\chi)\rangle\cdot\iota\langle\omega_-,\sP^-_\un(f_-)(\chi)\rangle
=\frac{(\zeta_\iota^+\zeta_\iota^-)^k\cdot\chi^{(\iota)}(t_+^{-1}t_-)}{4}\\
\times\sP_{\r{Wa}}(\Delta_{+,\iota}^{k-1}(\rR(j_\iota^+)\phi_\iota(f_+^*\omega_+)),\chi^{(\iota)+1})
\sP_{\r{Wa}}(\Delta_{-,\iota}^{k-1}(\rR(j_\iota^-)\phi_\iota(f_-^*\omega_-)),\chi^{(\iota)-1}),
\end{multline*}
where $\phi_\iota$ is defined in \eqref{eq:automorphic_form}, and
\begin{align*}
\sP_{\r{Wa}}(\Phi,\chi^{(\iota)\pm1})
=\int_{E^\times\dA^\times\backslash\dA^\times_E}\Phi(t)\chi^{(\iota)}(t)^{\pm 1}\rrd t
\end{align*}
is the classical torus period, that is, the automorphic period appearing in
the classical Waldspurger formula.
\end{lem}

\begin{proof}
Take $V^\p\in\Gamma_E$ under which $f_\pm$ and $\chi$ are invariant. By
Theorem \ref{th:family} and Definition \eqref{eq:universal_period}, we have
\begin{multline}\label{eq:period_ordinary}
\langle\omega_\pm,\sP^\pm_\un(f_\pm)(\chi)\rangle=\frac{\Omega_\lt^{-k}}{|E^\times\backslash
\dA^{\infty\times}_E/V^\p O_{E_\p,m}^\times|}\\
\times\sum_{E^\times\backslash\dA^{\infty\times}_E/V^\p O_{E_\p,m}^\times}
\Theta_\ord^{k-1}(f_\pm^*\omega_\pm)_\ord(t)\cdot(\chi^{\pm1}\omega_{\psi\pm}^{-k})(t)
\end{multline}
for some sufficiently large $m\geq n$. By \eqref{eq:ratio_zeta} and Lemma
\ref{le:theta_comparison}, we have
\begin{multline*}
\iota\langle\omega_\pm,\sP^\pm_\un(f_\pm)(\chi)\rangle=\frac{(\zeta_\iota^\pm)^k}{|E^\times\backslash
\dA^{\infty\times}_E/V^\p O_{E_\p,m}^\times|}\\
\times\sum_{E^\times\backslash\dA^{\infty\times}_E/V^\p O_{E_\p,m}^\times}
\Theta_\iota^{k-1}f_\pm^*\omega_\pm(t)\rrd z([\pm i,t_\pm
t])^{-k}\cdot\chi^{(\iota)}(t)^{\pm1}\\
=\frac{(\zeta_\iota^\pm)^k}{2}\int_{E^\times\dA^\times\backslash\dA^\times_E}
\rR(j_\iota^\pm)\phi_\iota(\Theta_\iota^{k-1}f_\pm^*\omega_\pm)(t)\cdot\chi^{(\iota)}(t)^{\pm1}\rrd
t,
\end{multline*}
which by Lemma \ref{le:shimura_maass} equals
\begin{align*}
=\frac{(\zeta_\iota^\pm)^k\cdot\chi^{(\iota)}(t_\pm^{\mp1})}{2}\int_{E^\times\dA^\times\backslash\dA^\times_E}
\Delta_{\pm,\iota}^{k-1}(\rR(j_\iota^\pm)\phi_\iota(f_\pm^*\omega_\pm))(t)\cdot
\chi^{(\iota)}(t)^{\pm1}\rrd t.
\end{align*}
This completes the proof.
\end{proof}

\begin{proposition}\label{pr:interpolation}
Given $\omega_\pm\in\rH^0(A^\pm,\Omega^1_{A^\pm})$, $n$-admissible stable
vectors $f_\pm\in(\Pi^\pm)^\heartsuit_K$, and a character
$\chi\in\Xi(\omega,K)^n_k$ of weight $k\geq 1$ and depth $n$, we have
\begin{multline*}
\iota\langle\omega_+,\sP^+_\un(f_+)(\chi)\rangle\cdot\iota\langle\omega_-,\sP^-_\un(f_-)(\chi)\rangle
=\iota\sQ(f_+,f_-)(\chi)\\
\times L(1/2,\rho_A^{(\iota)},\chi^{(\iota)})
\cdot\frac{2^{g-3}\delta_E^{1/2}\zeta_F(2)\bP_\iota(\chi)}{L(1,\eta)^2L(1,\rho_A^{(\iota)},\Ad)}
\frac{\iota\epsilon(1/2,\psi,\rho_{A,\p}\otimes\check\chi_{{\fP}^c})} {\iota
L(1/2,\rho_{A,\p}\otimes\check\chi_{{\fP}^c})^2}.
\end{multline*}
\end{proposition}

\begin{proof}
By the classical Waldspurger formula \cite{Wal85} (or see
\cite{YZZ13}*{Theorem 1.4.2}) and Proposition \ref{pr:matrix_integral}, we
have that
\begin{multline*}
\sP_{\r{Wa}}(\Delta_{+,\iota}^{k-1}(\rR(j_\iota^+)\phi_\iota(f_+^*\omega_+),\chi^{(\iota)+1})
\sP_{\r{Wa}}(\Delta_{-,\iota}^{k-1}(\rR(j_\iota^-)\phi_\iota(f_-^*\omega_-),\chi^{(\iota)-1})\\
=C_\iota\frac{1}{4}\frac{\zeta_F(2)L(1/2,\rho_A^{(\iota)},\chi^{(\iota)})}
{L(1,\rho_A^{(\iota)},\Ad)L(1,\eta)}\frac{2}{2^{-g}\delta_E^{-1/2}L(1,\eta)}
\frac{\iota\epsilon(1/2,\psi,\rho_{A,\p}\otimes\check\chi_{{\fP}^c})} {\iota
L(1/2,\rho_{A,\p}\otimes\check\chi_{{\fP}^c})^2}\cdot\iota\sQ(f_+,f_-)(\chi),
\end{multline*}
where $C_\iota$ is the complex constant such that
\[(\Delta_{+,\iota}^{k-1}(\rR(j_\iota^+)\phi_\iota(f_+^*\omega_+),\Delta_{-,\iota}^{k-1}(\rR(j_\iota^-)\phi_\iota(f_-^*\omega_-))_{\r{Pet}}
=C_\iota\cdot\iota(f_+,f_-)_A\] holds for all $f_+$ and $f_-$. The
proposition follows from the above formula, Lemma \ref{le:interpolation}, and
the formula
\[\bP_\iota(\chi)=C_\iota\cdot(\zeta_\iota^+\zeta_\iota^-)^k\cdot\chi^{(\iota)}(t_+^{-1}t_-).\]
\end{proof}

\begin{corollary}\label{co:ratio}
For $\chi\in\Xi(\omega_A,K)^n_k$ with $k\geq 1$, the ratio
\[\frac{\sP^+_\un(f_+)(\chi)\sP^-_\un(f_-)(\chi)}{\sQ(f_+,f_-)(\chi)}
\in(\Lie A^+\otimes_{F^M}\Lie A^-)\otimes_{F^M}K,\] if the denominator is
nonzero, is independent of the choice of $f_\pm\in(\Pi^\pm)^\heartsuit_K$
that is $n$-admissible. Moreover, for
$\iota\colon\dC_p\xrightarrow{\sim}\dC$,
\begin{multline*}
\iota\(\frac{\sP^+_\un(f_+)(\chi)\sP^-_\un(f_-)(\chi)}{\sQ(f_+,f_-)(\chi)}\)=\\
L(1/2,\rho_A^{(\iota)},\chi^{(\iota)})\cdot
\frac{2^{g-3}\delta_E^{1/2}\zeta_F(2)\bP_\iota(\chi)}{L(1,\eta)^2L(1,\rho_A^{(\iota)},\Ad)}
\frac{\iota\epsilon(1/2,\psi,\rho_{A,\p}\otimes\check\chi_{{\fP}^c})}{\iota
L(1/2,\rho_{A,\p}\otimes\check\chi_{{\fP}^c})^2}.
\end{multline*}
\end{corollary}

\begin{proposition}\label{pr:waldspurger_abelian}
For $n$-admissible stable vectors $f_\pm\in(\Pi^\pm)^\heartsuit_K$ and a
character $\chi\in\Xi(\omega_A,K)^n_0$ of weight $0$ and depth $n$, we have
\[\sP^\pm_\un(f_\pm)(\chi)=\log_{A^\pm}P^\pm_\chi(f_\pm).\]
\end{proposition}

\begin{proof}
We may choose a tame level $U^\p\in\Gamma$ that fixes $f_\pm$, and such that
$\chi$ is fixed by $U^\p\cap\dA^{\infty\times}_E$. We may realize $f_\pm$ as
a $K$-linear combination of morphisms from $X_{U^\p U_{\p,\pm m}}$ to $A^\pm$
for some sufficiently large integer $m\geq n$. By linearity, we may assume
$f_\pm$ is such a morphism.

For $\omega_\pm\in\rH^0(A^\pm,\Omega^1_{A^\pm})$, we have by Theorem
\ref{th:family} (3,4) that
\[\rd\bM((f_\pm^*\omega_\pm)_\ord)(\chi\res_{O^\times_{E_{\fP^c}}})
=\Theta_\ord\bM((f_\pm^*\omega_\pm)_\ord)(\chi\res_{O^\times_{E_{\fP^c}}})
=(f_\pm^*\omega_\pm)_\ord,\] which by definition is the restriction of
$f_\pm^*\omega_\pm$ to $\fX(m,U^\p)$. On the other hand, by Proposition
\ref{pr:coleman}, we know that $f_\pm^*\log_{\omega_\pm}$ is a Coleman
integral of $f_\pm^*\omega_\pm$ on (the generic fiber of) $\fX(m,U^\p)$.
Therefore,
\[\bM((f_\pm^*\omega_\pm)_\ord)(\chi\res_{O^\times_{E_{\fP^c}}})=f_\pm^*\log_{\omega_\pm}\]
on $\fX(m,U^\p)$ since both of them are Coleman integrals of
$f_\pm^*\omega_\pm$ on $\fX(m,U^\p)$ that belong to
$\sM_\flat^0(m,K)^\heartsuit$. The proposition follows by
\eqref{eq:heegner_cycle} and \eqref{eq:universal_period}.
\end{proof}

\subsection{Proof of main theorems}
\label{ss:proof_theorems}

Let the situation be as in Definition \ref{de:distribution_algebra}. Denote
by $\sC(\omega,K,V^\p)$ the (closed) subspace of $\sC(\omega,K)$ of functions
that are invariant under the right translation of $V^\p$, which is also a
closed subspace of $C(E^\times\backslash\dA^{\infty\times}_E/V^\p,K)$. By
duality, the strong dual of $\sC(\omega,K,V^\p)$ is canonically isomorphic to
$\sD(\omega,K,V^\p)$.

We consider totally definite (not necessarily incoherent) quaternion algebras
$\dB$ over $\dA$ such that for a finite place $v$ of $F$, $\epsilon(\dB_v)=1$
if $v$ is split in $E$ \emph{or} the Galois representation $\rho_{A,v}$
corresponds to a principal series.

For such a $\dB$, we choose an embedding as \eqref{eq:e_embedding}, which is
possible. We may define a representation
\[\Pi(\dB)_{A^\pm}^\tame={\bigotimes}'_M\Pi_{v,A^\pm},\]
where the restricted tensor product (over $M$) is taken over all finite
places $v\neq\p$, and $\Pi_{v,A^\pm}$ is an $M$-representation of
$\dB_v^\times$ determined by $\rho_{A^\pm,v}$ which is unique up to
isomorphism. In particular, if $\dB$ is incoherent, then
$\Pi(\dB)_{A^\pm}^\tame$ is isomorphic to the away-from-$\p$ component of
$\Pi(\dB)_{A^\pm}$. Let $\sI_\pm(\omega_A,K,V^\p)$ be the closed ideal of
$\sD(\omega_A,K,V^\p)$ generated by
\[\{\sQ(f_+,f_-)\res f_\pm\in(\Pi(\dB)_{A^\pm}^\tame)^{V^\p}\otimes_MK,\epsilon(\dB)=\pm1\},\]
where $\sQ(f_+,f_-)$ is similarly defined (as the product) in Lemma
\ref{le:matrix_tame}. It is topologically finitely generated. Let
$\sC_\pm(\omega_A,K,V^\p)$ be the subspace of $\sC(\omega_A,K,V^\p)$ of
functions lying in the kernel of every element in $\sI_\mp(\omega_A,K,V^\p)$,
which is closed. Put $\Xi(A,K,V^\p)=\Xi(A,K)\cap\sC(\omega_A,K,V^\p)$ and
$\Xi(\omega_A,K,V^\p)=\Xi(\omega_A,K)\cap\sC(\omega_A,K,V^\p)$, where
$\Xi(A,K)$ and $\Xi(\omega_A,K)$ are introduced in Definition
\ref{de:character_space}.

\begin{lem}\label{le:dichotomy}
Assume $V^\p\in\Gamma_E$ is sufficiently small. We have
\begin{enumerate}
  \item $\sI_+(\omega_A,K,V^\p)\cap\sI_-(\omega_A,K,V^\p)=0$;

  \item
      $\sI_+(\omega_A,K,V^\p)+\sI_-(\omega_A,K,V^\p)=\sD(\omega_A,K,V^\p)$;

  \item
      $\sC(\omega_A,K,V^\p)=\sC_+(\omega_A,K,V^\p)\oplus\sC_-(\omega_A,K,V^\p)$;

  \item the subset $\Xi(A,K,V^\p)$ is contained in and generates a dense
      subspace of $\sC_-(\omega_A,K,V^\p)$;

  \item $\sI_+(\omega_A,K,V^\p)$ is the closed ideal generated by
      elements that vanish on $\Xi(A,K,V^\p)$.
\end{enumerate}
\end{lem}

If we put $\sD(A,K,V^\p)=\sD(\omega_A,K,V^\p)/\sI_+(\omega_A,K,V^\p)$, then
\[\sD(A,K)=\varprojlim_{V^\p\in\Gamma_E}\sD(A,K,V^\p).\]
We have $\sD(A,K,V^\p)\wtimes_KK'\simeq\sD(A,K',V^\p)$ and
$\sD(A,K)\wtimes_KK'\simeq\sD(A,K')$ for a complete field extension $K'/K$.

\begin{proof}
We first realize that $\Xi(\omega_A,K,V^\p)$ generates a dense subspace of
$\sC(\omega_A,K,V^\p)$. Thus (1) follows from the dichotomy theorem of
Saito--Tunnell \cites{Tun83,Sai93}. For (2), assume the converse and suppose
that $\sI_+(\omega_A,K,V^\p)+\sI_-(\omega_A,K,V^\p)$ is contained in a
(closed) maximal ideal $\fm$ with the residue field $K'$. Then all local
period distributions $\sQ(f_+,f_-)$ will vanish on the character
\[E^\times\backslash\dA^{\infty\times}_E/V^\p\xrightarrow{[\;]}\sD(\omega_A,K,V^\p)\to K',\]
which contradicts the theorem of Saito--Tunnell. Part (3) is a direct
consequence of (1) and (2). It is clear that $\Xi(A,K,V^\p)$ is contained in
$\sC_-(\omega_A,K,V^\p)$ and by Saito--Tunnell,
$\Xi(\omega_A,K,V^\p)\setminus\Xi(A,K,V^\p)\subset\sC_+(\omega_A,K,V^\p)$,
which together imply (4). Finally, (5) follows from (4).
\end{proof}

\begin{remark}\label{re:nuclear}
In fact, for sufficiently small $V^\p\in\Gamma_E$, the morphism $\bw\colon
D(O_\p^\anti,K)\to\sD(\omega_A,K,V^\p)$ \eqref{eq:weight_map} is injective
with the quotient that is a finite \'{e}tale $K$-algebra. We also have
$D(O_\p^\anti,K)\cap\sI_+(\omega_A,K,V^\p)=\{0\}$. Thus if $K$ is discretely
valued, $\sD(A,K,V^\p)$ is a (commutative) nuclear Fr\'{e}chet--Stein
$K$-algebras (defined for example in \cite{Eme}*{Definition 1.2.10}).
Moreover, it is not hard to see that the transition homomorphism
$\sD(A,K,V'^\p)\to \sD(A,K,V^\p)$ is finite \'{e}tale for $V'^\p\subset
V^\p$. The rigid analytic $MF_\p$-variety associated to $\sD(A,MF_\p,V^\p)$
is a smooth curve, which may be regarded as an eigencurve for the group
$\rU(1)_{E/F}$ of tame level $V^\p$, twisted by (the cyclotomic character)
$\omega_A$ and cut off by the condition that $\epsilon(1/2,\rho_A,\;)=-1$.
\end{remark}

\begin{definition}[$p$-adic $L$-function]\label{de:l_function}
Note that the union $\bigcup_{k\geq1}\Xi(\omega_A,K)^0_k$ is already dense in
$\sC(\omega_A,K)$. Now let $MF_\p^\lt F_\p^\ab\subset K\subset\dC_p$ be a
complete intermediate field. By Corollary \ref{co:ratio}, the ratios
\[\frac{\sP^+_\un(f_+)\sP^-_\un(f_-)}{\sQ(f_+,f_-)}\] for $f_\pm$ running over
$(\Pi(\dB)_{A^\pm})^\heartsuit_K$ with $\epsilon(\dB)=-1$, together define an
element
\[\sL(A)\in(\Lie A^+\otimes_{F^M}\Lie A^-)\otimes_{F^M}\sD(A,K)\] by Lemma
\ref{le:dichotomy}, which is the \emph{anti-cyclotomic $p$-adic $L$-function}
attached to $A$. It actually belongs to $(\Lie A^+\otimes_{F^M}\Lie
A^-)\otimes_{F^M}\sD(A,MF_\p^\lt)$ by the lemma below.
\end{definition}

\begin{lem}
The element $\sL(A)$ belongs to $(\Lie A^+\otimes_{F^M}\Lie
A^-)\otimes_{F^M}\sD(A,MF_\p^\lt)$.
\end{lem}

\begin{proof}
In the definition of $\sL(A)$, we only need to consider
$f_\pm\in(\Pi(\dB)_{A^\pm})^\heartsuit_{MF_\p^\lt}$ such that $f_+$ (resp.\
$\Pi(\dB)_{A^-}(\rJ)f_-$) is invariant under $O_{\fP^c}^\times$. Then for
$\chi\in\bigcup_{k\geq1}\Xi(A,MF_\p^\lt)^0_k$, the value $\sL(A)(\chi)$
belongs to $(\Lie A^+\otimes_{F^M}\Lie A^-)\otimes_{F^M}MF_\p^\lt$ by the
formula \eqref{eq:period_ordinary}.
\end{proof}

\begin{proof}[Proof of Theorem \ref{th:l_function_abelian}]
We only need to show that the element
\[\sL(A)\in(\Lie A^+\otimes_{F^M}\Lie A^-)\otimes_{F^M}\sD(A,MF_\p^\lt)\] introduced in Definition
\ref{de:l_function} satisfies \eqref{eq:interpolation}, which follows from
Corollary \ref{co:ratio}.
\end{proof}

\begin{proof}[Proof of Theorem \ref{th:waldspurger_abelian}]
It follows from Proposition \ref{pr:waldspurger_abelian} and Proposition
\ref{pr:matrix_integral}.
\end{proof}

\section{Remarks on cases of higher weights}
\label{s:higher_weights}

In this chapter, we discuss how to generalize our formulation of $p$-adic
Waldspurger formula to $p$-adic Maass forms of other weights. We keep the
setup in \Sec\ref{ss:maass_functions} and assume $\p$ is split in $E$. Let
$\iota_1,\dots,\iota_g\colon F\hookrightarrow\dC_p$ be $g$ different
embeddings. For each $\iota_i$, there are two embeddings
$\iota_i^\circ,\iota_i^\bullet\colon E\hookrightarrow\dC_p$ extending it,
such that $\iota_i^\circ$ induces $\fP$ if $\iota_i$ induces $\fp$.

Take a totally definite incoherent quaternion algebra $\dB$ over $\dA$, and
put $\dB_p=\dB\otimes_\dQ\dQ_p$. To simplify notation, we regard $X$ as
defined over $\dC_p$. Moreover, we assume
$\dB^\infty\not\simeq\Mat_2(\dA^\infty)$ when $F=\dQ$ to avoid extra care for
cusps. We also ignore all arguments like doing things on $X_U$ first then
taking limits to $X$.

Let $\ul{w}=(w;w_1,\dots,w_g)$ be $g+1$ integers with the same parity such
that $w_i\geq 2$. The embedding $\iota_i$ gives rise to the two-dimensional
standard representation $\std_i$ of $\dB_p^\times$ with $\dC_p$-coefficient.
Put
\[V_{\ul{w}}=\bigotimes_{i=1}^g\Sym^{w_i-2}(\std_i)
\otimes(\iota_i\Nm_{\dB_p/F_p})^{\frac{w-w_i+2}{2}},\] which is a
$\dC_p$-representation of $\dB_p^\times$ of dimension $\prod_{i=1}^g(w_i-1)$.
Let $U_p$ be a compact open subgroup of $\dB_p^\times$. We choose an
$O_{\dC_p}$-lattice $V_{\ul{w}}^\infty$ of $V_{\ul{w}}$ that is stable under
the restricted action of $U_p$. For $n\geq 1$, let $U_{p,n}$ be the kernel of
the induced $U_p$-representation $V_{\ul{w}}^\infty/p^nV_{\ul{w}}^\infty$,
which is a subgroup of finite index. Then the quotient
\[\(V_{\ul{w}}^\infty/p^nV_{\ul{w}}^\infty\times X_{U_{p,n}}\)/(U_p/U_{p,n})\]
defines a locally constant \'{e}tale $O_{\dC_p}/p^n$-module $V_{\ul{w}}^n$ on
$X_{U_p}$. Put
\[\sV_{\ul{w}}=(\varprojlim_{n\geq 1}V_{\ul{w}}^n)\otimes\dQ,\]
which is a $\dB^{\infty\times}$-equivariant $\dC_p$-local system of rank
$\prod_{i=1}^g(w_i-1)$ over $X_{U_p}$ and hence $X$ by restriction. It is
easy to see that up to isomorphism, $\sV_{\ul{w}}$ is independent of the
choice of $U_p$ and the lattice $V_{\ul{w}}^\infty$.

\begin{definition}[$p$-adic Maass form]
Denote by $X^\ord$ the ordinary locus in the rigid analytification of $X$. We
choose an exhausting family of basic wide open neighborhoods $\{X^\ord_r\}_r$
of $X^\ord$ parameterized by real numbers $0<r<1$ (see, for example,
\cite{Kas04}*{\Sec 9}).
\begin{enumerate}
  \item The space of \emph{$p$-adic Maass forms of weight $\ul{w}$} is
      defined to be
      \[\sA_{\dC_p}^{\ul{w}}(\dB^\times)=\varinjlim_{r}\Gamma(X^\ord_r,\sV_{\ul{w}}),\]
      where $\Gamma(X^\ord_r,\sV_{\ul{w}})$ denotes the space of locally
      analytic sections of $\sV_{\ul{w}}$ over $X^\ord_r$, and the
      transition maps in the colimit are restriction maps.
\index{$p$-adic Maass form of weight $\ul{w}$,
$\sA_{\dC_p}^{\ul{w}}(\dB^\times)$}\index{$p$-adic Maass form of weight
$\ul{w}$, $\sA_{\dC_p}^{\ul{w}}(\dB^\times)$!classical}

  \item A $p$-adic Maass form is \emph{classical} if it comes from some
      section $\phi\in\Gamma(X^\ord_r,\sV_{\ul{w}})$ that is \emph{the}
      Coleman primitive (\cite{Col94}*{\Sec 10}) of some element in
      $\rH^0(X,\sV_{\ul{w}}\otimes_{\dC_p}\Omega^1_X)$. (Strictly
      speaking, when $\ul{w}=(w;2,\dots,2)$, Coleman primitives of a
      global differential form are unique up to an element in
      $\rH^0(X,\sO_X)$.)
\end{enumerate}
In fact, the space $\sA_{\dC_p}^{\ul{w}}(\dB^\times)$ up to isomorphism, and
the notion of being classical do not depend on the choice of the family of
basic wide open neighborhoods. Moreover, it is not hard to see that Hecke
actions of $\dB^{\infty\times}$ preserve $\sA_{\dC_p}^{\ul{w}}(\dB^\times)$.
An irreducible $\dB^{\infty\times}$-subrepresentation $\pi$ of
$\sA_{\dC_p}^{\ul{w}}(\dB^\times)$ is \emph{classical} if one (and hence all)
of its nonzero members are classical. For a classical $\pi$, we may define
its \emph{dual} $\pi^\vee$, which is contained in
$\sA_{\dC_p}^{\ul{w}^\vee}(\dB^\times)$ where
$\ul{w}^\vee\colonequals(-w;w_1,\dots,w_g)$.
\end{definition}

Let $\ul{l}=(l,l_1,\dots,l_g)$ be $g$ integers with the same parity. The
embedding $\iota_i^\circ$ can be viewed as a $1$-dimensional
$\dC_p$-representation of $E_p^\times$. Put
\[W_{\ul{l}}=\bigotimes_{i=1}^g(\iota_i^\circ)^{l_i}\otimes(\iota_i\circ\Nm_{E_p/F_p})^{\frac{-l-l_i}{2}},\]
which is a $\dC_p$-representation of $E_p^\times$ of dimension $1$. In the
same manner, we have an $\dA^{\infty\times}_E$-equivariant $\dC_p$-local
system $\sW_{\ul{l}}$ of rank $1$ over $Y$. Similar to the space
$\sA_{\dC_p}^{\ul{w}}(\dB^\times)$, we may also define the notion of a
classical irreducible $\dA^{\infty\times}_E$-subrepresentation
$\sigma\subset\Gamma(Y^+,\sW_{\ul{l}})$, and its dual
$\sigma^\vee\subset\Gamma(Y^-,\sW_{\ul{l}^\vee})$. Recall that $Y$ is
contained in $X^\ord$. The following lemma is immediate.

\begin{lem}
Suppose $l=w$ and $|l_i|<w_i$ for $1\leq i\leq g$. Then the
$\dA^{\infty\times}_E$-coinvariant space of
$\Gamma(Y^\pm,\sV_{\ul{w}}\otimes\sW_{\ul{l}})$ is of dimension $1$.
\end{lem}

From now on, we assume $l=w$ and $|l_i|<w_i$ for $1\leq i\leq g$. Denote by
$\Gamma(Y^\pm,\sV_{\ul{w}}\otimes\sW_{\ul{l}})_\flat$ the
$\dA^{\infty\times}_E$-coinvariant (quotient) space of
$\Gamma(Y^\pm,\sV_{\ul{w}}\otimes\sW_{\ul{l}})$, with the quotient map
formally denoted by the integration $\int_{Y^\pm}\rrd y$. For
$\phi\in\sA_{\dC_p}^{\ul{w}}(\dB^\times)$ and
$\varphi_\pm\in\Gamma(Y^\pm,\sW_{\ul{l}})$, consider
\[\sP_{Y^\pm}(\phi,\varphi_\pm)=\int_{Y^\pm}\phi(y)\varphi_\pm(y)\rrd y.\]
Now we adopt the $\pm$ convention for notation as in \Sec\ref{ss:l_function}
and below. In particular, we have classical representations
$\pi^\pm\subset\sA_{\dC_p}^{\ul{w}^\pm}(\dB^\times)$ and
$\sigma^\pm\subset\Gamma(Y^+,\sW_{\ul{l}^\pm})$. We make the following
choices:
\begin{itemize}
  \item an \emph{abstract conjugation}, that is, an
      $\dA^{\infty\times}_E$-equivariant automorphism $\bc$ of $Y$
      switching $Y^+$ and $Y^-$;
  \item a duality pairing
      $(\sV_{\ul{w}^+}\otimes\sW_{\ul{l}^+})\times\bc^*(\sV_{\ul{w}^-}\otimes\sW_{\ul{l}^-})\to\ul{\dC_p}$
      of local systems over $Y^+$;
\end{itemize}
They induce a pairing
\[(\;,\;)_\flat\colon\Gamma(Y^+,\sV_{\ul{w}^+}\otimes\sW_{\ul{l}^+})_\flat\times\Gamma(Y^-,\sV_{\ul{w}^-}\otimes\sW_{\ul{l}^-})_\flat\to\dC_p.\]
The $p$-adic Waldspurger formula in this generality would be a formal seeking
for the value of
$(\sP_{Y^+}(\phi_+,\varphi_+),\sP_{Y^-}(\phi_-,\varphi_-))_\flat$ for
$\phi_\pm\in\pi^\pm$ and $\varphi_\pm\in\sigma^\pm$, in terms of certain
$p$-adic $L$-function.

\appendix

\section{Compatibility of logarithm and Coleman integral}
\label{s:compatibility_logarithm}

In this appendix, we generalize a result of Coleman in \cite{Col85} about the
compatibility of $p$-adic logarithm and Coleman integral. Such result will
only be used in the proof of Proposition \ref{pr:waldspurger_abelian}.

Let $F$ be a local field contained in $\dC_p$ with the ring of integers $O_F$
and the residue field $k$. Let $X$ be a quasi-projective scheme over $F$ and
$U\subset X^\rig$ an affinoid domain with good reduction. We say a closed
rigid analytic $1$-form $\omega$ on $U$ is \emph{Frobenius proper} if there
exits a Frobenius endomorphism $\phi$ of $U$ and a polynomial $P(X)$ over
$\dC_p$ such that $P(\phi^*)\omega$ is the differential of a rigid analytic
function on $U$ and such that no root of $P(X)$ is a root of unity.
Therefore, by \cite{Col85}*{Theorem 2.1}, there exits a locally analytic
function $f_\omega$ on $U(\dC_p)$, unique up to an additive constant on each
geometric connected component, such that
\begin{itemize}
  \item $\rd f_\omega=\omega$;
  \item $P(\phi^*)f_\omega$ is rigid analytic.
\end{itemize}
Such $f_\omega$ is known as a \emph{Coleman integral} of $\omega$ on $U$,
which is independent of the choice of $P$ \cite{Col85}*{Corollary 2.1b}.

\begin{proposition}\label{pr:coleman}
Let $X$ and $U$ be as above. Let $A$ be an abelian variety over $F$ which has
either totally degenerate reduction or potentially good reduction. For a
morphism $f\colon X\to A$ and a differential form $\omega\in\Omega^1(A/F)$,
$f^*\omega\res_U$ is Frobenius proper which admits $f^*\log_\omega\res_U$ as
a Coleman integral, where $\log_\omega\colon A(\dC_p)\to\dC_p$ is the
$p$-adic logarithm associated to $\omega$.
\end{proposition}

\begin{proof}
We may assume $X$ is projective, and replace $F$ by a finite extension.
Therefore, we may assume $A$ has good reduction, or split totally degenerate
reduction, that is, the connected neutral component $\cA^\circ_s$ of the
special fiber $\cA_s$ of the N\'{e}ron model $\cA$ of $A$ is isomorphic to
$\Gm{k}^d$, where $d$ is the dimension of $A$. The first case follows from
\cite{Col85}*{Theorem 2.8, Proposition 2.2}.

Now we consider the second case. Denote by $\cA^\circ_\eta$ the analytic
domain of $A^\rig$ of points whose reduction is in $\cA^\circ_s$. By the
well-known uniformization, we have $A^\rig\simeq(\Gm{F}^\rig)^d/\Lambda$ for
a lattice $\Lambda\subset\Gm{F}^d(F)$. Moreover, $\cA^\circ_\eta$ is
isomorphic to $\Sp F\ang{T_1,\dots,T_d,T_1^{-1},\dots,T_d^{-1}}$, the rigid
analytic multi-torus of multi-radius $1$.

Choose an admissible covering $\sU$ of $X^\rig$ containing $U$, which
determines a formal model $X_\sU$ of $X$ over $O_F$. Since $X$ is projective,
we may assume $X_\sU$ is algebraic. Let $Z$ be the non-smooth locus of
$X_\sU$ over $O_F$. The set of closed points of $X$ whose reduction is not in
$Z$ forms an analytic domain $W$ of $X^\rig$. Since $U$ has good reduction,
we have $U\subset W$. By N\'{e}ron mapping property, the morphism $f$ extends
unique to a morphism $X_\sU-Z\to\cA$, which induces a morphism $f'\colon U\to
A^\rig$. Without loss of generality, we assume $f'(U)$ is contained in
$\cA^\circ_\eta$. By \cite{Col85}*{Proposition 2.2}, we only need to show
that $\omega\res_{\cA^\circ_\eta}$ is Frobenius proper and
$\log_\omega\res_{\cA^\circ_\eta}$ is a Coleman integral of it.

In fact, we have
\[\left\{\left.\omega\res_{\cA^\circ_\eta}\right|\omega\in\Omega^1(A/F)\right\}=
\Span_F\left\{\frac{\rd T_1}{T_1},\cdots,\frac{\rd T_d}{T_d}\right\}.\] By
linearity, we may assume
$\omega^\circ\colonequals\omega\res_{\cA^\circ_\eta}=\rd T_1/T_1$. We choose
the Frobenius endomorphism on $\cA^\circ_\eta$ to be given by
$\phi((T_1,\dots,T_d))=(T_1^q,\dots,T_d^q)$ where $q=|k|$. We have that
$P(\phi^*)\omega^\circ=0$ for $P(X)=X-q$. On the other hand, the $p$-adic
logarithm $\log$ on $\Sp F\ang{T_1,T_1^{-1}}$ is also killed by $P(\phi^*)$.
Therefore, the function $(\log,1,\dots,1)$ on $\Sp
F\ang{T_1,T_1^{-1}}\times\cdots\times\Sp
F\ang{T_d,T_d^{-1}}\simeq\cA^\circ_\eta$ is a Coleman integral of
$\omega^\circ$, which coincides with the restriction of $\log_\omega$.
\end{proof}

\section{Serre--Tate local moduli for $\cO$-divisible groups (d'apr\`{e}s N.~Katz)}
\label{s:serre_tate}

In this appendix, we generalize a classical theorem of Katz \cite{Kat81}
describing the Kodaira--Spencer isomorphism in terms of the Serre--Tate
coordinate for ordinary $p$-divisible groups to ordinary $\cO$-divisible
groups. Only Theorem \ref{th:serre_tate} and Theorem \ref{th:o_main_a} will
be used in the main part of the article. Some notation in this appendix may
be different from those in \Sec\ref{ss:notation_conventions}.

\subsection{$\cO$-divisible groups and Serre--Tate coordinates}
\label{ss:o_divisible}

Let $F$ be a finite field extension of $\dQ_p$ where $p$ is a rational prime.
Denote by $\breve{F}$ the completion of a maximal ramified extension of $F$.
The ring of integers of $F$ (resp.\ $\breve{F}$) is denoted by $\cO$ (resp.\
$\breve\cO$). Let $k$ be the residue field of $\breve\cO$, which is
isomorphic to $\dF_p^\ac$. For a $p$-divisible group $G$ over $\Spec R$, we
denote by $\Omega(G/R)$ the $R$-modules of invariant differentials of $G$
over $R$, which is the dual $R$-module of the tangent space $\Lie(G/R)$ at
the identity.

Let $S$ be an $\breve\cO$-scheme. Recall that an $\cO$-divisible group over
$S$ is a $p$-divisible group $G$ over $S$ with an action by $\cO$ such that
the induced action of $\cO$ on the sheaf $\ul\Lie(G/S)$ coincides with the
natural action as an $\sO_S$-module (hence an $\cO$-module). Denote by
$\BT^\cO_S$ the category of $\cO$-divisible groups over $S$, which is an
abelian category. We omit the superscript $\cO$ if it is $\dZ_p$. The height
$h$ of $G$, as a $p$-divisible group, must be divisible by $[F:\dQ_p]$. We
define the \emph{$\cO$-height} of $G$ to be $[F:\dQ_p]^{-1}h$. An
$\cO$-divisible group $G$ is connected (resp.\ \'{e}tale) if its underlying
$p$-divisible group is. We denote by $\LT$ the Lubin--Tate $\cO$-group over
$\Spec\breve\cO$, which is unique up to isomorphism. We use the same notation
for its base change to $S$.

For an $\cO$-divisible group $G$ over $S$, define its \emph{$\cO$-Cartier
dual} to be $G^D\colonequals\varinjlim_n\ul\Hom_\cO(G[p^n],\LT[p^n])$
\cite{Fal02}. An $\cO$-divisible group $G$ is \emph{ordinary} if $(G^0)^D$ is
\'{e}tale, where $G^0$ is the connected part of $G$. Denote by $\Tp
G=\varprojlim_n G[p^n]$ the Tate module functor. Denote by
$\Nilp_{\breve\cO}$ the category of $\breve\cO$-schemes on which $p$ is
locally nilpotent.

\begin{theorem}[Serre--Tate coordinate]\label{th:serre_tate}
Let $\bG$ be an ordinary $\cO$-divisible group over $k$. Consider the moduli
functor $\fM_\bG$ on $\Nilp_{\breve\cO}$ such that for every
$\breve\cO$-scheme $S$ on which $p$ is locally nilpotent, $\fM_\bG(S)$ is the
set of isomorphism classes of pairs $(G,\varphi)$ where $G$ is an object in
$\BT^\cO_S$ and $\varphi\colon
G\times_S(S\otimes_{\breve\cO}k)\to\bG\times_{\Spec
k}(S\otimes_{\breve\cO}k)$ is an isomorphism. Then $\fM_\bG$ is canonically
pro-represented by the $\breve\cO$-formal scheme
$\Hom_\cO(\Tp\bG(k)\otimes_\cO\Tp\bG^D(k),\LT)$.

In particular, for every artinian local $\breve\cO$-algebra $R$ with the
maximal ideal $\fm_R$ and $G/R$ a deformation of $\bG$, we have a pairing
\[q(G/R;\;,\;)\colon\Tp\bG(k)\otimes_\cO\Tp\bG^D(k)\to\LT(R)=1+\fm_R.\]
It satisfies:
\begin{enumerate}
  \item For every $\alpha\in\Tp\bG(k)$ and $\alpha_D\in\Tp\bG^D(k)$,
      \[q(G/R;\alpha,\alpha_D)=q(G^D/R;\alpha_D,\alpha).\]

  \item Suppose we have another ordinary $\cO$-divisible groups $\bH$
      over $k$, and its deformation $H$ over $R$. Let
      $\b{f}\colon\bG\to\bH$ be a homomorphism and $\b{f}^D$ be its dual.
      Then $\b{f}$ lifts to a (unique) homomorphism $f\colon G\to H$ if
      and only if
      \[q(G/R;\alpha,\b{f}^D\beta_D)=q(H/R;\b{f}\alpha,\beta_D)\]
      for every $\alpha\in\Tp\bG(k)$ and $\beta_D\in\Tp\bH^D(k)$.
\end{enumerate}
\end{theorem}

By abuse of notation, we will use $\fM_\bG$ to denote the formal scheme
$\Hom_\cO(\Tp\bG(k)\otimes_\cO\Tp\bG^D(k),\LT)$. The proof of the theorem
follows exactly in the way of \cite{Kat81}*{Theorem 2.1}.

\begin{proof}
The fact that $\fM_G$ is pro-presentable is well-known. Now we determine the
representing formal scheme.

Since $\bG$ is ordinary, we have a canonical isomorphism
\[\bG\simeq \bG^0\times\Tp\bG(k)\otimes_\cO F/\cO.\] By the
definition of $\cO$-Cartier duality, we have a morphism
\[e_{p^n}\colon\bG[p^n]\times\bG^D[p^n]\to\LT[p^n].\] The restriction of the
first factor to $\bG^0[p^n]$ gives rise an isomorphism
\[\bG^0[p^n]\xrightarrow{\sim}\Hom_\cO(\bG^D[p^n](k),\LT[p^n])\]
of group schemes over $k$ preserving $\cO$-actions. Passing to limit, we
obtain an isomorphism of $\cO$-formal groups over $k$
\[\bG^0\xrightarrow{\sim}\Hom_\cO(\Tp\bG^D(k),\LT),\] which
induces a pairing
\[E_\bG\colon\bG^0\times\Tp\bG^D(k)\to\LT.\]

Let $G/R$ be a deformation of $\bG$, then we have an extension
\begin{align}\label{eq:o_extension}
\xymatrix{ 0
\ar[r] & G^0 \ar[r] & G \ar[r] & \Tp\bG(k)\otimes_\cO F/\cO\ar[r]
& 0}
\end{align}
of $\cO$-divisible groups. We have pairings
\begin{align*}
E_{G,p^n}&\colon G^0[p^n]\times\bG^D[p^n]\to\LT[p^n], \\
E_G&\colon G^0\times\Tp\bG^D(k)\to\LT,
\end{align*}
which lift $e_{p^n}$ and $E_\bG$, respectively.

Similar to the $p$-divisible group case, the extension \eqref{eq:o_extension}
is obtained from the extension
\[\xymatrix{ 0
\ar[r] & \Tp\bG(k) \ar[r] & \Tp\bG(k)\otimes_\cO F \ar[r] &
\Tp\bG(k)\otimes_\cO F/\cO\ar[r] & 0}\]
by pushing out along a unique $\cO$-linear homomorphism
\[\varphi_{G/R}\colon \Tp\bG(k)\to G^0(R).\]
The homomorphism $\varphi_{G/R}$ may be recovered from \eqref{eq:o_extension}
in the way described in \cite{Kat81}*{page 151}. It is the composite
\[\Tp\bG(k)\to \Tp\bG[p^n](k)\xrightarrow{\ang{p^n}}G^0(R)\]
for any $n\geq 1$ such that $\fm_R^{n+1}=0$. Therefore, from $G/R$, we obtain
a pairing
\[q(G/R;\;,\;)=E_G(R)\circ(\varphi_{G/R},\id)\colon\Tp\bG(k)\otimes_\cO\Tp\bG^D(k)\to\LT(R)=1+\fm_R.\]
This shows that the functor $\fM_\bG$ is canonically pro-represented by the
$\breve\cO$-formal scheme $\Hom_\cO(\Tp\bG(k)\otimes_\cO\Tp\bG^D(k),\LT)$.

For (2), if the given homomorphism $\b{f}\colon\bG\to\bH$ can be lifted to
$f\colon G\to H$, then we must have the following commutative diagram
\[\xymatrix{
0 \ar[r] & \Hom_\cO(\Tp\bG^D(k),\LT) \ar[r]\ar[d]_-{\circ\Tp\b{f}^D(k)} & G \ar[r]\ar[d]^-{f}
& \Tp\bG(k)\otimes_\cO F/\cO  \ar[r]\ar[d]^-{\Tp\b{f}(k)\otimes_\cO F/\cO} & 0  \\
0 \ar[r] & \Hom_\cO(\Tp\bH^D(k),\LT) \ar[r] & H \ar[r] & \Tp\bH(k)\otimes_\cO F/\cO  \ar[r] & 0.
}\] Conversely, if we may fill $f$ in the above diagram, then $\b{f}$ lifts.

The existence of the middle arrow is equivalent to that the push-out of the
top extension by the left arrow is isomorphic to the pull-back of the lower
extension by the right arrow. The above mentioned push-out is an element of
\[\Ext_{\BT^\cO_R}(\Tp\bG(k)\otimes_\cO F/\cO,\Hom_\cO(\Tp\bH^D(k),\LT))\]
which is isomorphic to
\[\Hom_\cO(\Tp\bG(k)\otimes_\cO\Tp\bH^D(k),\LT(R))\] by the
bilinear pairing
\[(\alpha,\beta_D)\mapsto q(G/R;\alpha,\b{f}^D\beta_D).\]
Similarly, the above mentioned pull-back is an element in
\[\Hom_\cO(\Tp\bG(k)\otimes_\cO\Tp\bH^D(k),\LT(R))\] defined by the bilinear pairing
\[(\alpha,\beta_D)\mapsto q(H/R;\b{f}\alpha,\beta_D).\]

It remains to prove (1). Choose $n$ such that $\fm_R^{n+1}=0$. Then both
$G^0(R)$ and $(G^D)^0(R)$ are annihilated by $p^n$. Denote by $\alpha(n)$ the
image of $\alpha$ under the canonical projection $\Tp\bG(k)\to
G[\varpi^n](k)$ and similarly for $\alpha_D(n)$. By construction, we have
$\varphi_{G/R}(\alpha)=\ang{p^n}\alpha(n)\in G^0(R)$ and
$\varphi_{G^D/R}(\alpha_D)=\ang{p^n}\alpha_D(n)\in(G^D)^0(R)$. Therefore, we
have
\[q(G/R;\alpha,\alpha_D)=E_{G,p^n}(\ang{p^n}\alpha(n),\alpha_D(n)).\]
Similarly, we have
$q(G^D/R;\alpha_D,\alpha)=E_{G^D,p^n}(\ang{p^n}\alpha_D(n),\alpha(n))$.

The remaining argument is formal and one only needs to replace $\Gmf$ (resp.\
abelian varieties) by $\LT$ (resp.\ $\cO$-divisible groups) in the proof of
\cite{Kat81}*{Theorem 2.1}. In particular, we have the following lemma.
\end{proof}

\begin{lem}
Given any $n\geq 1$, $x\in G^0[p^n](R)$ and $y\in\bG^D[p^n](k)$, there exist
an artinian local ring $R'$ that is finite and flat over $R$, and a point
$Y\in G^D[p^n](R')$ lifting $y$. For every such $R'$ and $Y$, we have the
equality $E_{G,p^n}(x,y)=e_{p^n}(x,Y)$ inside $\LT(R')$.
\end{lem}

\subsection{Main theorem}

We fix an ordinary $\cO$-divisible group $\bG$ over $k$. Denote by $\fR$ the
coordinate ring of $\fM_\bG$, which is a complete $\breve\cO$-algebra. We
have the universal pairing
\[q\colon \Tp\bG(k)\otimes_\cO\Tp\bG^D(k)\to\LT(\fR)\subset\fR^\times.\]
Therefore, we may regard $q(\alpha,\alpha_D)$ as a regular function on
$\fM_\bG$. For each $\cO$-linear form
$\ell\in\Hom_\cO(\Tp\bG(k)\otimes_\cO\Tp\bG^D(k),\cO)$, denote by $D(\ell)$
the translation-invariant continuous derivation of $\fR$ given by
\[\bD(\ell)q(\alpha,\alpha_D)=\ell(\alpha\otimes\alpha_D)\cdot q(\alpha,\alpha_D).\]
By abuse of notation, we also denote by $\bD(\ell)$ the corresponding map
$\Omega_{\fR/\breve\cO}\to\fR$. Denote by $\fG$ the universal $\cO$-divisible
group over $\fM_\bG$.

We choose a normalized logarithm $\log\colon\LT\to\Gaf$ over
$\breve{\cO}\otimes\dQ$, and put $\omega_0=\log^*\rd T$, which is a generator
of the free $\breve\cO$-module $\Omega(\LT/\breve\cO)$ of rank $1$.

Let $R$ be as in Theorem \ref{th:serre_tate} and $G/R$ be a deformation of
$\bG$. We have the canonical isomorphism of $\cO$-modules
\[\lambda_G\colon\Tp\bG^D(k)\xrightarrow{\sim}\Hom_{\BT^\cO_R}(G^0,\LT).\]
Define the $\cO$-linear map $\omega_G\colon\Tp\bG^D(k)\to\Omega(G/R)$ by the
formula
\[\omega_G(\alpha_D)=\lambda_G(\alpha_D)^*\omega_0\in\Omega(G^0/R)=\Omega(G/R).\]
Let $L_G\colon\Hom_\cO(\Tp\bG^D(k),\cO)\to\Lie(G/R)$ be the unique
$\cO$-linear map such that
\[\omega_G(\alpha_D)\cdot L_G(\alpha_D^\vee)=\alpha_D\cdot\alpha_D^\vee\in\cO.\]
In fact, the $R$-linear extensions
\[\omega_G\colon\Tp\bG^D(k)\otimes_\cO R\to\Omega(G/R)\]
and
\[L_G\colon\Hom_\cO(\Tp\bG^D(k),R)\to\Lie(G/R)\] are isomorphisms.
Similarly, we have an isomorphism
\[\lambda_{G^\vee}\colon\Tp\bG(k)=\Tp\bG^{\et}(k)=\Tp G^{\et}(R)
\xrightarrow{\sim}\Hom_{\BT_R}((G^{\et})^\vee,\Gmf),\] which induces an isomorphism
\[\Tp\bG(k)\otimes_{\dZ_p}R\xrightarrow{\sim}\Omega((G^{\et})^\vee/R)\]
by pulling back $\rd T/T$. It further induces an isomorphism
\[\omega_{G^\vee}\colon\Tp\bG(k)\otimes_\cO R=(\Tp\bG(k)\otimes_{\dZ_p}R)_\cO
\xrightarrow{\sim}\Omega((G^{\et})^\vee/R)_\cO.\] Here, the subscript $\cO$
denotes the maximal flat quotient on which $\cO$ acts via the structure map.
By construction, we have the following functoriality.

\begin{lem}\label{le:o_functorial}
Let $\b{f}\colon\bG\to\bH$ be as in Theorem \ref{th:serre_tate} and $f\colon
G\to H$ lifts $\b{f}$. Then
\begin{enumerate}
  \item
      $((f^{\et})^\vee)^*(\omega_{G^\vee}(\alpha))=\omega_{H^\vee}(\b{f}\alpha)$
      for every $\alpha\in\Tp\bG(k)$, where $f^{\et}\colon G^{\et}\to
      H^{\et}$ is the induced homomorphism on the \'{e}tale quotient.

  \item $f_*(L_H(\alpha_D^\vee))=L_G(\alpha_D^\vee\circ\b{f}^D)$ for
      every $\alpha_D^\vee\in\Hom_\cO(\Tp\bG^D(k),\cO)$.
\end{enumerate}
\end{lem}

Denote by $\rD(\bG)$ the (contra-variant) Dieudonn\'{e} crystal of $\bG$. We
have the following exact sequence
\[\xymatrix{
0 \ar[r] & \Omega(\fG^\vee/\fR) \ar[r] & \rD(\bG^\vee)_{\fR} \ar[r] &
\Lie(\fG/\fR) \ar[r] & 0 }\] and the \emph{Gauss--Manin connection}
\[\nabla\colon\rD(\bG^\vee)_\fR\to \rD(\bG^\vee)_\fR\otimes_\fR\Omega_{\fR/\breve\cO}.\]
They together define the following \emph{(universal) Kodaira--Spencer map}
\[\KS\colon \Omega(\fG^\vee/\fR)\to\Lie(\fG/\fR)\otimes_\fR\Omega_{\fR/\breve\cO},\]
which factors through the quotient
$\Omega(\fG^\vee/\fR)\to\Omega(\fG^\vee/\fR)_\cO$. The following lemma is
immediate.

\begin{lem}
The natural map $\Omega(\fG^\vee/\fR)_\cO\to\Omega((\fG^{\et})^\vee/\fR)_\cO$
is an isomorphism.
\end{lem}

In particular, we may regard $\omega_{\fG^\vee}$ as a map from $\Tp\bG(k)$ to
$\Omega(\fG^\vee/\fR)_\cO$. The following result on the compatibility of the
Kodaira--Spencer map and the Serre--Tate coordinate is the main theorem of
this appendix.

\begin{theorem}\label{th:o_main_a}
We have the following equality in $\Omega_{\fR/\breve\cO}$
\[\omega_\fG(\alpha_D)\cdot\KS(\omega_{\fG^\vee}(\alpha))=\rd\log(q(\alpha,\alpha_D))\]
for every $\alpha\in\Tp\bG(k)$ and $\alpha_D\in\Tp\bG^D(k)$.
\end{theorem}

Note that the definition of $\omega_\fG$, but not $\omega_{\fG^\vee}$,
depends on the choice of $\log$, which is compatible with the right-hand
side.

\subsection{Frobenius}

Denote by $\sigma$ the Frobenius automorphism of $\breve\cO$ that fixes every
element in $\cO$. Put $X^\sigma=X\otimes_{\breve\cO,\sigma}\breve\cO$ for
every $\breve\cO$-(formal) scheme $X$, $\Sigma_X\colon X^\sigma\to X$ the
natural projection, and $\Fr_X\colon X\to X^\sigma$ the relative Frobenius
morphism which is $\breve\cO$-linear. We omit the subscript $X$ when it is
$\fM_\bG$.

\begin{lem}\label{le:frobenius_a}
We have
\begin{enumerate}
  \item There is a natural isomorphism
      \[\fM_\bG^\sigma\xrightarrow{\sim}\fM_{\bG^\sigma}\]
      under which the regular function
      $q(\sigma(\alpha),\sigma(\alpha_D))$ is mapped to
      $\Sigma^*q(\alpha,\alpha_D)$.

  \item Under $\Sigma_{(G^{\et})^\vee}\colon
      (G^{\sigma\et})^\vee\simeq((G^{\et})^\vee)^\sigma\to
      (G^{\et})^\vee$, we have
      \[\Sigma_{(G^{\et})^\vee}^*\omega_{G^\vee}(\alpha)=\omega_{(G^\sigma)^\vee}(\sigma\alpha)\]
      for every $\alpha\in\Tp\bG(k)$.

  \item Under $\Fr_G\colon G\to G^\sigma$, we have
      \[\Fr_{G*} L_G(\alpha_D^\vee)=L_{G^\sigma}(\alpha_D^\vee\circ\sigma^{-1})\]
      for every $\alpha_D^\vee\in\Hom_\cO(\Tp\bG^D(k),\cO)$.
\end{enumerate}
\end{lem}

\begin{proof}
The proof is same as \cite{Kat81}*{Lemma 4.1.1 \& 4.1.1.1}.
\end{proof}

From now on, we choose a uniformizer $\varpi$ of $F$, which gives rise to an
isomorphism $\LT^\sigma\simeq\LT$. In particular, we may identify
$(G^D)^\sigma$ and $(G^\sigma)^D$.

For a deformation $G/R$ of $\bG$, we denote by $G'/R$ the quotient of $G$ by
subgroup $G^0[\varpi]$. The induced projection map
\[\cF_G\colon G\to G'\] lifts the relative Frobenius morphism
\[\Fr_\bG\colon\bG\to\bG^\sigma.\]
Define the \emph{Verschiebung} to be
\[\Ver_\bG=(\Fr_{\bG^D})^D\colon\bG^\sigma\simeq\bG^{D\sigma D}\to\bG.\]
Note that the isomorphism depends on $\varpi$.

\begin{lem}\label{le:frobenius_b}
For $\alpha\in\Tp\bG(k)$ and $\alpha_D\in\Tp\bG^D(k)$, we have formulas
\begin{enumerate}
  \item $\Fr_\bG(\alpha)=\sigma\alpha$ and
      $\Ver_\bG(\sigma\alpha)=\varpi\alpha_D$;

  \item $q(G'/R;\sigma\alpha,\sigma\alpha_D)=
      \varpi.q(G/R;\alpha,\alpha_D)$.
\end{enumerate}
\end{lem}

\begin{proof}
The proof is the same as \cite{Kat81}*{Lemma 4.1.2}, with
$\Ver_\bG\circ\Fr_\bG=\varpi$.
\end{proof}

\begin{lem}\label{le:frobenius_c}
For $\alpha\in\Tp\bG(k)$ and $\alpha_D^\vee\in\Hom_\cO(\Tp\bG^D(k),\cO)$, we
have
\begin{enumerate}
  \item
      $((\cF_G^{\et})^\vee)^*\omega_{G^\vee}(\alpha)=\omega_{G'^\vee}(\sigma\alpha)$;

  \item $\cF_{G*}L_G(\alpha_D^\vee)=\varpi
      L_{G'}(\alpha_D^\vee\circ\sigma^{-1})$.
\end{enumerate}
\end{lem}

\begin{proof}
It follows from Lemmas \ref{le:o_functorial} and \ref{le:frobenius_b}.
\end{proof}

If we apply the construction to the universal object $\fG$, we obtain a
formal deformation $\fG'/\fR$ of $G^\sigma$. Its classifying map is the
unique morphism
\begin{align*}
\Phi\colon\fM_\bG\to\fM_{\bG^\sigma}\xrightarrow{\sim}\fM_\bG^\sigma
\end{align*}
such that $\Phi^*\fG^\sigma\simeq\fG'$. Therefore, we may regard $\cF_\fG$ as
a morphism
\[\cF_\fG\colon\fG\to\Phi^*\fG^\sigma\]
of $\cO$-divisible groups over $\fM_\bG$. Taking dual, we have
\[\cF_{\fG}^\vee\colon\Phi^*\fG^{\sigma\vee}\simeq(\Phi^*\fG^\sigma)^\vee\to\fG^\vee.\]

\begin{lem}\label{le:rationality}
We have
\begin{enumerate}
  \item The map $\omega_{\fG^\vee}\colon\Tp\bG(k)\otimes_\cO\fR
      \to\Omega(\fG^\vee/\fR)_\cO$ induces an isomorphism
      \[\Tp\bG(k)\xrightarrow{\sim}\Omega(\fG^\vee/\fR)_\cO^1\colonequals
      \{\omega\in\Omega(\fG^\vee/\fR)_\cO\res(\cF_\fG^\vee)^*\omega=\Phi^*\Sigma_{\fG^\vee}^*\omega\}\]
      of $\cO$-modules.

  \item The map $L_\fG\colon\Hom_\cO(\Tp\bG^D(k),\fR)\to\Lie(\fG/\fR)$
      induces an isomorphism
      \[\Hom_\cO(\Tp\bG^D(k),\cO)\xrightarrow{\sim}\Lie(\fG/\fR)^0\colonequals\{\delta\in\Lie(\fG/\fR)\res
      \cF_{\fG*}\delta=\varpi\Phi^*\Fr_{\fG*}\delta\}\] of $\cO$-modules.
\end{enumerate}
\end{lem}

\begin{proof}
It can be proved by the same way as \cite{Kat81}*{Corollary 4.1.5} by using
Lemmas \ref{le:frobenius_a} and \ref{le:frobenius_c}.
\end{proof}

Consider the following commutative diagram:
\begin{align}\label{eq:diagram_table}
\xymatrix{
0\ar[r]&\Omega(\fG^\vee/\fR)_\cO \ar[r]^-{a}\ar[d]_-{(\cF_\fG^\vee)^*}
& (\rD(\bG^\vee)_\fR)_\cO \ar[r]^-{b}\ar[d]^-{\rD(\Fr_\bG^\vee)} & \Lie(\fG/\fR) \ar[d]^-{\cF_{\fG*}} \ar[r]&0\\
0\ar[r]&\Omega(\fG'^\vee/\fR)_\cO \ar[r]
& (\rD(\bG^{\sigma\vee})_\fR)_\cO  \ar[r]& \Lie(\fG'/\fR) \ar[r]&0 \\
0\ar[r]&\Omega(\fG^\vee/\fR)_\cO \ar[r]\ar[u]^-{\Phi^*\circ\Sigma_{\fG^\vee}^*}
& (\rD(\bG^\vee)_\fR)_\cO \ar[r]\ar[u]_-{\rD(\Sigma_{\bG^\vee})} & \Lie(\fG/\fR) \ar[u]_-{\Phi^*\circ\Fr_{\fG*}}\ar[r]&0.
}
\end{align}

For $k\in\dZ$, we define $\cO$-modules
\begin{align*}
\rD(\bG^\vee)_\fR^k&=\{\xi\in(\rD(\bG^\vee)_\fR)_\cO\res \rD(\Fr_\bG^\vee)\xi=\varpi^{1-k}\rD(\Sigma_{\bG^\vee})\xi\}\\
&=\{\xi\in(\rD(\bG^\vee)_\fR)_\cO\res\rD(\Ver_\bG^\vee)\rD(\Sigma_{\bG^\vee})\xi=\varpi^k\xi\}.
\end{align*}

\begin{lem}
The maps $\omega_{\fG^\vee}$ and $a$ in \eqref{eq:diagram_table} induce an
isomorphism
\[a_1\colon\Tp\bG(k)\xrightarrow{\sim}\rD(\bG^\vee)_\fR^1\] of $\cO$-modules. The maps
$L_\fG$ and $b$ in \eqref{eq:diagram_table} induce an isomorphism
\[b_0\colon\rD(\bG^\vee)_\fR^0\xrightarrow{\sim}\Hom_\cO(\Tp\bG^D(k),\cO)\]
of $\cO$-modules.
\end{lem}

\begin{proof}
For the first part, by a similar argument in \cite{Kat81}*{Lemma 4.2.1}, we
know that $b(\xi)=0$ for $\xi\in\rD(\bG^\vee)_\fR^1$, that is, $\xi$ is in
the image of $a$. The conclusion then follows from Lemma \ref{le:rationality}
(1).

For the second part, it is easy to see that
$\IM(a)\cap\rD(\bG^\vee)_\fG^0=\{0\}$ by choosing an $\cO$-basis of
$\Tp\bG(k)$. Therefore, $b$ restricts to an injective map
$\rD(\bG^\vee)_\fR^0\to\Lie(\fG/\fR)^0$. We only need to show that this map
is also surjective. For every $\delta\in\Lie(\fG/\fR)^0$, choose an element
$\xi_0\in(\rD(\bG^\vee)_\fR)_\cO$. Put
$\xi_{n+1}=\rD(\Ver_\bG^\vee)\rD(\Sigma_{\bG^\vee})\xi_n$ for $n\geq0$. Then
$b(\xi_n)=\delta$ and $\{\xi_n\}$ converge to an element
$\xi\in\rD(\bG^\vee)_\fR^0$.
\end{proof}

\begin{lem}\label{le:nabla_pre}
For every $\ell\in\Hom_\cO(\Tp\bG(k)\otimes_\cO\Tp\bG^D(k),\cO)$, the action
of $D(\ell)$ under the Gauss--Manin connection on $(\rD(G^\vee)_{\fR})_\cO$
satisfies the formula
\[D(\ell)(\nabla(\rD(\Ver_\bG^\vee)\rD(\Sigma_{G^\vee})\xi))
=\varpi\rD(\Ver_\bG^\vee)\rD(\Sigma_{\bG^\vee})(D(\ell)(\nabla\xi))\] for
every $\xi\in(\rD(\bG^\vee)_\fR)_\cO$.
\end{lem}

\begin{proof}
It is proved in the same way as \cite{Kat81}*{Lemma 4.3.3}.
\end{proof}

\begin{lem}\label{le:nabla}
If $\xi\in(\rD(\bG^\vee)_\fR)_\cO$ satisfies
$\rD(\Ver_\bG^\vee)\rD(\Sigma_{\bG^\vee})\xi=\lambda\xi$ for some
$\lambda\in\breve\cO$, then for every
$\ell\in\Hom_\cO(\Tp\bG(k)\otimes_\cO\Tp\bG^D(k),\cO)$, the element
$D(\ell)(\nabla\xi)\in(\rD(\bG^\vee)_\fR)_\cO$ satisfies
\[\varpi\rD(\Ver_\bG^\vee)\rD(\Sigma_{\bG^\vee})(D(\ell)(\nabla\xi))=\lambda D(\ell)(\nabla\xi).\]
\end{lem}

\begin{proof}
It follows immediately from Lemma \ref{le:nabla_pre}.
\end{proof}

\begin{proposition}
For $\alpha\in\Tp\bG(k)$ and $\alpha_D\in\Tp\bG^D(k)$, there exists a unique
character $Q(\alpha,\alpha_D)$ of $\fM_\bG$ such that
\[\omega_\fG(\alpha_D)\cdot\KS(\omega_{\fG^\vee}(\alpha))=\rd\log Q(\alpha,\alpha_D).\]
\end{proposition}

\begin{proof}
Let $\{\alpha_i\}$ (resp.\ $\{\alpha_{D,j}\}$) be an $\cO$-basis of
$\Tp\bG(k)$ (resp.\ $\Tp\bG^D(k)$). Let $\{\ell_{i,j}\}$ be the basis of
$\Hom_\cO(\Tp\bG(k)\otimes_\cO\Tp\bG^D(k),\cO)$ dual to
$\{\alpha_i\otimes\alpha_{D,j}\}$. Then for every element
$\xi\in(\rD(\bG^\vee)_\fR)_\cO$, we have
\[\nabla\xi=\sum_{i,j}D(\ell_{i,j})(\nabla\xi)\otimes\rd\log q(\alpha_i,\alpha_{D,j}).\]
In particular, for $\xi=\omega_{\fG^\vee}(\alpha)$, we have
\[\nabla\omega_{\fG^\vee}(\alpha)=\sum_{i,j}D(\ell_{i,j})(\nabla\omega_{\fG^\vee}(\alpha))\otimes\rd\log q(\alpha_i,\alpha_{D,j}).\]
By Lemmas \ref{le:rationality} and \ref{le:nabla},
$\nabla\omega_{\fG^\vee}(\alpha)\in\rD(\bG^\vee)_\fR^0$. Therefore, there
exist unique elements $\alpha_{D,i,j}^\vee\in\Hom_\cO(\Tp\bG^D(k),\cO)$ such
that
\[\nabla\omega_{\fG^\vee}(\alpha)=b_0^{-1}(\alpha_{D,i,j}^\vee)\] for every $i,j$.
By definition,
\[\KS(\omega_{\fG^\vee}(\alpha))=\sum_{i,j}L_\fG(\alpha_{D,i,j}^\vee)\otimes\rd\log q(\alpha_i,\alpha_{D,j}),\]
and
\[\omega_\fG(\alpha_D)\cdot\KS(\omega_{\fG^\vee}(\alpha))=\rd\log
\(\prod_{i,j}q(\alpha_i,\alpha_{D,j})^{\alpha_D\cdot\alpha_{D,i,j}^\vee}\).\]
\end{proof}

\begin{corollary}
For elements $\alpha\in\Tp\bG(k)$, $\alpha_D\in\Tp\bG^D(k)$ and
$\ell\in\Hom_\cO(\Tp\bG(k)\otimes_\cO\Tp\bG^D(k),\cO)$, we have that
$D(\ell)(\omega_\fG(\alpha_D)\cdot\KS(\omega_{\fG^\vee}(\alpha)))$ is a
constant in $\cO$.
\end{corollary}

\begin{corollary}
Suppose for every integer $n\geq 1$, we can find a homomorphism
$f_n\colon\fR\to\breve\cO/p^n$ such that
\[f_n(D(\ell)(\omega_\fG(\alpha_D)\cdot\KS(\omega_{\fG^\vee}(\alpha))))=\ell(\alpha\otimes\alpha_D)\]
holds $W_n$. Then $Q=q$ and Theorem \ref{th:o_main_a} follows.
\end{corollary}

The condition of this corollary is fulfilled by Theorem \ref{th:o_main_b}.
Therefore, we have reduced Theorem \ref{th:o_main_a} to Theorem
\ref{th:o_main_b}.

\subsection{Infinitesimal computation}

Let $R$ be an (artinian) local $\breve\cO$-algebra with the maximal ideal
$\fm_R$ satisfying $\fm_R^{n+1}=0$. We suppose $G/R$ is canonical deformation
of $\bG$. Let $\tilde{G}$ be a deformation of $G$ to $\tilde{R}\colonequals
R[\varepsilon]/(\varepsilon^2)$, which gives rise to a map
$\partial\colon\Omega(G^\vee/R)\to\Lie(G/R)$. Note that the target may be
identified with $\Ker(G^0(\tilde{R})\to G^0(R))$.

\begin{lem}
The reduction map $\Tp G(R)\to\Tp\bG(k)$ is an isomorphism.
\end{lem}

\begin{proof}
It follows from the same argument in \cite{Kat81}*{Lemma 6.1}.
\end{proof}

In particular, we may define
$\lambda_{G^\vee}\colon\Tp\bG(k)\to\Hom_{\BT_R}(G^\vee,\Gmf)$ and
\begin{align}\label{eq:o_pairing}
\omega_{G^\vee}\colon \Tp\bG(k)\to\Omega(G^\vee/R).
\end{align}

\begin{theorem}\label{th:o_main_b}
The Serre--Tate coordinate for $\tilde{G}/\tilde{R}$ satisfies
\[q(\tilde{G}/\tilde{R};\alpha,\alpha_D)=1+\varepsilon\omega_G(\alpha_D)
\cdot\partial(\omega_{G^\vee}(\alpha)).\]
\end{theorem}

\begin{lem}
For $\alpha_D\in\Tp\bG^D(k)$ and $\alpha\in\Ker(G^0(\tilde{R})\to
G^0(R))=\Lie(G/R)$, we have
\[E_G(\alpha,\alpha_D)=1+\varepsilon\omega_G(\alpha_D)\alpha.\]
\end{lem}

\begin{proof}
By functoriality, we only need to prove the lemma for the universal object
$\fG/\fR$. By definition,
\[1+\varepsilon\omega_\fG(\alpha_D)\alpha=1+\varepsilon\(\lambda_\fG(\alpha_D)_*\alpha\cdot\omega_0\)\in\LT(\tilde{R}).\]
We also have
\[\lambda_\fG(\alpha_D)_*\alpha\cdot\omega_{\LT}=(\log\circ\lambda_\fG(\alpha_D))_*\alpha\cdot\rd T\]
in $\fR[p^{-1}]$. Therefore, we have the equality
\[E_\fG(\alpha,\alpha_D)=1+\varepsilon\omega_\fG(\alpha_D)\alpha\] in
$\Ker(\LT(\tilde{\fR}[p^{-1}])\to\LT(\fR[p^{-1}]))$.
\end{proof}

For an integer $N>n$, denote by $\alpha_N$ the image of $\alpha$ in
$G[p^N](R)$. Let $\tilde\alpha_N\in\tilde{G}(\tilde{R})$ be an arbitrary
lifting of $\alpha_N$. Then
\[p^N\tilde\alpha_N\in\Ker(\tilde{G}(\tilde{R})\to G(R))
=\Ker(\tilde{G}^0(\tilde{R})\to G^0(R))\simeq\Lie(G/R).\] Such process
defines a map $\varphi_G\colon\Tp G(R)\to\Lie(G/R)$.

\begin{proposition}\label{pr:infinitestimal}
We have $\partial\omega_{G^\vee}(\alpha)=\varphi_G(\alpha)$ for every
$\alpha\in\Tp G(R)$.
\end{proposition}

Assuming the above proposition, we prove Theorem \ref{th:o_main_b}.

\begin{proof}[Proof of Theorem \ref{th:o_main_b}]
It is clear that $G^0\otimes_R\tilde{R}$ is the unique, up to isomorphism,
deformation of $G^0$ to $\tilde{R}$. Then the deformation $\tilde{G}$
corresponds to the extension
\[\xymatrix{
0 \ar[r] &  G^0\otimes_R\tilde{R} \ar[r] & \tilde{G} \ar[r] &
\Tp\bG(k)\otimes_\cO F/\cO \ar[r] & 0. }\] In particular, we may identify
$\tilde{G}^0$ with $G^0\otimes_R\tilde{R}$. We have
\[\Ker(\tilde{G}(\tilde{R})\to G(R))
=\Ker(\tilde{G}^0(\tilde{R})\to G^0(R))=\Ker(G^0(\tilde{R})
\to G^0(R))=\Lie(G^0/R).\] For $D\in\Lie(G^0/R)$, we have
\[E_{\tilde{G}}(D,-)=E_G(D,-)\colon\Tp\bG^D(k)\simeq\Tp(\bG^0)^D(k)\to\LT(\tilde{R}),\]
where in the pairing $E_{\tilde{G}}$ (resp.\ $E_G$), we view $D$ as an
element of $\tilde{G}^0(\tilde{R})$ (resp.\ $G^0(\tilde{R})$). For
$\alpha_D\in\Tp\bG^D(k)$, we have
\[E_G(D,\alpha_D)=1+\varepsilon\omega_G(\omega_\alpha)D\]
by definition. Therefore, Theorem \ref{th:o_main_b} follows from Proposition
\ref{pr:infinitestimal} and the construction of $q$.
\end{proof}

The rest of the appendix is devoted to the proof of Proposition
\ref{pr:infinitestimal}. We will reduce it to certain statements in
\cite{Kat81} about abelian varieties. It is interesting to find a proof
purely using $\cO$-divisible groups.

Recall that ordinary $\cO$-divisible groups over $k$ are classified by its
dimension and $\cO$-height. Let $\bG_{r,s}$ be an $\cO$-divisible group of
dimension $r$ and $\cO$-height $r+s$ with $r\geq 0,r+s>0$.

Choose a totally real number field $E^+$ such that $F\simeq
E^+\otimes_\dQ\dQ_p\simeq$, and an imaginary quadratic field $K$ in which
$p=\fp^+\fp^-$ splits. Put $E=E^+\otimes_\dQ K$. Suppose
$\tau_1,\tau_2,\dots,\tau_h$ are all complex embeddings of $E^+$. Consider
the data $(\bA_{r,s},\theta,i)$ where
\begin{itemize}
  \item $\bA_{r,s}$ is an abelian variety over $k$;
  \item $\theta\colon\bA_{r,s}\to\bA_{r,s}^\vee$ is a prime-to-$p$
      polarization;
  \item $i\colon O_E\to\End_k\bA_{r,s}$ is an $O_E$-action which sends
      the complex conjugation on $O_E$ to the Rosati involution and such
      that, in the induced decomposition
      \[\bA_{r,s}[p^\infty]=\bA_{r,s}[p^\infty]^+\oplus\bA_{r,s}[p^\infty]^-\]
      of the $O_E\otimes\dZ_p$-module $\bA_{r,s}[p^\infty]$,
      $\bA_{r,s}[p^\infty]^+$ is isomorphic to $\bG_{r,s}$ as an
      $\cO$-divisible group.
\end{itemize}
It is clear that the polarization $\theta$ induces an isomorphism
$\bA_{r,s}[p^\infty]^+\xrightarrow{\sim}(\bA_{r,s}[p^\infty]^-)^\vee$. By
Serre--Tate theorem, $\fM_{\bG_{r,s}}$ also parameterizes deformation of the
triple $(\bA_{r,s},\theta,i)$. In what follows, we fix $r,s$ and suppress
them from notation. Let $R$ be as in Theorem \ref{th:serre_tate}, $A/R$ be
the canonical deformation of $\bA/k$, and $\tilde{A}$ be a deformation of $A$
to $\tilde{R}$ such that $\tilde{G}\simeq\tilde{A}[p^\infty]^+$.

There is a similar map \eqref{eq:o_pairing} for $A$ and we have
$\omega_{G^\vee}(\alpha)=\omega_{A^\vee}(\alpha)$ for $\alpha\in\Tp
G(R)\subset\Tp A(R)$, where we view $\Omega(G^\vee/R)$ as a submodule of
$\rH^0(A^\vee,\Omega^1_{A^\vee/R})$. Moreover, the map $\varphi_G\colon\Tp
G(R)\to\Lie(G/R)$ can be extended in a same way to a map $\varphi_A\colon\Tp
A(R)\to\Lie(A/R)$. Then Proposition \ref{pr:infinitestimal} follows from
\cite{Kat81}*{Lemma 5.4 \& \Sec 6.5}, where the argument uses normalized
cocycles and does \emph{not} require $A$ to be ordinary in the usual sense.

\printindex

\end{document}